\documentclass[11pt,onecolumn]{article}

\usepackage{fullpage}
\usepackage{graphicx}
\usepackage{amsmath}
\usepackage{amsthm}
\usepackage{mathrsfs}
\usepackage{amssymb}
\usepackage{enumitem}
\usepackage{verbatim}
\usepackage{authblk}
\usepackage{array}
\usepackage[authoryear,round]{natbib}
\usepackage{setspace}
\usepackage{color}
\usepackage[plainpages=false]{hyperref}
\usepackage{dsfont}
\usepackage{sectsty}
\usepackage{tikz}
\usepackage{subcaption}
\usepackage{pgfplots}
\usepackage[ruled,vlined,linesnumbered]{algorithm2e}
\usepackage{algorithmic}
\usepackage{cleveref}
\usepackage{booktabs}

\usepackage{rotating}
\usepackage{multirow}
\usepackage[left=1in,right=1in,top=1in,bottom=1in,footnotesep=0.5cm]{geometry}
\usepackage{bm}
\usepackage{bbm}

\usepackage{color,soul}
\setstcolor{newcc}

\usepackage{csquotes}


\usepackage{etoolbox}
\let\bbordermatrix\bordermatrix
\patchcmd{\bbordermatrix}{8.75}{4.75}{}{}
\patchcmd{\bbordermatrix}{\left(}{\left[}{}{}
\patchcmd{\bbordermatrix}{\right)}{\right]}{}{}

\setul{0.5ex}{0.3ex}
\setulcolor{blue}
\setstcolor{blue}

\newtheoremstyle{theoremsansserif} 
    {\topsep}                    
    {\topsep}                    
    {\itshape}                   
    {}                           
    {\sffamily\bfseries }        
    {.}                          
    {.5em}                       
    {}  

\theoremstyle{theoremsansserif}
\newtheorem{thrm}{Theorem}[section]
\newtheorem{thm}{Theorem}[section]
\newtheorem{lemma}[thrm]{Lemma}
\newtheorem{lem}[thrm]{Lemma}
\newtheorem{prop}[thrm]{Proposition}
\newtheorem{coro}[thrm]{Corollary}

\theoremstyle{definition}
\newtheorem{exmp}{\sffamily\bfseries Example}[section]

\newtheorem{rmk}{\sffamily\bfseries Remark}[section]
\newtheorem{assumption}{\sffamily\bfseries Assumption}[section]

\newcommand{\E}{{\mathbb{E}}}
\renewcommand{\P}{{\mathbb{P}}}
\newcommand{\RR}{\mathbb{R}}
\newcommand{\II}{\mathbb{I}}

\newcommand{\ba}{\mathbf{a}} 
\newcommand{\bb}{\mathbf{b}} 

\newcommand{\bc}{\mathbf{c}}
\newcommand{\bd}{\mathbf{d}}

\newcommand{\bp}{\mathbf{p}}
\newcommand{\bq}{\mathbf{q}}

\newcommand{\bv}{\mathbf{v}}

\newcommand{\bx}{\mathbf{x}}
\newcommand{\bz}{\mathbf{z}}
\newcommand{\by}{\mathbf{y}}

\newcommand{\blambda}{\boldsymbol{\lambda}}

\newcommand{\bbbeta}{\boldsymbol{\eta}}

\newcommand{\bxi}{\boldsymbol{\xi}}

\newcommand{\cA}{\mathcal{A}}
\newcommand{\cB}{\mathcal{B}}
\newcommand{\cC}{\mathcal{C}}
\newcommand{\cD}{\mathcal{D}}
\newcommand{\cE}{\mathcal{E}}

\newcommand{\cI}{\mathcal{I}}
\newcommand{\cH}{\mathcal{H}}
\newcommand{\cG}{\mathcal{G}}
\newcommand{\cL}{\mathcal{L}}

\newcommand{\cP}{\mathcal{P}}
\newcommand{\cQ}{\mathcal{Q}}

\newcommand{\cS}{\mathcal{S}}

\newcommand{\cX}{\mathcal{X}}

\newcommand{\beqn}{\begin{eqnarray*}}
\newcommand{\eeqn}{\end{eqnarray*}}

\newcommand{\beq}{\begin{eqnarray}}
\newcommand{\eeq}{\end{eqnarray}}

\newcommand{\Al}{\underline{A}}
\newcommand{\Ab}{\bar{A}}
\newcommand{\Au}{{\bar{A}}}

\newcommand{\PP}{{\mathbb{P}}}
\newcommand{\EE}{{\mathbb{E}}}

\definecolor{color1}{rgb}{0,0.5,0}

\allowdisplaybreaks

\allsectionsfont{\sffamily}
\sffamily

\DeclareSymbolFont{extraup}{U}{zavm}{m}{n}
\DeclareMathSymbol{\varheart}{\mathalpha}{extraup}{86}
\DeclareMathSymbol{\vardiamond}{\mathalpha}{extraup}{87}

\definecolor{darkblue}{rgb}{0.05,0.25,0.60}
\definecolor{darkgreen}{rgb}{0,0.6,0}
\definecolor{darkred}{rgb}{0.75,0,0}
\definecolor{dbbpurple}{rgb}{0.5,0,0.9}

\def\NNN{\nonumber}%
\def\argmax{\mathop{\rm arg\,max}}%
\def\argmin{\mathop{\rm arg\,min}}%

\newcommand{\ft}{{\tilde{\blambda}_t}}
\newcommand{\dt}{{\blambda^{\star}_t}}
\newcommand{\dtt}{{\bar{\blambda}^{\star}_t}}

\newcommand{\lama}{{\blambda_1}}
\newcommand{\lamb}{{\blambda_2}}

\pgfplotsset{compat=1.17}
\crefdefaultlabelformat{#2#1#3}

\begin{document}
\setstretch{1.3}
\doublespace


\newcommand{\yc}[1]{{\color{purple}[YC: #1]}}
\newcommand{\ww}[1]{{\color{violet}[WW: #1]}}
\newcommand{\ycedit}[1]{{\color{teal}#1}}

\title{\sf{\fontsize{16.0pt}{18pt}\textbf{Beyond Non-Degeneracy: Revisiting Certainty Equivalent Heuristic for Online Linear Programming}}\\
}
\author[$\dag$]{\normalsize Yilun Chen}
\author[$\ddag$]{\normalsize Wenjia Wang}

\affil[$\dag$]{\sf School of Data Science, the Chinese University of Hong Kong, Shenzhen (CUHK Shenzhen)}
\affil[$\ddag$]{\sf Data Science and Analytics Thrust, Information Hub, the Hong Kong University of Science and Technology (Guangzhou)}
\affil[ ]{\texttt{chenyilun@cuhk.edu.cn, wenjiawang@hkust-gz.edu.cn}}

\date{}

{\singlespace \maketitle}

\begin{abstract}\singlespace
The {\sf Certainty Equivalent} ({\sf CE}) heuristic is a widely-used algorithm for various dynamic resource allocation problems in OR and OM. Despite its popularity, existing theoretical guarantees of {\sf CE} are limited to settings satisfying restrictive fluid regularity conditions, particularly, the non-degeneracy conditions, under the widely held belief that the violation of such conditions leads to performance deterioration and necessitates algorithmic innovation beyond {\sf CE}. 

In this work, we conduct a refined performance analysis of {\sf CE} within the general framework of online linear programming. We show that {\sf CE} achieves uniformly near-optimal regret (up to a polylogarithmic factor in $T$) under only mild assumptions on the underlying distribution, without relying on any fluid regularity conditions. Our result implies that, contrary to prior belief, {\sf CE} effectively beats the curse of degeneracy for a wide range of problem instances with continuous conditional reward distributions, highlighting the distinction of the problem's structure between discrete and non-discrete settings. Our explicit regret bound interpolates between the mild $(\log T)^2$ regime and the worst-case $\sqrt{T}$ regime with a parameter $\beta$ quantifying the minimal rate of probability accumulation of the conditional reward distributions, generalizing prior findings in the multisecretary setting. 

To achieve these results, we develop novel algorithmic analytical techniques. Drawing tools from the empirical processes theory, we establish strong concentration analysis of the solutions to random linear programs, leading to improved regret analysis under significantly relaxed assumptions. These techniques may find potential applications in broader online and offline stochastic decision-making contexts.

\vspace{10mm}
\noindent\it{Key Words}: \rm  Certainty Equivalent Heuristic, Regret Analysis, Degeneracy, Online Linear Programming, Network Revenue Management, Dynamic Resource Allocation, Multisecretary
\end{abstract}

\thispagestyle{empty}
\pagebreak

\setcounter{page}{1}

\setstretch{1.46}

\pagebreak




\section{Introduction}\label{sec:intro}
A variety of important OR and OM problems involve allocating finite, non-replenishable resources to sequentially arriving, random requests, in order to maximize cumulative rewards. Typical examples include the multisecretary problem \citep{kleinberg2005multiple, arlotto2019uniformly,  besbes2024dynamic}, network revenue management (NRM) \citep{gallego1994optimal, talluri2006theory, jasin2012re}, dynamic bidding in repeated auctions \citep{balseiro2015repeated} and order fulfillment \citep{jasin2015lp}, among others. Despite their diverse application background, these problems share common characteristics and can be treated through a unified modeling framework, typically termed \textit{online linear programming} (OLP) as is noticed and explored in a recent line of research \citep{vera2021bayesian,li2022online, bray2024logarithmic}. 

The workhorse algorithm for OLP is the \textsf{certainty equivalent heuristic (CE)}. Initially proposed in the broader context of stochastic control (cf. \cite{bertsekas2012dynamic}), \textsf{CE} relies on the idea of replacing all random variables with their average values and repeatedly solving the resulting static and  deterministic problem (usually referred to as the fluid problem) at each decision epoch to facilitate its online decision-making. This heuristic algorithm comes with computational tractability, offering an advantage over the optimal, yet computationally burdensome dynamic programming (DP) approach. The standard \textsf{CE} heuristic has been extensively researched in various contexts of OLP (or more generally, dynamic optimization with resource constraints), sometimes under alternative names such as the {\sf frequent re-solving heuristic} or the {\sf online greedy} \citep{lueker1998average,jasin2012re,li2022online, bray2024logarithmic, balseiro2023survey}. Additionally, many recently proposed algorithms are variants of the standard {\sf CE}, including the {\sf Bayes Selector} \citep{vera2021bayesian}, the {\sf IRT} \citep{bumpensanti2020re}, the {\sf CwG} \citep{besbes2024dynamic} and the {\sf boundary attracted algorithm} \citep{jiang2022degeneracy}, among others. A central focus in this body of work is to understand the algorithm's performance, typically measured by the regret, which captures the asymptotic loss in reward relative to the optimal DP (or other performance benchmarks such as the hindsight optimum or the fluid optimum), as the system scales.

It is well-known that the standard {\sf CE} achieves impressive low regret guarantees—either independent of $T$ or growing logarithmically in $T$ (where $T$ is the length of the time horizon)—when the OLP instance satisfies certain regularity conditions, in particular, the \textit{non-degeneracy} conditions and the \textit{second-order growth} conditions (see literature review). Roughly speaking, an OLP instance is non-degenerate if the optimal solution to the fluid problem exhibits stability, i.e. the set of binding constraints at optimum remains unchanged under perturbation in the initial inventory of resources  (cf. Assumption \ref{def: dual stability} and Assumption \ref{def: dual non-degeneracy}). And it satisfies second-order growth if the (dual) fluid objective exhibits local curvature near the optimum point, i.e. bounded from below (locally) by a positive quadratic function (cf. Assumption \ref{def: 2-order-growth}). 

Despite their prevalence, these conditions are unsatisfying for several reasons. First, except for highly structured settings e.g. the multisecretary problem, these conditions are typically hard to verify. In fact, they are not imposed directly on the problem's primitives, i.e. the distribution of the requests, but on the solution to the fluid problem. The fluid problem of an OLP instance with multiple resource constraints is a multi-dimensional linear program or even an infinite-dimensional convex program. Verifying non-degeneracy and second-order growth requires solving the fluid problem and check the stability and the local curvature property near its optimal solution, which is highly non-trivial. Second, these conditions are fairly restrictive. In fact, degeneracy can arise in practice where the initial inventory of resources often follow the square-root inventory law \citep{bumpensanti2020re, jiang2022degeneracy} and is on the verge of being binding/non-binding. In such cases, small perturbation of the initial inventory of resources will affect the set of binding constraints at the optimum in the fluid problem, violating non-degeneracy conditions and thus limiting the applicability of performance guarantees that depend on them.

These unsatisfying aspects naturally motivate the following question:
\begin{enumerate}
    \item[\textit{(i)}] \textit{Can {\sf CE} still perform well without imposing the non-degeneracy conditions and the second-order growth conditions?} 
\end{enumerate}
Existing evidence is generally negative. In fact, in the special \textit{discrete} setting (where the underlying requests only have finitely many types), 
both theoretical analysis and numerical evidence indicate that {\sf CE} suffers from performance deterioration for degenerate instances (cf. \citep{bumpensanti2020re}). These negative results have motivated a number of algorithmic innovations—mostly variants of the standard {\sf CE}—aiming for uniform regret guarantee regardless of degeneracy (see literature review).  

In this work, we revisit {\sf CE} in a richer instance space that incorporates non-discrete request distributions and, rather surprisingly, provide a (conditionally) affirmative answer to question \textit{(i)} in a fairly broad sense. As an added bonus, for a number of interesting OLP instances, our analysis yields the \textit{first} known uniform low-regret guarantee—no additional algorithmic innovation is needed, the standard {\sf CE} suffices. Furthermore, through an in-depth discussion, we clarify how the notion of non-degeneracy differs between discrete and non-discrete settings, thereby resolving the apparent paradox between our findings and the previous established literature. Next, we detail our main contributions.

\paragraph{Algorithmic contribution.}
This work makes two-fold algorithmic contributions to the literature. First, we prove that {\sf CE} achieves $\mathcal{O}\left((\log T)^2\right)$ hindsight regret under significantly relaxed conditions (cf. Theorem \ref{thm: regret of CE} and Proposition \ref{thm: fundamental lower bound}), requiring only that the underlying request distribution belongs to one of two broad distribution classes (Assumptions \ref{assum: starting-from-zero} and \ref{assum: small-probability-starting-from-zero}). To the best of our knowledge, this is the first generic low-regret guarantee of {\sf CE} that does not rely on imposing either non-degeneracy or second-order growth conditions. Our result advances the understanding of {\sf CE}'s effectiveness for a rich array of OLP instances, including both well-explored ones in the literature such as the multisecretary problem (cf. Example \ref{exmp:multisec}) and less understood but practically relevant ones such as the generalized linear models (cf. Example \ref{exmp: linear}), among others. The distributional conditions that we impose are both natural and to some extent necessary. 
Intuitively, they can be viewed as the multi-dimensional generalization of the class of ``gap-free'' reward distributions, i.e. distributions supported on intervals, a necessary structure for {\sf CE} to achieve low ($o(\sqrt{T})$) regret for the multisecretary problem (cf. \citep{besbes2024dynamic}), which is a specific subset of OLP with a single resource constraint and constant (unit) resource consumption per request. See Table \ref{tab: CE condition comparison} for a comparison between our results and the existing analysis of {\sf CE} in terms of achievable regret bounds and the fluid regularity conditions required.

We also contribute to the expanding literature on \textit{uniform}-regret algorithm design (beyond {\sf CE}) for OLP.  The distribution classes that we identify contain natural request distributions that are beyond the scope of algorithms and/or algorithmic analysis in  prior works. For OLP instances with these request distributions, we provide the first uniform regret guarantee that overcomes degeneracy. For example, consider a simple request distribution with $2$-dimensional resource consumption vectors uniformly supported on the $\frac{1}{N}$-mesh of $[1,2]^2$, precisely $\left\{(1 + a_i, 1 + a_j), 0\leq i, j \leq N\right\}$ where $a_i = \frac{i}{N}$ for all $i$. The rewards are independent ${\sf Unif}[0, 1]$ conditional on each grid point. This stylized example reflects practical scenarios where greater data availability allows more fine-grained decision contexts, here represented by a larger $N$. However, this setting lies beyond the modeling and analytical capabilities of earlier works \citep{bumpensanti2020re, vera2021bayesian, vera2021online, besbes2024dynamic, jiang2022degeneracy}, as their algorithmic guarantee all crucially rely on the discreteness of the underlying resource consumption distribution, with either the size of its support or the inverse minimum probability mass entering the final regret bounds. By contrast, our first distribution class (cf. Assumption \ref{assum: starting-from-zero}) covers the distribution in this example as $N$ scales. In fact, Assumption \ref{assum: starting-from-zero} allows for \textit{arbitrary} bounded resource consumption distributions—whether discrete, continuous or of a mixed type, and only the upper and lower bounds of the distribution's range enter the final regret bound. With the additional regularity assumption that the resource consumption distribution is ``uniform-like'' continuous, Assumption \ref{assum: small-probability-starting-from-zero} further permits more general conditional reward distributions—supported on intervals not necessarily starting from zero, which captures more sophisticated request structures such as the generalized linear models (cf. Example \ref{exmp: linear}). Our regret guarantees for OLP instances with these request distributions are achieved by the standard {\sf CE} without the need for any algorithmic adjustments, showing the effectiveness of {\sf CE} for a broader range of instances than previously understood. See a precise comparison of the classes of distributions of different uniform-regret algorithms in Table \ref{tab:uniform-regret algorithm comparison}.

Our analysis extends beyond instances with the common $\log T$ regret. Similar to \cite{besbes2024dynamic} in the multisecretary setting, we obtain a spectrum of regret guarantees, with the explicitly scaling of $\mathcal{O}\left(T^{\frac{1}{2} - \frac{1}{2(1 + \beta)}} (\log T)^{\frac{2 + \beta}{2 + 2 \beta} + \mathbbm{1}\{\beta = 0\}} \right)$, that interpolates between the mild $(\log T)^2$ regime and the worst-case $\sqrt{T}$ regime with a parameter $\beta \in [0, \infty)$ depending on the underlying request distribution. In particular, this parameter quantifies the minimal probability mass accumulation of the conditional reward CDFs. Specifically, $\beta = 0$ corresponds to uniform-like conditional reward distributions, yielding the mild regret bound of $\mathcal{O}\left((\log T)^2\right)$. As $\beta \to \infty$, the conditional reward CDFs converge towards distributions with gaps on their supports, and the achievable regret scaling worsens, approaching $\sqrt{T}$. This achievable regret scaling provides a near-optimal guarantee of {\sf CE}: it matches (up to a polylogarithmic factor in $T$) the previously established lower bound $\Omega\left(T^{\frac{1}{2} - \frac{1}{2(1 + \beta)}}\left(\log T\right)^{\mathbbm{1}\{\beta = 0\}}\right)$ on the hindsight regret for the multisecretary problem (cf. \cite{besbes2024dynamic, bray2024logarithmic}),  yet holds in a far more general OLP setting.

\begin{center}
    \begin{table}[h]
    \centering
    \caption{Comparison of fluid regularity conditions and achievable regret for {\sf CE}. 
    Here \cite{li2022online, bray2024logarithmic} consider the scenario where $F$ is unknown and needs to be learned. To achieve fair comparison, we only demonstrate regret bounds under the condition that the (conditional) reward distribution is continuous with bounded density.}
    \begin{tabular}{cccc}
    \toprule
        &  Non-degeneracy & $2^{\rm nd}$-order growth & Regret \\
       \midrule
       \multicolumn{1}{c}{\cite{li2022online}}  & Yes & Yes & $\mathcal{O}(\log T\log\log T)$ \\
       \multicolumn{1}{c}{\cite{bray2024logarithmic}} & Yes & Yes & $\mathcal{O}(\log T)$ \\ 
       \multicolumn{1}{c}{\cite{balseiro2023survey}} & Yes & Yes & $\mathcal{O}(\log T)$ \\
       \multicolumn{1}{c}{\cite{jiang2022degeneracy}} & No & Yes & $\mathcal{O}(\log T)$ \\
       \multicolumn{1}{c}{\textcolor{blue}{This work}} & \textcolor{blue}{No} &  \textcolor{blue}{No} & \textcolor{blue}{$\mathcal{O}\left((\log T)^2\right)$} \\
       \bottomrule
    \end{tabular}
    \label{tab: CE condition comparison}
    \end{table}
\end{center}

\begin{center}
    \begin{table}[h]
    \centering
    \caption{Comparison of request distribution types for different uniform-regret algorithms. Resource consumption $\equiv 1$ indicates a multisecretary  problem.  {\sf RAMS} achieves the performance of any row above (up to simulation errors) by taking the corresponding algorithm as a reference algorithm.}
    \begin{tabular}{ccccc}
    \toprule
        &  Resource & \multirow{2}{*}{Reward}  & \multirow{2}{*}{Lower Bound} & \multirow{2}{*}{Regret} \\
        & consumption & & & \\
        \midrule
       {\sf Budget-Ratio} & \multirow{2}{*}{$\equiv 1$} & \multirow{2}{*}{Discrete} & \multirow{2}{*}{$\Omega(1)$} & \multirow{2}{*}{$\mathcal{O}(1)$} \\
       \cite{arlotto2019uniformly}  & & & & \\
       \hline
       {\sf Bayes Selector}  & \multirow{2}{*}{Discrete} & \multirow{2}{*}{Discrete} & \multirow{2}{*}{$\Omega(1)$} & \multirow{2}{*}{$\mathcal{O}(1)$} \\
       \citep{vera2021bayesian}  & & & & \\
       \hline
       \multicolumn{1}{c}{{\sf IRT}} & \multirow{2}{*}{Discrete} & \multirow{2}{*}{Discrete} & \multirow{2}{*}{$\Omega(1)$} & \multirow{2}{*}{$\mathcal{O}(1)$} \\
       \citep{bumpensanti2020re}  & & & & \\
       \hline
       \multicolumn{1}{c}{{\sf CwG}} & \multirow{2}{*}{$\equiv 1$} & $\beta$-H\"older cont. & \multirow{2}{*}{$\tilde\Omega(T^{\frac{1}{2} - \frac{1}{2(1 + \beta)}})$} & \multirow{2}{*}{${\tilde{\mathcal{O}}}(T^{\frac{1}{2} - \frac{1}{2(1 + \beta)}})$} \\
       \citep{besbes2024dynamic}  & & with gaps & & \\
       \hline
       \multicolumn{1}{c}{{\sf Boundary-attracted}} & \multirow{2}{*}{Discrete} & \multirow{2}{*}{Continuous} & \multirow{2}{*}{$\Omega(\log T)$} & \multirow{2}{*}{$\mathcal{O}\left((\log T)^2)\right)$} \\
       \cite{jiang2022degeneracy}  & & & & \\
       \hline
       \multicolumn{1}{c}{{\sf RAMS}} & \multicolumn{4}{c}{\multirow{2}{*}{Any line above}}\\
       \citep{besbes2024dynamic}  & & & & \\
       \hline
       \multicolumn{1}{c}{\textcolor{blue}{\textsf{CE}}} & \multirow{2}{*}{\textcolor{blue}{Arbitrary}} & \textcolor{blue}{\multirow{2}{*}{$\beta$-H\"older cont.}}  & \multirow{2}{*}{\textcolor{blue}{$\tilde\Omega(T^{\frac{1}{2} - \frac{1}{2(1 + \beta)}})$}} & \multirow{2}{*}{\textcolor{blue}{${\tilde{\mathcal{O}}}(T^{\frac{1}{2} - \frac{1}{2(1 + \beta)}})$}} \\
       \textcolor{blue}{(Ours with Assump. \ref{assum: starting-from-zero})}  & &  & & \\
       \hline
       \multicolumn{1}{c}{\textcolor{blue}{\textsf{CE}}} & \multirow{2}{*}{\textcolor{blue}{Continuous}} & \textcolor{blue}{\multirow{2}{*}{$\beta$-H\"older cont.}} & \multirow{2}{*}{\textcolor{blue}{$\tilde\Omega(T^{\frac{1}{2} - \frac{1}{2(1 + \beta)}})$}} & \multirow{2}{*}{\textcolor{blue}{${\tilde{\mathcal{O}}}(T^{\frac{1}{2} - \frac{1}{2(1 + \beta)}})$}}\\
       \textcolor{blue}{(Ours with Assump. \ref{assum: small-probability-starting-from-zero})}  & & & & \\
       \hline
       \bottomrule
    \end{tabular}
    \label{tab:uniform-regret algorithm comparison}
    \end{table}
\end{center}

\paragraph{Insights on degeneracy.\ }
Compared to existing results, our findings reveal notably different patterns in the performance of {\sf CE}. In particular, {\sf CE} effectively ``beats'' the curse of degeneracy for a broad class of request distributions, achieving uniformly low and near-optimal regret even when standard non-degeneracy conditions are violated (cf. Corollary \ref{coro: non essential non degeneracy}).  This contrasts sharply with earlier observations, e.g. \cite{bumpensanti2020re} in the discrete setting, where degeneracy provably results in $\Theta(\sqrt{T})$ regret for {\sf CE}, regardless of the underlying distribution. 

Motivated by this paradoxical phenomenon, we investigate the notion of degeneracy in different OLP settings, shedding light on how it impacts algorithmic performance (see Section~\ref{sec:implications}). We find that the necessity of the standard non-degeneracy notion for {\sf CE} to achieve low regret is a special property only valid in discrete settings. Concretely, when the fluid problem is an LP (as in discrete settings), standard non-degeneracy conditions are \textit{equivalent} to several other properties (cf. Lemma \ref{lem: non-degeneracy discrete}), including \textit{dual uniqueness}. Any violation of these conditions implies extreme sensitivity of the optimal solution to small perturbations in the inventory of resources, which leads to large regret. However, such a non-degeneracy-dual-uniqueness equivalence breaks down in general. In fact, when the underlying request distribution belongs to the classes that we identify, the corresponding dual fluid problem becomes a non-linear convex program with a smooth objective function. For such a problem, dual uniqueness always holds regardless of the inventory of resources, yet the non-degeneracy conditions may fail (cf. Lemma \ref{lem: nondiscrete non kink}). The failure of non-degeneracy in such settings only indicates instability in the set of binding constraints, but the optimal solution itself always remains stable, a consequence of dual uniqueness, leading to mild regret. Overall, our theoretical results strongly suggest that dual uniqueness, rather than non-degeneracy, is the critical determinant of whether {\sf CE} suffers from performance deterioration (cf. Proposition \ref{prop: necessity of dual-unique in multisec}) in both discrete and non-discrete settings.

We illustrate these insights using simple OLP instances. Let resource consumption $a \sim 0.5 + {\sf Bern}(\frac{1}{2})$ and the initial resource inventory is $b T$. Standard non-degeneracy conditions are violated under any reward $r$ when $b = 1$, since $\EE[a] = 1 = b$, i.e. the resource consumption just barely binds at the optimal solution to the fluid problem. When $r \equiv 1$, this instance becomes a stochastic knapsack problem with finitely many types, and is discrete. At the degeneracy point $b = 1$, the dual uniqueness is also violated, with $[0, \frac{1}{1.5}]$ being the set of dual optimal solutions. Such a violation of non-degeneracy is costly—{\sf CE} will frequently decide to accept requests of $a = 1.5$ due to perturbations in the remaining inventory of resources, and end up mistakenly accepting $\Omega(\sqrt{T})$ more requests of $a = 1.5$ comparing to the hindsight optimal decision, leading to a regret of $\Omega(\sqrt{T})$. One must resort to e.g. the {\sf Bayes Selector} (cf. \cite{vera2021bayesian}) to handle this situation to get a uniform $\mathcal{O}(1)$ hindsight regret. Alternatively, when $r \sim {\sf Unif}[0, 1]$ for both $a = 0.5$ and $a = 1.5$, the violation of non-degeneracy at $b = 1$ does not ``hurt''. In this case, the dual objective function becomes smooth. Consequently, the dual optimal solution is unique (equal to $0$), and {\sf CE} achieves mild $(\log T)^2$ regret by Theorem~\ref{thm: regret of CE} \footnote{One might conjecture that the hindsight regret can be tightened to $\mathcal{O}(\log T)$. Interestingly, we found no existing result in the literature that guarantees this, though it seems straightforward to derive. For our purposes, we leave it at that.}. 

Note that these insights can be revealed only in general OLP settings beyond multisecretary. In particular, the special structure of deterministic resource consumption ($a \equiv 1$) in multisecretary renders the violation of non-degeneracy at $b=1$ trivial. Indeed, if $b=1$, any reasonable algorithm will not incur any regret under any request distribution, be it discrete or non-discrete—simply by accepting all requests. Consequently, the multisecretary problem is inherently immune to the standard notion of ``degeneracy'' (e.g. Assumption 2.2 in \cite{balseiro2023survey}), i.e. \textit{the instability (binding/unbinding) of a resource constraint} at the fluid optimal solution, and the failure of {\sf CE} for multisecretary is solely driven by the request distribution's irregularity. (See Section \ref{sec:implications} for more discussion)

\paragraph{Novel analytical methodology.\ } 
We develop novel regret analysis techniques that bypass the need for non-degeneracy and second-order growth conditions. In prior literature, non-degeneracy conditions are often imposed to ensure the set of binding resource constraints remains identical throughout the algorithm's execution. This further implies that the process of resource inventories induced by the algorithm (after proper transformation) is a martingale, enabling subsequent concentration analysis. By contrast, we decompose the hindsight regret of \textsf{CE} (Lemma \ref{lem: ce-regret-decomp}) and perform a \textit{generic} concentration analysis of the solution to the per-step sample average approximation (SAA) of the fluid dual problem at each time $t$ (Lemma \ref{lem: concentration results of dual optimum}), without the need for martingales.

The concentration analysis is the key technical innovation that allows us to establish improved regret guarantees. Using the peeling device from empirical processes theory, we partition the solution space and apply union bounds with carefully controlled entropy numbers to establish robust worst-case concentration bounds. These uniform concentration bounds hold even when the set of binding constraints changes over time, effectively overcoming the key technical challenge that necessitates either the non-degeneracy conditions or the uniform version of the second order growth conditions in prior work. Furthermore, the parameter $\beta$ of the underlying request distribution directly determines the rate of concentration in our bounds, which in turn determines the scaling of the regret, without the need for imposing additional second-order growth conditions. Together, these techniques enable us to attain regret guarantee under significantly relaxed assumptions made solely on the request distribution. The worst-case analysis leaves us with an unavoidable additional $\log T$ in the regret bounds, which is a mild price to pay to get rid of the fluid regularity conditions.  We believe the concentration analyses are of independent interest and the techniques developed may have broader applicability, particularly since such concentration properties are relevant to many of the online and offline stochastic optimization tasks.

\subsection{Related Literature}

\paragraph{Analysis of {\sf CE}. } In the revenue management literature, the (quantity-based) NRM problem has been formulated as an discrete OLP instance and received considerable attention \citep{talluri2006theory}. Earlier work \citep{gallego1994optimal} established $\mathcal{O}\left(\sqrt{T}\right)$ regret guarantee under static heuristic algorithms. \cite{reiman2008asymptotically} obtained an improved $o(\sqrt{T})$ regret bound by introducing re-solving in their algorithm design. \cite{jasin2012re} analyzed the {\sf frequent re-solving heuristic} (a primal version of {\sf CE}) and established an $\mathcal{O}(1)$ regret under the non-degeneracy condition of the fluid LP.  \cite{bumpensanti2020re}, among others, further confirmed that non-degeneracy is necessary for {\sf CE} to achieve this optimal (fluid) regret. In non-discrete settings, where requests are modeled with continuous distributions, \cite{lueker1998average} initiated the study of OLP with a single resource constraint and proved an $\mathcal{O}(\log T)$ regret bound for {\sf CE} (referred to as the {\sf online greedy}). This bound was shown to be tight in the multisecretary problem by \cite{bray2024logarithmic}. Recent works, including \cite{li2022online}, \cite{bray2024logarithmic}, and \cite{balseiro2023survey} extended the $\mathcal{O}(\log T)$ regret guarantee of {\sf CE} to general multi-resource settings, but their analyses critically rely on restrictive non-degeneracy and second-order growth conditions, which we aim to remove in this work. More generally, similar algorithmic design problems have been studied in the price-based NRM setting \cite{jasin2014reoptimization, wang2022constant}, where the performance of {\sf CE} seems to also depend on the non-degeneracy conditions. We leave the investigation of this setting for future research.

\paragraph{Uniform Loss in OLP. }
To handle the challenges posed by degeneracy on {\sf CE}, a recent stream of papers have proposed new algorithms with $\mathcal{O}(1)$ hindsight regret uniformly across degenerate and non-degenerate instances \citep{arlotto2019uniformly, bumpensanti2020re, vera2021bayesian, vera2021online}. These works are all in the discrete setting, with the request distribution supported on a finite number of different request types. The design and analysis of uniform-regret algorithms in the non-discrete settings, however, is substantially more challenging. The two existing papers that are mostly relevant to our work are \cite{besbes2024dynamic} and \cite{jiang2022degeneracy}. 

\cite{besbes2024dynamic} first focused on the multisecretary problem. They introduced {\sf CwG}—an algorithmic adjustment to the standard {\sf CE}, that achieves near optimal hindsight regret guarantees. Our work differs from theirs in three ways: (i) \textsf{CwG} is both designed and analyzed for the multisecretary setting, whereas we focus on the broader OLP framework; (ii) in \cite{besbes2024dynamic}, the reward distributions may be irregular (with gaps in their supports), which is beyond our regular distribution class (Assumptions \ref{assum: starting-from-zero} or \ref{assum: small-probability-starting-from-zero}); and (iii) we only analyze the standard {\sf CE}, whereas {\sf CE} provably incurs $\Omega(\sqrt{T})$ regret in the setting of \cite{besbes2024dynamic}. Compared to the multisecretary problem, violating standard non-degeneracy conditions in general OLP settings poses significant algorithmic analytical challenges—even when the underlying request distributions are regular (see Section \ref{sec:implications}). A key contribution of our work is to demonstrate that {\sf CE} remains effective in these broader OLP problems despite such violations. Investigating irregular request distributions in this broader setting, where the design of new algorithm is necessary, lies beyond our current scope and is a natural direction for future work. As a second contribution, in the general OLP setting, \cite{besbes2024dynamic} introduced a generic simulation-based algorithm called {\sf RAMS}. The performance analysis of {\sf RAMS} does not directly yield precise regret bounds; instead, it provides a guarantee expressed as the sum of (i) the regret of a reference algorithm and (ii) a simulation error term. In contrast, our analysis provides the first known uniform low-regret guarantee for {\sf CE} in certain OLP instances, a guarantee not recoverable by {\sf RAMS} under its current reliance on existing algorithmic results. However, with our new regret bounds, we expect {\sf RAMS} to match this performance (up to simulation error) by taking {\sf CE} as its reference algorithm. 

\cite{jiang2022degeneracy} studied a ``semi-discrete'' setting with finitely many resource consumption types and uniform-like conditional reward distributions supported on intervals for each resource consumption type. They first proposed the {\sf boundary-attracted algorithm}, an adjusted {\sf CE} heuristic that extends the {\sf CwG} idea, which achieves a uniform $\mathcal{O}\left((\log T)^2\right)$ regret guarantee regardless of degeneracy. Compared to \cite{jiang2022degeneracy}, the setting of our work is more general, allowing fully non-discrete request distributions, with arbitrary resource consumption distributions. Meanwhile, \cite{jiang2022degeneracy} allows for distributions which are effectively irregular and not contained in our distribution classes (see Example \ref{exmp:violating unique-dual}), which fail {\sf CE} and necessitate their algorithmic design. As their second contribution, \cite{jiang2022degeneracy} also conducted a performance analysis of the standard {\sf CE} in their setting. This result not only relies on the semi-discrete structure of the OLP instance, but also requires a rather strong uniform second-order growth condition, for which we provide further discussion in Appendix \ref{appendix: second-order}. By contrast, our regret guarantee of {\sf CE} does not rely on either the the semi-discrete structure or any form of the second-order growth condition. 

Overall, all of these prior results critically rely on the fluid problem being linear, necessitating a (semi-)discrete request distribution. In contrast, we introduce new technical tools that facilitate algorithmic analysis when the fluid problem is non-linear. This leads to the surprising discovery that standard {\sf CE} alone achieves uniform near-optimal regret—effectively overcoming the curse of degeneracy—for a broad class of regular, ``gap-free'' distributions.
 

\paragraph{Degeneracy in Stochastic Control. } In the broader context of stochastic dynamic control, the impact of fluid degeneracy/instability on algorithm performance has been widely noted, for example, in stochastic network optimization \citep{huang2009delay}, centralized dynamic matching \citep{kerimov2024dynamic, gupta2024greedy, wei2023constant} and network revenue management with reusable resources \citep{xie2024dynamic, balseiro2023dynamic}. While these settings differ from those considered in this work, they bear notable similarities, especially in the role of the fluid relaxation in algorithm design and performance analysis. Particularly, \cite{gupta2024greedy, xie2024dynamic, chen2024incentivizing} highlight the critical role of dual uniqueness in guaranteeing good algorithm performance in their respective settings. These results are primarily in the discrete setting, similar to \cite{bumpensanti2020re}. The perspective provided in this work complements and extends this understanding of the significance of dual uniqueness by offering non-discrete evidence.

\paragraph{Online Stochastic Knapsack. } There is also a rich relevant literature in theoretical computer science studying the problem of online stochastic knapsack. \cite{kleinberg2005multiple} considers a multisecretary problem in the random input model and derive an (asymptotic) competitive ratio of $ 1- \mathcal{O}\left(\sqrt{\frac{1}{k}}\right)$, where $k$ is the fixed budget (analogous to $\bb$ in this paper). Subsequent work \cite{devanur2009adwords, feldman2010online, molinaro2014geometry,agrawal2014dynamic} extend \cite{kleinberg2005multiple} to the general setting of online knapsack/online packing and achieve progressively improving competitive ratios. \cite{hajiaghayi2007automated, alaei2014bayesian,chawla2023static, jiang2022tight} study the $k$-unit prophet problem where the inputs are independent (but not identical) random variables. Competitive ratios of $1 - \tilde{\Theta}(\sqrt{\frac{1}{k}})$ for large $k$ as well as concrete constants for small $k$ were shown. The setting of this work is different from the aforementioned line of works: (i) the performance metric we use in this paper is the additive regret rather than the competitive ratio, and (ii) the focus is on understanding the performance of {\sf CE} in the \emph{i.i.d.} input model for regular distributions, rather than competing against the worst case. The main result of this work can be essentially translated to an $\left(1 - \mathcal{O}(\frac{(\log T)^2}{T})\right)$-competitive ratio for distributions with nice structures and the initial inventory of resources scaling linearly in $T$, for a general online stochastic knapsack problem with \emph{i.i.d.} inputs.    

\subsection{Organization}
The paper is organized as follows. We formally set up the problem in \Cref{sec:setup}. The main results are stated in \Cref{sec:results}. We then discuss the implication of our results on degeneracy in \Cref{sec:implications}. The proof sketch of our main results is provided in \Cref{sec:sketch}. The conclusions follow in
\Cref{sec:conclusion}.

\section{Problem Formulation}\label{sec:setup}
\subsection{Model}\label{sec: formulation}
There are $m$ resources with an initial capacity of $\bb = (b_1, b_2, \dots, b_m)^{\top} \in \RR^m_{\geq 0}$. At each time period $t = 1, \dots, T$, a demand request $(\ba_t, r_t)$ arrives, assumed to be drawn \emph{i.i.d.} from a distribution with joint CDF $F(\cdot),$ where $F, \bb$ and $T$ are known \textit{a priori}. Upon the arrival of the request and the revelation of $(\ba_t, r_t)$, a decision maker (DM) needs to immediately and irrevocably decide whether or not to accept it. The \textsc{accept} decision results in the consumption of $\ba_t = (a_{1 t}, \dots, a_{m t})^{\top}$ of each resource and an earned reward of $r_t$. The \textsc{reject} decision imposes no change on the resources and garners zero reward. The \textsc{accept} decision is feasible if and only if the remaining capacity of resource $i$ is at least $a_i$,  for $i = 1, \dots, m$. The goal of the DM is to maximize the expected total reward collected through the $T$ periods subject to the resource capacity constraint.

A dynamic allocation policy $\pi$ specifies, for each demand sequence realization $\cI \triangleq \{(\ba_t ,r_t)\}_{j = 1}^T$, a (possibly random) sequence of binary decisions $\{x^{\pi}_t\}_{t = 1}^T$ with $x^{\pi}_t = 1 (0)$ corresponding to the accept (reject) decision. $\pi$ is \textit{non-anticipatory} if each $x^{\pi}_t$ is independent of the future request realization $\cI_t \triangleq \{(\ba_j, r_j)\}_{j = t + 1}^T$ and depends only on the current history $\cH_t \triangleq \{(\ba_j, r_j)\}_{j = 1}^t$. $\pi$ is \textit{feasible} if $w.p.1$ the capacity constraint is satisfied throughout, i.e. $\sum_{t = 1}^T a_{i t} x^{\pi}_{t} \leq b_i, i = 1, \dots, m.$ A policy is admissible if it is both non-anticipatory and feasible. Let $\Pi$ denote the set of admissible policies. The DM seeks to maximize $\E\left[\sum_{t = 1}^T r_t x^{\pi}_{t} \right]$, the expected total reward collected under an admissible policy $\pi \in \Pi$.\\
\textbf{Performance Metric.\ } The performance of an admissible policy $\pi$ is typically measured by its revenue loss, namely $\sup_{\tau \in \Pi} \E\left[\sum_{t = 1}^T r_t x^{\tau}_{t} \right] - \E\left[\sum_{t = 1}^T r_t x^{\pi}_{t} \right]$, where $\sup_{\tau \in \Pi} \E\left[\sum_{t = 1}^T r_t x^{\tau}_{t} \right]$ is the optimal dynamic programming (DP) value. However, the so-defined revenue loss is often not a viable performance metric due to the lacking of tractability of the optimal DP value. A common practice is to instead consider upper bounds of the revenue loss, derived by replacing the DP value by the optimal values of certain tractable relaxations. Two most popular such relaxations are the hindsight relaxation and the fluid relaxation. 

\textbf{The hindsight relaxation.\ } Under the hindsight relaxation, the non-anticipatory constraint is removed, and the resource allocation problem becomes a multi-knapsack Linear Program (LP) (for a given demand sequence realization $\cI$):
\begin{align}\label{program: offline primal}
    \max &\quad  \sum_{t = 1}^T r_t x_t\\
    \mbox{s.t.} & \quad \sum_{t = 1}^T a_{i t} x_t \leq b_{i},\quad i = 1, \dots, m, \nonumber
    \\& x_t \in [0, 1], \quad t = 1,\dots, T. \nonumber
\end{align}
Taking expectation of (\ref{program: offline primal}) (over $\cI$) yields the optimal hindsight value, denoted by $V^{\textrm{hind}}_{\bb, T}$.

\textbf{The fluid relaxation.\ } The fluid relaxation further assumes that $T$ is prohibitively large, to the point that all randomness in (\ref{program: offline primal}) is averaged out (with a proper scaling on the order of $\frac{1}{T}$), which yields
\begin{align}\label{program: fluid primal}
    \sup_{x} &\quad  \E_{(\ba, r) \sim F}\left[r x(\ba, r)\right]\\
    \mbox{s.t.} & \quad \E_{(\ba, r) \sim F}[\ba x(\ba, r) ] \leq \bd,\ \ x(\cdot, \cdot) \in [0, 1], \nonumber
\end{align}
where $\bd \triangleq \frac{\bb}{T}$ is the normalized inventory of resources. We denote the optimal value of the above fluid program by $V^{{\rm fluid}}_{\bd}.$ 

It is fairly straightforward to see that $V^{\textrm{fluid}}_{\bd} T \geq V^{\textrm{hind}}_{\bb, T}$ both are upper bounds on the optimal DP value, for $T \geq 1$, for which we omit the proof. We denote by $\textsc{Reg}^{\textrm{hind}}_{\bb, T}(\pi) \triangleq V^{\textrm{hind}}_{\bb, T} - \E\left[\sum_{t = 1}^T r_t x^{\pi}_t\right]$ the \textit{hindsight regret} incurred by policy $\pi$, and $\textsc{Reg}^{\textrm{fluid}}_{\bb, T}(\pi) \triangleq V^{\textrm{fluid}}_{\bd} T - \E\left[\sum_{t = 1}^T r_t x^{\pi}_t\right]$ the \textit{fluid regret} incurred by policy $\pi$. Throughout this paper, the hindsight regret\footnote{The notion ``regret'' is typically used in a learning scenario. We slightly deviate from this convention as in our setting, the underlying distribution $F$ is known. Instead, our DM regrets for not knowing the future demand sequence in advance.} will be our main performance metric. We thereby denote $\textsc{Reg}_{\bb, T}(\pi)  \triangleq \textsc{Reg}^{{\rm hind}}_{\bb, T}(\pi)$. We are interested in characterizing the scaling of $\textsc{Reg}_{\bb, T}(\pi)$ as $T$ grows.

\subsection{The {\sf Certainty Equivalent} Heuristic}\label{sec:algo}

The \textsf{CE} heuristic is a special threshold-based policy that leverages the fluid relaxation (\ref{program: fluid primal}) to facilitate its dynamic decision-making. Following \cite{li2022online, bray2024logarithmic, jiang2022degeneracy}, we consider a dual-based {\sf CE}. The dual of (\ref{program: fluid primal}) has a compact form
\begin{align}\label{eq: dual-formulation}
    \min_{\blambda \in \RR^{m}_{\geq 0}} f_{\bd}(\blambda) \triangleq {\bd}^{\top} \blambda  + \E_{(\ba, r) \sim F}[(r - \ba^{\top} \blambda)^+],
\end{align}
where $f_{\bd}(\cdot)$ is convex for any value that $\bd$ takes. Strong duality holds, namely we have $V^{{\rm fluid}}_{\bd} = \min_{\blambda \in \RR^{m}_{\geq 0}} f_{\bd}(\blambda).$ We refer the reader to \cite{li2022online, balseiro2023survey} for the derivation of \eqref{eq: dual-formulation} and further discussions.

The {\sf CE} heuristic solves a perturbed version of (\ref{eq: dual-formulation}) at each period $t = 1, \dots, T$ to facilitate decision making. More precisely, let $\bb^{\sf CE}_t$ denote the sequence of remaining inventory of resources under the {\sf CE} policy, with $\bb^{\sf CE}_0 = \bb$, and $\bd^{\sf CE}_t \triangleq \frac{\bb^{\sf CE}_{t-1}}{T - t}$ be the corresponding normalized remaining inventory of resources. The {\sf CE} heuristic solves for 
\begin{equation}\label{program: to go fluid dual}
    \tilde \blambda_t \in \argmin_{\blambda \in \RR^{m}_{\geq 0}} \left(\bd^{\sf CE}_t\right)^{\top}\blambda + \E[(r - \ba^{\top} \blambda)^+]
    \end{equation}
at each period $t = 1, \dots, T$. Then the decision follows by setting threshold at $\ba^{\top} \tilde \blambda_t $:
\[ x^{\sf CE}_t = 1 \textrm{\ if and only if\ } r_t \geq \ba^{\top}_t {\tilde \blambda_t} \ \textrm{\ and \ } \ba_t \leq \bb^{\sf CE}_{t-1},\ \ \ \  t = 1, \dots, T, \]
where $ \bb^{\sf CE}_t = \bb^{\sf CE}_{t - 1} - x^{\sf CE}_t \ba_t$, namely, the induced sequence of remaining inventory of resources. The inequality in $\ba_t \leq \bb^{\sf CE}_{t-1}$ is element-wise. A formal description is provided in Algorithm \ref{alg:ce}. 

\begin{algorithm}[ht]
\caption{\textup{\textsf{Certainty Equivalent (CE) Heuristic} }}
\label{alg:ce}
\begin{algorithmic}[1]
\REQUIRE Problem instance $(F, \bb, T).$
        \FOR{$t=1,\ldots, T$}
            \STATE{Observe instance $(\ba_t,r_t)$.}
            \STATE{Solve \eqref{program: to go fluid dual} and obtain $\ft \in \argmin_{\blambda \in \RR^{m}_{\geq 0}}f_{\bd^{\sf CE}_t}(\blambda)$.}
            \STATE{Set \[ x^{\sf CE}_t = 1 \textrm{\ if and only if\ } r_t \geq \ba^{\top}_t {\tilde \blambda_t} \ \textrm{\ and \ } \ba_t \leq \bb^{\sf CE}_{t-1},\] and $x^{\sf CE}_t = 0$ otherwise, where $\bb^{\sf CE}_0 = \bb, \bb^{\sf CE}_t = \bb^{\sf CE}_{t - 1} - x^{\sf CE}_t \ba_t$ is the induced sequence of remaining resources. }
        \ENDFOR
\end{algorithmic}
\end{algorithm}
We take the dual-based definition of {\sf CE} in our setting because dual fluid relaxation (\ref{eq: dual-formulation}), in a cleaner form, both benefits the theoretical analysis and yields a more practical algorithm. In fact, the primal problem (\ref{program: fluid primal}) is a possibly infinite-dimensional optimization problem since we allow $F$ to be a general continuous distribution, while the dual fluid problem (\ref{eq: dual-formulation}) is a convex program with a simple feasible region (the $m$-dimensional positive orthant).

\subsection{Assumptions}\label{sec: assumption}
Our main results are presented under two sets of distributional assumptions on $F$. \textbf{Either} set of assumptions suffices to guarantee the desired performance of {\sf CE} (cf. Theorem \ref{thm: regret of CE}). We first introduce some additional notation. 
\paragraph{Additional notation. } Let $F^{\ba}(\cdot)$ denote the marginal CDF of $\ba$. Further let $F^r_{\ba'}(\cdot)$ denote the conditional CDF of the reward $r$ given $\ba = \ba'$. Let ${\rm supp}(F), {\rm supp}(F^{\ba})$ and ${\rm supp}(F^r_{\ba})$ denote the support of the corresponding distributions, respectively.


\begin{assumption}\label{assum: starting-from-zero}
The joint distribution $F$ satisfies the following conditions:
\begin{enumerate}[label=$(\roman*)$]
    \item (boundedness) There exist constants $0 < \Al \leq  \Au$ and $ \bar r > 0$, s.t. ${\rm supp}(F) \in [\Al, \Ab]^m \times [0, \bar r].$
    \item For any $\ba \in \textrm{supp}(F^{\ba}), \textrm{supp}(F^r_\ba) = [0,r_{\ba}]$ for some $r_{\ba} > 0.$
    \item ((reverse) H\"older condition) There exist non-negative constant $\beta$ and positive constants $\nu, c_\nu, c_\beta$, such that for any $ 0 \leq z_1 < z_2 \leq r_{\ba}$, $
   c_{\beta}(z_2 - z_1)^{1 + \beta} \leq F^r_{\ba}(z_2) - F^r_{\ba}(z_1) \leq c_{\nu}(z_2 - z_1)^{\nu}.$
\end{enumerate}

\end{assumption}

\begin{assumption}\label{assum: small-probability-starting-from-zero}
The joint distribution $F$ satisfies the following conditions:
\begin{enumerate}[label=$(\roman*)$]
\item (boundedness) There exist constants $0 < \Al \leq  \Au$ and $ \bar r > 0$, s.t. ${\rm supp}(F) \in [\Al, \Ab]^m \times [0, \bar r].$

\item ($F$ regularity) $\textrm{supp}(F^{\ba})$ is convex and compact, and the probability density associated to $F^{\ba}$ is bounded from below by a constant $l_f > 0.$ $\textrm{supp}(F^r_\ba) = [l(\ba),r(\ba)]$, where $l(\ba)$ and $r(\ba)$ are Lipschitz continuous as functions of $\ba$ with Lipschitz constant $c_L$.

\item There exist $\ba_o \in {\rm supp}(F^{\ba})$ and $r_0 > 0$ s.t. $l(\ba) = 0$ for all $\ba \in \cB(\ba_o, r_0) \cap \textrm{supp}(F^{\ba}).$

\item ((reverse) H\"older condition) There exist non-negative constant $\beta$ and positive constants $\nu, c_\nu, c_\beta$, such that for any $ l(\ba) \leq z_1 < z_2 \leq r(\ba)$, $
   c_{\beta}(z_2 - z_1)^{1 + \beta} \leq F^r_{\ba}(z_2) - F^r_{\ba}(z_1) \leq c_{\nu}(z_2 - z_1)^\nu.$
\end{enumerate}
\end{assumption}
Assumption \ref{assum: starting-from-zero} \textit{(i)} enforces boundedness on $\ba$ and $r$. In \textit{(ii)}, we require the conditional reward distribution given any $\ba$ to be supported on an interval starting from zero. This is a necessary restriction for {\sf CE} to achieve $o(\sqrt{T})$ regret guarantee, when $\ba$ is allowed to be arbitrarily distributed (cf. Example \ref{exmp:violating unique-dual} for bad examples violating \textit{(ii)}). \textit{(iii)} further excludes point masses on the conditional reward distributions ($\nu$ bounded away from zero), and restricts the minimal rate of probability accumulation ($\beta$ bounded away from $\infty$). In particular, $\beta = 0$ and $\nu = 1$ corresponds to the case where the conditional distribution of the reward has lower and upper bounded density on the support. We remark on the asymmetric role of parameters $\beta$ and $\nu$: $\beta$ critically affects the best achievable regret scaling, while $\nu$ only appears in the constant term and does not affect the regret scaling (cf. Theorem \ref{thm: regret of CE}, Proposition \ref{thm: fundamental lower bound}). We refer the readers to Appendix \ref{appendix: second-order} for more discussion on parameter $\beta$, in connection with the second-order growth conditions often made in the literature. 

Assumption \ref{assum: small-probability-starting-from-zero} relaxes the global requirement of Assumption \ref{assum: starting-from-zero} \textit{(ii)} to only a local condition \textit{(iii)}, to hold only in the neighborhood of an arbitrary point in the support. To permit this relaxation, Assumption \ref{assum: small-probability-starting-from-zero} \textit{(ii)} enforces natural regularity of $F$. Intuitively, it requires the conditional reward distributions to change ``continuously'' in $\ba$, thus excluding the possible ``holes'' in the support of the joint distribution. We still require a weak(local) condition in \textit{(iii)} to avoid hard corner-case instances, for which we provide further explanation and discussion in Appendix \ref{appendix: second-order} (cf. Example \ref{exmp:violating 2-order}).

In the sequel, we illustrate Assumption \ref{assum: starting-from-zero} and Assumption \ref{assum: small-probability-starting-from-zero} with several examples. 

\begin{exmp}[Multisecretary problem]\label{exmp:multisec}
    The multisecretary problem has $m = 1$ and $F^{\ba}(\cdot) = \delta_1(\cdot)$. Suppose the reward distribution is (i) without point mass, and (ii) supported on an interval starting from zero, namely, without gaps, then Assumption \ref{assum: starting-from-zero} holds. A notable example is a reward distribution specified by \emph{p.d.f.} $h_r(x) = (1 + \beta)|1 - 2 x|^{\beta},$ for $ 0 \leq x \leq 1$ with $\beta \geq 0$.
\end{exmp}

\begin{exmp}[Hyper-cube models]\label{exmp: hyper-cube}
Suppose ${\rm supp}\left(F\right) = [1, 2]^m\times [0, 1]$. The joint density function of $F$ is bounded from above and below by a pair of positive constants on the support. Then both Assumptions \ref{assum: starting-from-zero} and \ref{assum: small-probability-starting-from-zero} are satisfied with $\beta = 0$ and $\nu = 1$.
\end{exmp}

\begin{exmp}[Generalized linear models]\label{exmp: linear}
Let $r = g\left(\ba^{\top} \bz \right)+ \epsilon$ for a non-negative, Lipschitz continuous function $g$, a fixed vector $\bz \in \RR^m_{\geq 0}$ and a noise random variable $\epsilon$, with $F^{\ba}$ satisfying Assumption \ref{assum: small-probability-starting-from-zero} \textit{(i)} and \textit{(ii)}. $\epsilon$ is supported on an interval $[-\cL, \cL]$ with \emph{p.d.f.} bounded from above and below by a pair of positive constants. Suppose there exists $\ba'$ such that $g\left((\ba')^\top \bz\right) < L - \eta$ for a positive constant $\eta$. Then Assumption \ref{assum: small-probability-starting-from-zero} effectively holds with $\beta = 0$ and $\nu = 1$, modulo the possible negative rewards which can be dealt with in a straightforward manner. 
\end{exmp}

\begin{rmk}
The simple structure of multisecretary allows for a further relaxation of Assumption \ref{assum: starting-from-zero} \textit{(ii)}, such that ${\rm supp}\left(F^r_{\ba}\right)$ is an interval not necessarily starting from zero, under which our main results (cf. Theorem \ref{thm: regret of CE}) remain valid.
\end{rmk}

\begin{rmk}
We defer a formal analysis of Example \ref{exmp: linear} to Appendix \ref{appendix: example and assumption}. Note that negative rewards, as appeared in Example \ref{exmp: linear}, do not incur regret:  the optimal decision is always to reject them. They will not affect our analysis. 
\end{rmk}

Later in \Cref{sec:results}, we state the regret guarantee of \textsf{CE} for each of the above examples under no other assumptions (cf. Corollary \ref{coro: regret of examples}). To obtain similar regret guarantees,  prior work typically imposes additional fluid regularity conditions (cf. Section \ref{sec:results} for details). In later sections we provide a systematic review—and a comparison with our own assumptions—of these fluid regularity conditions, with a detailed investigation of non-degeneracy conditions in Section \ref{sec:implications}, and discussion on the second-order growth conditions in Appendix \ref{appendix: second-order}.
\section{Main Results}\label{sec:results}
\subsection{Achievable regret}
\begin{thm}[Achievable regret of {\sf CE}]\label{thm: regret of CE}
    Under either Assumption \ref{assum: starting-from-zero} or Assumption \ref{assum: small-probability-starting-from-zero}, the {\sf CE} heuristic achieves a hindsight regret  
    \begin{align*}
    \textsc{Reg}_{\bb, T}(\pi^{\sf CE}) \ \leq \begin{cases}
        \cC (\log T)^2 \quad\quad &\beta = 0,\\
        \tilde \cC T^{\frac{1}{2} - \frac{1}{2(1 + \beta)}}  (\log T)^{\frac{2 + \beta}{2 + 2 \beta}} \quad\quad &\beta > 0,
    \end{cases}
\end{align*}
for arbitrary $\bb$ and $T>3$, where $\cC$ and $\tilde \cC$ are constants independent of $T$ and $\bb$ and depend only on model primitives regarding $F$ through the two assumptions, respectively. 
\end{thm}
\begin{rmk}
    Precise forms of $\cC$ and $\tilde \cC$ are provided in Appendix \ref{appendix: proof of ce regret thm}.
\end{rmk}
Theorem \ref{thm: regret of CE} establishes theoretical guarantees for {\sf CE} for a wide range of OLP instances, requiring only that $F$ belongs to specific distribution classes that are both natural and easy to verify. This significantly relaxes the conditions typically needed to obtain $o(\sqrt{T})$ regret, in particular, the non-degeneracy conditions and/or (uniform) second-order growth conditions. In Section \ref{sec:implications} and Appendix \ref{appendix: second-order}, we systematically examine the relationship between these fluid regularity conditions and our conditions (Assumptions \ref{assum: starting-from-zero} and \ref{assum: small-probability-starting-from-zero}), demonstrating that the former are not only technically unnecessary for algorithmic analysis but are often overly restrictive. In that sense, Theorem \ref{thm: regret of CE} extends the state-of-the-art understanding of {\sf CE}’s range of effectiveness. We provide a proof sketch of Theorem \ref{thm: regret of CE} in Section \ref{sec:sketch}. The detailed proof can be found in Appendix \ref{appendix: proof of ce regret thm}. With Theorem \ref{thm: regret of CE}, the following regret scaling of {\sf CE} on concrete examples is an immediate corollary.
\begin{coro}\label{coro: regret of examples}
   The regret guarantee in Theorem \ref{thm: regret of CE} holds for multisecretary instances of Example \ref{exmp:multisec} with each choice of $\beta$. For instances of Example \ref{exmp: hyper-cube} and Example \ref{exmp: linear}, \textsf{CE} achieves $\mathcal{O}\left((\log T)^2\right)$ regret. 
 \end{coro}


\subsection{Fundamental regret lower bound}

The achievable regret of \textsf{CE} stated in Theorem \ref{thm: regret of CE} is near optimal, as we formalize through a fundamental regret lower bound. 
\begin{prop}[Fundamental Regret Lower Bound]\label{thm: fundamental lower bound}
There exists OLP instance $(F, \bb, T)$ satisfying Assumption \ref{assum: starting-from-zero} and Assumption \ref{assum: small-probability-starting-from-zero}, such that 
\begin{align*}
    \inf_{\pi}\textsc{reg}_{\bb,T}(\pi) \geq \frac{c}{1 + \beta} T^{\frac{1}{2} - \frac{1}{2(1 + \beta)}} \mathbbm{1}(\beta > 0) + c \log T \mathbbm{1}(\beta = 0),
\end{align*}
where $c$ is a constant independent of $T, \bb$ and $\beta$. 
\end{prop}

\begin{rmk}
    The fundamental regret lower bound is inspired by earlier works on the multisecretary problem (\cite{bray2024logarithmic} for $\beta = 0$ and \cite{besbes2024dynamic} for $\beta > 0$). While \cite{besbes2024dynamic} initially considers reward distributions with a gap in their support, we observe that these gaps can be removed so that the resulting instances satisfy our Assumptions, yet their analysis still works (formally, by setting $g = 0$ in the proof of Theorem 1 in \cite{besbes2024dynamic}) and the regret lower bound still holds. In particular, the lower bound is attained in the multisecretary instance of Example~\ref{exmp:multisec}. We refer the interested reader to \cite{besbes2024dynamic} (cf. Theorem 1) for detailed proof.
\end{rmk}

  Combining Theorem \ref{thm: regret of CE} with Proposition \ref{thm: fundamental lower bound}, we find that under our assumptions, the \textsf{CE} heursitic can achieve a regret that matches the fundamental lower bound up to a polylogarithmic factor in $T$, thus demonstrating its near-optimality. 



\section{Does Degeneracy Cause {\sf CE} to Fail?}\label{sec:implications}



Fluid degeneracy is generally believed to cause the failure of {\sf CE} in stochastic control. In this section, we challenge—and refine—this perspective within the OLP framework. Building on Theorem \ref{thm: regret of CE}, we demonstrate a large class of instances that, despite violating standard non-degeneracy conditions in the literature, still achieve low regret. By closely examining the concept of degeneracy, we attribute this deviation from conventional understanding to a key distinction between how degeneracy manifests in discrete versus non-discrete contexts. This analysis enables us to systematically review existing notions of non-degeneracy and shed light on the critical geometric structures of the fluid problem that fundamentally drive the regret accumulation of {\sf CE}.


\subsection{Existing non-degeneracy conditions}
Non-degeneracy conditions are imposed on the associated fluid problem of an OLP instance. More precisely, we specify these conditions with respect to a fluid instance $(F, \bd)$, where we recall that $\bd = \frac{\bb}{T}$ is the normalized (initial) resource capacity.

\begin{assumption}[Primal stability condition]\label{def: dual stability}
There exists an optimal $\bx^{\star}_{\bd}$ to primal problem (\ref{program: fluid primal}), and a neighborhood $\mathcal{N}(\bd)$ of $\bd$, such that for any $\bd' \in \mathcal{N}(\bd)$, there exists an optimal $\bx^{\star}_{\bd'}$ to the perturbed problem (\ref{program: fluid primal}) with RHS constraint $\bd'$, for which the set of binding resource constraints remains unchanged comparing to that of $\bx^{\star}_{\bd}$.
\end{assumption}

\begin{assumption}[Strict complementary slackness condition]\label{def: dual non-degeneracy} 
    $\blambda^{\star}$ is the unique solution to the dual problem (\ref{eq: dual-formulation}) $\min_{\blambda \in \RR^m_{\geq 0}} f_{\bd}(\blambda)$. Furthermore, there exists an optimal $\bx^{\star}$ such that strict complementary slackness is satisfied. In particular, when $F$ has no point mass, the condition is typically stated without involving $\bx^{\star}$: for any $i = 1, \dots, m$, $\blambda^{\star}_i = 0$ if and only if $ d_i - \E_{(\ba, r) \sim F}\left[a_i \mathbbm{1}\left(r > \ba^{\top}\blambda^{\star}\right)\right] > 0.$
\end{assumption}
Assumptions \ref{def: dual stability} and \ref{def: dual non-degeneracy} are two commonly imposed ``non-degeneracy'' conditions in prior work (cf. Assumption 2 and SC 8 of \cite{balseiro2023survey}, Assumption 2(c) of \cite{li2022online}, Assumption 5 of \cite{bray2024logarithmic}), where it is known that Assumption \ref{def: dual non-degeneracy}, together with certain smoothness conditions on $F$ implies Assumption \ref{def: dual stability} (cf. Lemma 3 of \cite{balseiro2023survey}). We note that the definition of Assumption \ref{def: dual stability} allows for potentially multiple primal optimal solutions\footnote{Here we say two solutions $\bx$ and $\bx'$ are different in a probability sense. Namely, $\PP\left(\bx(\ba, r) \neq \bx'(\ba, r)\right) > 0$}. Assumption \ref{def: dual non-degeneracy} usually appears in its simplified form that only contains $\blambda^{\star}$.

Assumptions \ref{def: dual stability} and \ref{def: dual non-degeneracy} are critical to the low-regret analysis in the aforementioned prior work. More precisely, these non-degeneracy conditions ensure that the binding resource constraints in the fluid problems solved in each step of Algorithm \ref{alg:ce} remain unchanged during the algorithm's execution, thus permitting the martingale-based argument used in these prior work that ultimately leads to $\mathcal{O}(\log T)$ regret. The only known performance guarantee of {\sf CE} that does not rely on these non-degeneracy conditions is provided by \cite{jiang2022degeneracy}, where they instead have to impose a strong, uniform version of second-order growth conditions, for which we defer more discussion to Appendix \ref{appendix: second-order}.

In addition to Assumptions \ref{def: dual stability} and \ref{def: dual non-degeneracy}, the following condition is often assumed implicitly.
\begin{assumption}[Dual uniqueness condition]\label{def: unique-dual} 
     $\blambda^{\star}$ is the unique solution to the dual problem (\ref{eq: dual-formulation}).
\end{assumption}
Assumption~\ref{def: unique-dual} and Assumptions~\ref{def: dual stability}, \ref{def: dual non-degeneracy} do not imply each other in general. 
\begin{exmp}\label{exmp:dual unique but non-degenerate}
    $\P(\ba = 1) = \P(\ba = 2) = \frac{1}{2}$. $ r \equiv {\sf Unif}[0, 1]$ for both values of $\ba$. The initial inventory for the resource is $\bb = \bd T = 1.5 T$. Then Assumption~\ref{def: dual stability} is violated yet Assumption~\ref{def: unique-dual} holds.  
\end{exmp}
\begin{exmp}\label{exmp: non-degenerate but multiple dual}
    $\ba \equiv 1$. $ r = {\sf Unif}\left([0, \frac{1}{3}]\cup [\frac{2}{3}, 1]\right)$. The initial inventory for the resource is $\bb = \bd T = 0.5 T$. Then Assumption~\ref{def: dual stability} and \ref{def: dual non-degeneracy} both hold yet Assumption~\ref{def: unique-dual} is violated.  
\end{exmp}
\begin{rmk}
    Example~\ref{exmp: non-degenerate but multiple dual} is the multisecretary problem with gap on its reward distribution.
\end{rmk} 
\begin{rmk}
    Some authors define degeneracy as the violation of dual uniqueness Assumption \ref{def: unique-dual}. However, as shown in Examples~\ref{exmp:dual unique but non-degenerate} and \ref{exmp: non-degenerate but multiple dual}, dual uniqueness and the commonly adopted notion of non-degeneracy are in general not equivalent. In fact, the dual uniqueness is typically implied instead by the second-order growth conditions (see e.g. Assumption 2(b) of \cite{li2022online} and Assumption 2.1 of \cite{balseiro2023dynamic}). In this work, we follow the convention \citep{li2022online, bray2024logarithmic, balseiro2023survey, jiang2022degeneracy} where ``degeneracy'' refers to the violation Assumptions \ref{def: dual stability} and \ref{def: dual non-degeneracy}.
\end{rmk}


\subsection{The (near)optimality of {\sf CE} beyond non-degeneracy}

All aforementioned non-degeneracy conditions impose restrictions on both the underlying distribution $F$ and the normalized resource capacity, $\bd$, whereas Assumptions \ref{assum: starting-from-zero} and \ref{assum: small-probability-starting-from-zero} specify distribution classes which purely rely on the properties of $F$. The following lemma explicitly characterizes the distinction between Assumptions \ref{assum: starting-from-zero} and \ref{assum: small-probability-starting-from-zero} and the non-degeneracy conditions in the literature. 

\begin{lemma}\label{lem: non essential non degeneracy}
For any distribution $F$ satisfying Assumption \ref{assum: starting-from-zero} or Assumption \ref{assum: small-probability-starting-from-zero}, there exists $\bd \in \RR^m_{\geq 0}$ such that $(F, \bd)$ violates the non-degeneracy conditions (Assumptions \ref{def: dual stability} and \ref{def: dual non-degeneracy}).
\end{lemma}
Together with Theorem \ref{thm: regret of CE}, we immediately conclude that non-degeneracy is irrelevant to the performance of {\sf CE}.
\begin{coro}\label{coro: non essential non degeneracy}
Assumptions \ref{def: dual stability} and \ref{def: dual non-degeneracy} are \textbf{not} necessary for {\sf CE} to achieve $o(\sqrt{T})$ regret.
\end{coro}
We note that the type of degeneracy characterized in Lemma \ref{lem: non essential non degeneracy} is hardly a theoretical artifact; it often arises in practical applications. In scenarios where the inventory of resources is endogenously determined, the vector $\bd$ frequently becomes asymptotically close to $\EE_{(\ba, r) \sim F}\left[\ba \mathbbm{1}\left(r > \ba^{\top}\blambda^{\star}\right)\right]$ to avoid waste, typically within $\mathcal{O}\left(\frac{1}{\sqrt{T}}\right)$, as dictated by the square-root inventory law \citep{jiang2022degeneracy, bumpensanti2020re}. Now, any redundant resource with zero dual price in the fluid limit would result in degeneracy. We defer a detailed proof to Appendix \ref{appendix: degeneracy}.

Corollary \ref{coro: non essential non degeneracy} contradicts the general view and prior evidence that degeneracy causes {\sf CE} to fail (cf. Figure 4 in \cite{bumpensanti2020re}), and motivates a deeper investigation into the role that degeneracy plays in the regret accumulation of {\sf CE}. In the next two sections, we present two complementary perspectives, one rooted in the discrete contexts and the other in the non-discrete contexts. Together, they establish a more comprehensive understanding of whether and how degeneracy impacts the performance of {\sf CE}, resolving the paradox.



\subsection{Fluid degeneracy: the discrete perspective}
Consider a discrete $F$ supported on $n$ request types ${(\ba^1, r^1), \dots, (\ba^n, r^n)}$, with corresponding probabilities $p_1, \dots, p_n$.  This is a setting that receives dominant attention in the NRM literature. Here, the primal fluid problem (\ref{program: fluid primal}) becomes an LP (often referred to as the Deterministic Linear Program (DLP)). Concretely,
\begin{align}\label{program: fluid primal-DLP}
    \max_{x} &\quad \sum_{j = 1}^n p_j r^j x_j\\
    \mbox{s.t.} & \quad \sum_{j = 1}^n \ba^j p_j x_j \leq \bd, \ \ 0 \leq \bx \leq 1. \nonumber
\end{align}
In the NRM literature, ``non-degeneracy'' traditionally refers to the (LP) non-degeneracy of the DLP \citep{jasin2012re}. More precisely, 
\begin{assumption}[DLP non-degeneracy]\label{def: DLP non-degeneracy}
The DLP (\ref{program: fluid primal-DLP}) has a non-degenerate optimal solution $\bx^{\star}$,
    \begin{equation}\label{eq: DLP non degeneracy}
   \left|j \in [n]: x^*_j = 0\ \textrm{or} \ x^*_j = 1\right| + |i \in [m]: \sum_{j = 1}^n p_j a^j_i x^*_j = d_i | = n.
\end{equation}
\end{assumption}
\cite{bumpensanti2020re} demonstrates both numerically and theoretically that violating Assumption \ref{def: DLP non-degeneracy} indeed results in performance deterioration of {\sf CE} (cf. Figure 4, Propositions 2\&3 in \cite{bumpensanti2020re}), establishing the necessity of DLP non-degeneracy for guaranteeing {\sf CE}'s performance in the discrete setting. Now one may naturally wonder about the connection between Assumption \ref{def: DLP non-degeneracy} and the aforementioned conditions. Let's first introduce a stronger version of Assumption~\ref{def: dual stability} in the discrete setting.
\begin{assumption}[Primal stability in the discrete setting]\label{def: dual stability LP}
    There exists an optimal $\bx^{\star}_{\bd}$ to DLP (\ref{program: fluid primal-DLP}), and a neighborhood $\mathcal{N}(\bd)$ of $\bd$, such that for any $\bd' \in \mathcal{N}(\bd)$, there exists an optimal $\bx^{\star}_{\bd'}$ to the perturbed problem (\ref{program: fluid primal}) with RHS constraint $\bd'$, for which the set of binding constraints in \eqref{program: fluid primal-DLP} remains unchanged comparing to that of $\bx^{\star}_{\bd}$.
\end{assumption}
Assumption~\ref{def: dual stability LP} is nearly identical to Assumption~\ref{def: dual stability}, only that the former requires additionally that the $0 \leq \bx \leq 1$ constraints are also stable. Therefore Assumption~\ref{def: dual stability LP} implies Assumption~\ref{def: dual stability}. In the discrete setting, we establish a connection between the various assumptions.
\begin{lemma}\label{lem: non-degeneracy discrete}
    Suppose $F$ is a discrete distribution. Premised that the optimal solution $\bx^{\star}$ to DLP (\ref{program: fluid primal-DLP}) is unique, Assumptions \ref{def: dual stability LP}, \ref{def: dual non-degeneracy}, \ref{def: unique-dual} and \ref{def: DLP non-degeneracy} are equivalent.
\end{lemma}
\begin{rmk}
The premise of primal $\bx^{\star}$ uniqueness is not the weakest under which the equivalence in Lemma \ref{lem: non-degeneracy discrete} holds, but exploring it further is beyond the current scope. 
\end{rmk}
We defer the proof to Appendix \ref{appendix: degeneracy} and instead provide an intuitive explanation here. Lemma \ref{lem: non-degeneracy discrete} demonstrates that in the discrete setting, there is essentially only one type of degeneracy: either all non-degeneracy conditions hold simultaneously and {\sf CE} performs well, or the problem exhibits degeneracy, resulting in {\sf CE}’s failure. Geometrically, this dichotomy arises from the rigid, piecewise-linear structure of the fluid problem in discrete settings, as is illustrated in Figure \ref{Fig: discrete-dege}.

\begin{figure}[h!]
    \centering
    \begin{subfigure}[b]{0.3\textwidth}
        \centering
        \begin{tikzpicture}[scale=1]
            \draw[->] (0,0) -- (3.5,0) node[right] {$d$};
            \draw[->] (0,0) -- (0,3.5) node[above] {$V^{\rm fluid}_d$};
            
            \draw[blue, thick] (0,0) -- (1,1);
            \draw[red, thick] (1,1) -- (2.5,2.5);
            \draw[red] (2.5,2.5) -- (3,2.5);

            \filldraw[black] (1,1) circle (2pt);
            \filldraw[black] (2.5,2.5) circle (2pt);

            \draw[dotted] (1,0) -- (1,1);
            \draw[dotted] (2.5,0) -- (2.5,2.5);

            \node[below] at (1,0) {$d^2$};
            \node[below] at (2.5,0) {$d^1$};
        \end{tikzpicture}
        \caption{Function $V^{\rm fluid}_d$ at degenerate and non-degenerate points}
        \label{Fig: discrete-dege-V}
    \end{subfigure}
    \hfill
    \begin{subfigure}[b]{0.3\textwidth}
        \centering
        \begin{tikzpicture}[scale=1]
            \draw[->] (0,0) -- (3.5,0) node[right] {$\lambda$};
            \draw[->] (0,0) -- (0,3.5) node[above] {$f_{d^1}(\lambda)$};

            \draw[thick] (0,1) -- (1,1) -- (2,1.5) -- (3, 3);
            \draw[thick, dotted]  (0,1) -- (1,1.2) -- (2,1.9) -- (3, 3.6);
            \draw[thick, dotted]  (0,1) -- (1,0.8) -- (2,1.1) -- (3, 2.4);

            \draw[red, thick] (0,1) -- (1,1);

            \filldraw[black] (1,1) circle (2pt);
            \filldraw[black] (2,1.5) circle (2pt);
            \filldraw[black] (1,1.2) circle (2pt);
            \filldraw[black] (2,1.1) circle (2pt);
            \filldraw[black] (2,1.9) circle (2pt);
             \filldraw[blue] (0,1) circle (2pt);
            \filldraw[blue] (1,0.8) circle (2pt);
        \end{tikzpicture}
        \caption{Dual function at degenerate point. Multiple optimal dual solutions and instability.}
        \label{Fig: discrete-dege-f}
    \end{subfigure}
    \hfill
    \begin{subfigure}[b]{0.3\textwidth}
        \centering
        \begin{tikzpicture}[scale=1]
            \draw[->] (0,0) -- (3.5,0) node[right] {$\lambda$};
            \draw[->] (0,0) -- (0,3.5) node[above] {$f_{d^2}(\lambda)$};

            \draw[thick] (0.,1.5) -- (1,1) -- (2,1.5) -- (2.5,2.5);
            \draw[thick, dotted] (0.,1.5) -- (1,1.2) -- (2,1.9) -- (2.5,3);
            \draw[thick, dotted] (0.,1.5) -- (1,0.8) -- (2,1.1) -- (2.5,2);
            \filldraw[blue] (1,1) circle (2pt);
            \filldraw[blue] (1,1.2) circle (2pt);
            \filldraw[blue] (1,0.8) circle (2pt);   
            \filldraw[black] (2,1.5) circle (2pt);
            \filldraw[black] (2,1.9) circle (2pt);
            \filldraw[black] (2,1.1) circle (2pt);
        \end{tikzpicture}
        \caption{Dual function at non-degenerate point. Unique, stable optimal dual solution.}
        \label{Fig: discrete-nondege-f}
    \end{subfigure}
    \caption{Function $V^{\rm fluid}$ and dual function $f_{d}$ in degenerate and non-degenerate settings.}
    \label{Fig: discrete-dege}
\end{figure}
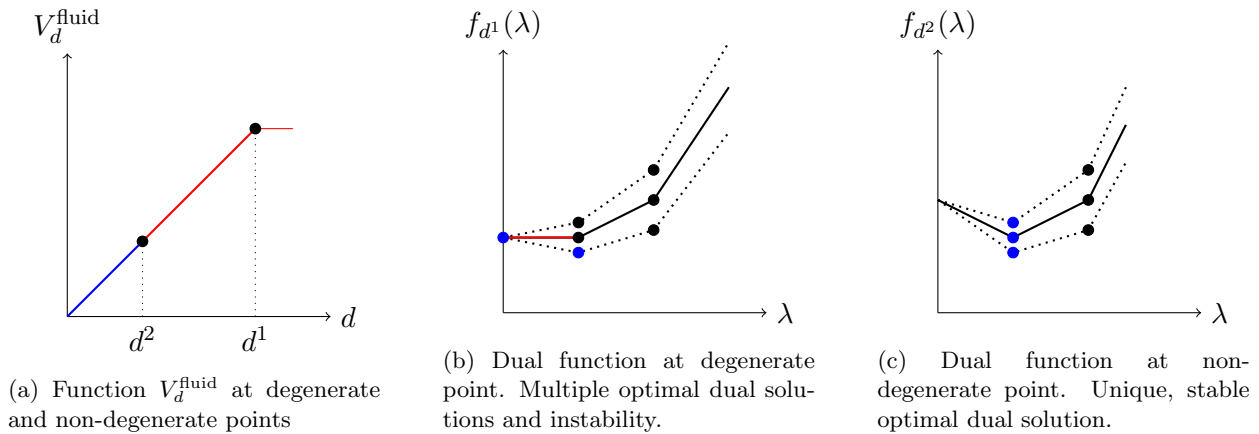

We plot $V^{\rm fluid}_{\bd} := \min_{\blambda \in \RR^m_{\geq 0}} f_{\bd}(\blambda)$ as a function of $\bd$ in Figure \ref{Fig: discrete-dege-V}, and $f_{\bd}(\blambda)$ as a function of $\blambda$ with $\bd = d^1$ (degenerate case), and $\bd = d^2$ (non-degenerate case) in Figure \ref{Fig: discrete-dege-f} and Figure \ref{Fig: discrete-nondege-f}, respectively. We also plot how perturbing $\bd$ around $d^1$ or $d^2$ affects functions $f_{\bd}(\blambda)$. As shown in Figure \ref{Fig: discrete-dege}, fluid degeneracy manifests in several equivalent forms. In Figure \ref{Fig: discrete-dege-V}, the degenerate point $d^1$ corresponds to a kink in the function $V^{\rm fluid}_{\bd}$. By standard LP theory (cf. \cite{bertsimas1997introduction}), the supergradient of $V^{\rm fluid}_{\bd}$ at $d^1$ are the optimal dual solutions to $\min_{\blambda \in \RR^m_{\geq 0}} f_{d^1}(\blambda)$. The existence of multiple supergradients at $d^1$ thereby corresponds to multiple dual solutions, as illustrated in Figure \ref{Fig: discrete-dege-f}. Furthermore, we observe in Figure \ref{Fig: discrete-dege-f} that perturbing $\bd$ around $d^1$ results in drastic changes of the dual solution, demonstrating instability. In contrast, the non-degenerate case avoids all these issues. In Figure \ref{Fig: discrete-dege-V}, the function $V^{\rm fluid}_{\bd}$ is smooth at $d^2$, with a unique gradient. In Figure \ref{Fig: discrete-nondege-f},  $f_{d^2}(\blambda)$ has a unique minimizer, and perturbing $\bd$ around $d^2$ does not change the minimizer, demonstrating strong stability. 

With Lemma \ref{lem: non-degeneracy discrete}, the necessity of DLP non-degeneracy (Assumption \ref{def: DLP non-degeneracy}) for {\sf CE} to achieve low regret simultaneously extends to all notions of non-degeneracy. Thus it is with minimal loss of accuracy to assert that ``degeneracy causes {\sf CE} to fail.''


\subsection{Fluid degeneracy: the non-discrete perspective}
The clear and unified picture of degeneracy in discrete settings does not extend to non-discrete settings. The following result demonstrates how the equivalence between different notions of degeneracy breaks down.
\begin{lemma}\label{lem: nondiscrete non kink}
Suppose non-discrete distribution $F$ is such that function $V^{\rm fluid}_{\bd}$ is smooth in $\bd$, then Assumption \ref{def: unique-dual} always holds regardless of $\bd$. In contrast, whether Assumptions \ref{def: dual stability} and \ref{def: dual non-degeneracy} hold or not depend on $\bd$.
\end{lemma}
\begin{figure}[h!]
    \centering
    \begin{subfigure}[b]{0.3\textwidth}
        \centering
        \begin{tikzpicture}[scale=1]
            \draw[->] (0,0) -- (3.5,0) node[right] {$d$};
            \draw[->] (0,0) -- (0,3.5) node[above] {$V^{\rm fluid}_d$};
            
            \draw[blue, thick] plot[domain=0.1:2,samples=100] (\x,{-(\x - 2)^2/2 + 2});
            \draw[blue, thick] plot[domain=2:2.5,samples=100] (\x,{ 2});

            \filldraw[black] (1,1.5) circle (2pt);
            \filldraw[black] (2,2) circle (2pt);  
            
            \draw[dotted] (1,0) -- (1,1.5);

            \draw[dotted] (2,0) -- (2,2);

            \draw[densely dashed] (0,0.5) -- (2,2.5);  

            \node[below] at (1,0) {$d^2$};
            \node[below] at (2,0) {$d^1$};
        \end{tikzpicture}
        \caption{Function $V^{\rm fluid}_d$ and uniqueness of gradient}
        \label{Fig: nondiscrete-dege-V}
    \end{subfigure}
    \hfill
    \begin{subfigure}[b]{0.3\textwidth}
        \centering
        \begin{tikzpicture}[scale=1]
            \draw[->] (0,0) -- (3.5,0) node[right] {$\lambda$};
            \draw[->] (0,0) -- (0,3.5) node[above] {$f_{d^1}(\lambda)$};

            \draw[thick, domain=0:2,samples=100] plot (\x,{(\x)^2/2 + 1.5});
            \draw[thick, dotted, domain=-0.5:2,samples=100] plot (\x,{(\x)^2/2 + 1.5 + 0.4 * \x});
            \draw[thick, dotted, domain=0:2,samples=100] plot (\x,{(\x)^2/2 + 1.5 - 0.5 * \x});
            \filldraw[blue] (0,1.5) circle (2pt);  
            \filldraw[blue] (0.5,0.125 + 1.5 - 0.25) circle (2pt);

        \end{tikzpicture}
        \caption{Dual function at $d^1$. Unique degenerate optimal dual solution and binding set instability.}
        \label{Fig: nondiscrete-dege-f}
    \end{subfigure}
    \hfill
    \begin{subfigure}[b]{0.3\textwidth}
        \centering
        \begin{tikzpicture}[scale=1]
            \draw[->] (0,0) -- (3.5,0) node[right] {$\lambda$};
            \draw[->] (0,0) -- (0,3.5) node[above] {$f_{d^2}(\lambda)$};

            \draw[thick, domain=0.0:2.5,samples=100] plot (\x,{(\x)^2/2 + 1.5 - \x});
            \draw[thick, dotted, domain=-0.5:2,samples=100] plot (\x,{(\x)^2/2 + 1.5 + 0.4 * \x - \x});
            \draw[thick, dotted, domain=0:2,samples=100] plot (\x,{(\x)^2/2 + 1.5 - 0.5 * \x - \x});

            \filldraw[blue] (1,1) circle (2pt);  
            \filldraw[blue] (0.6,0.36/2 + 1.5 - 0.36) circle (2pt);  
            \filldraw[blue] (1.5 ,1.5^2/2 + 1.5 - 1.5^2) circle (2pt);  
        \end{tikzpicture}
        \caption{Dual function at $d^2$. Unique non-degenerate optimal dual solution.}
        \label{Fig: nondiscrete-nondege-f}
    \end{subfigure}
    \caption{Function $V^{\rm fluid}$ and dual function $f_{d}$ for two different parameters in a non-discrete setting.}
    \label{Fig: nondiscrete-dege}
\end{figure}
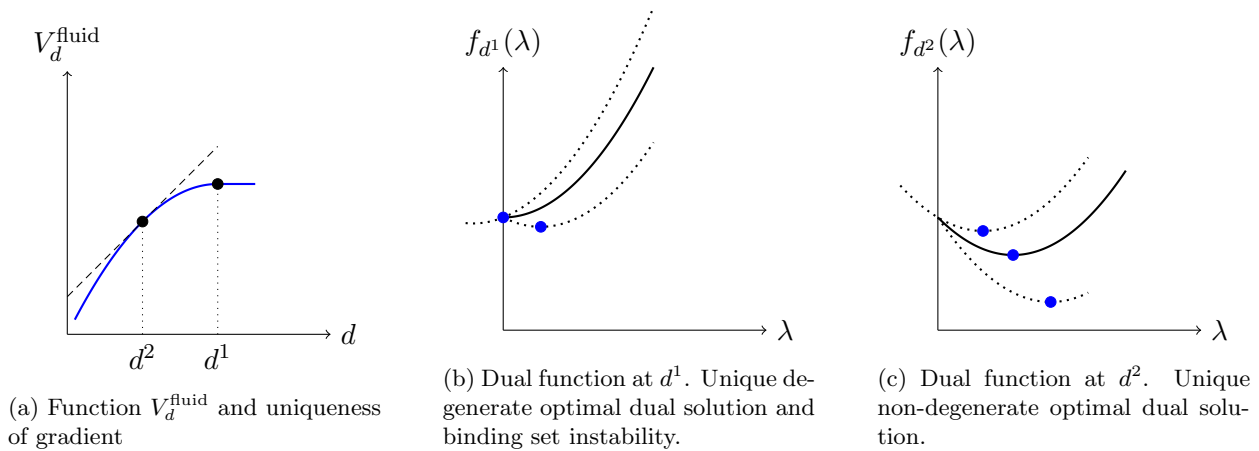
 Figure \ref{Fig: nondiscrete-dege} visually illustrates the intuition behind Lemma \ref{lem: nondiscrete non kink}, the proof of which we omit. In Figure \ref{Fig: nondiscrete-dege-V}, we plot the function $V^{\rm fluid}_{\bd}$. Unlike the discrete case shown in Figure \ref{Fig: discrete-dege-V}, $V^{\rm fluid}_{\bd}$ here is smooth, with no kinks, and its gradient exists and is unique for any $\bd$. As established in standard convex optimization theory (cf. \cite{boyd2004convex}), the gradient of $V^{\rm fluid}_{\bd}$ corresponds to the dual optimal solution. Consequently, Assumption \ref{def: unique-dual} is satisfied for all $\bd$, as illustrated in Figures \ref{Fig: nondiscrete-dege-f} and \ref{Fig: nondiscrete-nondege-f}.
 
 On the other hand, in Figure \ref{Fig: nondiscrete-dege-f} we plot the dual function, projected onto a dimension in the dual space, and illustrate the degenerate scenario where both Assumptions \ref{def: dual stability} and \ref{def: dual non-degeneracy} fail. At $d^1$, $\lambda = 0$ is the unique dual price. However, $\lambda = 0$ happens to be the global minimizer of the dual function $f_{d^1}(\lambda)$, indicating that the corresponding resource happens to be binding, thus violating the strict complementary slackness condition. Small perturbations around $d^1$ lead the dual price to oscillate between zero and a positive value, causing the resource’s status to alternate between binding and non-binding, violating the primal stability condition. Hence, the non-degenerate conditions fail at $d^1$. Consequently, Assumptions \ref{def: dual stability} and \ref{def: dual non-degeneracy} fundamentally rely on  $\bd$ to avoid such degenerate cases. Overall, Lemma \ref{lem: nondiscrete non kink} essentially implies that Assumption \ref{def: unique-dual} and Assumptions \ref{def: dual stability} and \ref{def: dual non-degeneracy} can not be equivalent in non-discrete settings with smooth $V^{\rm fluid}_{\bd}$.

A key feature of the non-discrete setting is that, the premise in Lemma \ref{lem: nondiscrete non kink}, namely that $V^{\rm fluid}_{\bd}$ is smooth, is fairly common and satisfied by a wide class of distributions $F$, in particular, those defined by Assumptions \ref{assum: starting-from-zero} or \ref{assum: small-probability-starting-from-zero}.
\begin{lemma}\label{lem: smoothness of dual objective under our assumption}
    For distribution $F$ satisfying either Assumption \ref{assum: starting-from-zero} or Assumption \ref{assum: small-probability-starting-from-zero}, the corresponding dual objective $f_\bc(\blambda)$ is smooth for any fixed $\bc$. In particular, $\nabla f_{\bc}(\blambda)$ exists. Furthermore, $V^{\rm fluid}_{\bc}$ is smooth, with unique gradient $\nabla V^{\rm fluid}$ at any $\bc$.
\end{lemma}
The proof is in Appendix \ref{appendix: degeneracy}. Combining Lemma \ref{lem: smoothness of dual objective under our assumption} and Lemma \ref{lem: nondiscrete non kink}, Assumption \ref{def: unique-dual} always holds under Assumption \ref{assum: starting-from-zero} or Assumption \ref{assum: small-probability-starting-from-zero}, yet the non-degeneracy conditions (Assumptions \ref{def: dual stability} and \ref{def: dual non-degeneracy}) may fail (cf. Lemma \ref{lem: non essential non degeneracy}). Our main contribution essentially lies in proving that in such cases, {\sf CE} achieves uniformly near-optimal regret even if these non-degeneracy conditions fail (cf. Theorem \ref{thm: regret of CE}). Along the line of this discussion, we may provide some geometric intuition of our performance guarantee. Observe in Figure \ref{Fig: nondiscrete-dege-f} that the smoothness of the function $f_{\bd}(\blambda)$ ensures the dual price does not oscillate drastically when $\bd$ is perturbed around the degenerate point $d^1$. This inherent stability, even as the resource constraint alternates between binding and non-binding, allows {\sf CE} to maintain good performance. This stands in sharp contrast to the discrete setting where across all distribution $F$, function $V^{\rm fluid}_{\bd}$ is non-smooth, and degeneracy ($\bd$ taking certain corner values) always comes with drastic instability phenomenon, leading to performance deterioration of {\sf CE}. 

Meanwhile, not all non-discrete distributions yield smooth dual objectives as in Lemma \ref{lem: smoothness of dual objective under our assumption}. We provide two such examples that have appeared in prior work. \cite{besbes2024dynamic} considered multisecretary instances where the reward distributions have gaps in their supports (cf. Example 3 in \cite{besbes2024dynamic}). \cite{jiang2022degeneracy} studied a semi-discrete NRM instances with finitely many request types, where conditional on any fixed type, the reward distribution is supported on an arbitrary interval with lower bounded density (cf. Assumption 1 \cite{jiang2022degeneracy}). Both examples have non-smooth $V^{\rm fluid}_{\bd}$. At those $\bd$ that corresponds to a kink on $V^{\rm fluid}_{\bd}$, Assumption \ref{def: unique-dual} no longer holds. {\sf CE} provably fails in these cases:
\begin{prop}\label{prop: necessity of dual-unique in multisec}
Consider a multisecretary instance with $F^r_{\ba} = F^r \triangleq {\sf Unif}[0, 1]\cup[2,3]$ and $\bb = \frac{1}{2} T$. Then $\textsc{Reg}_{\bb, T}({\sf CE}) \geq c_0 \sqrt{T}$ for some absolute constant $c_0 > 0$.
\end{prop}
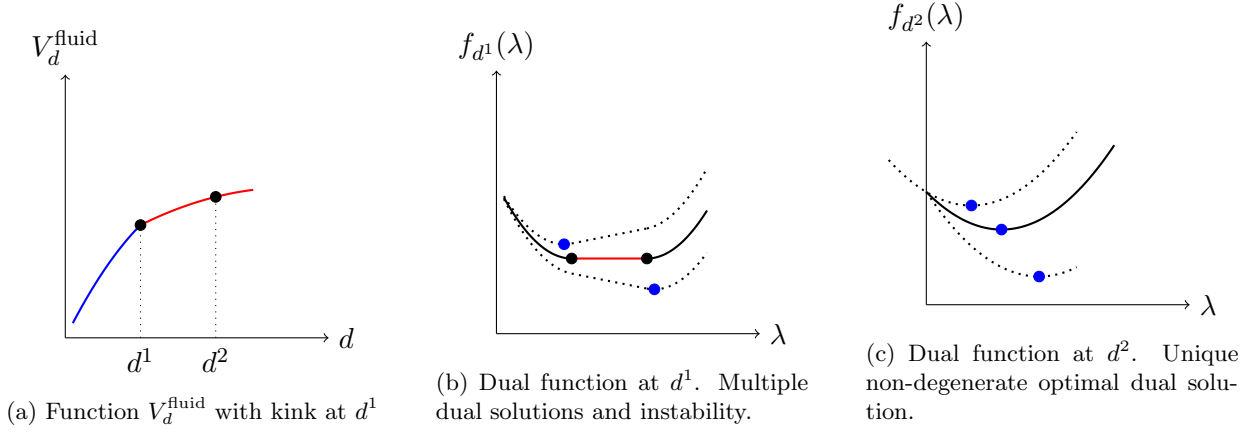
\begin{figure}[h!]
    \centering
        \begin{subfigure}[b]{0.3\textwidth}
        \centering
        \begin{tikzpicture}[scale=1]
            \draw[->] (0,0) -- (3.5,0) node[right] {$d$};
            \draw[->] (0,0) -- (0,3.5) node[above] {$V^{\rm fluid}_d$};
            
            \draw[blue, thick] plot[domain=0.1:1,samples=100] (\x,{-(\x - 2)^2/2 + 2});
            \draw[red, thick] plot[domain=1:2.5,samples=100] (\x,{-(\x - 3)^2/8 + 2});

            \filldraw[black] (1,1.5) circle (2pt);
            \filldraw[black] (2,2 - 0.125) circle (2pt);
            
            \draw[dotted] (1,0) -- (1,1.5);
            \draw[dotted] (2,0) -- (2,2 - 1/8);
            \node[below] at (1,0) {$d^1$};
            \node[below] at (2,0) {$d^2$};
        \end{tikzpicture}
        \caption{Function $V^{\rm fluid}_d$ with kink at $d^1$}
        \label{Fig: gap-dege-V}
    \end{subfigure}
    \hfill
    \begin{subfigure}[b]{0.3\textwidth}
        \centering
        \begin{tikzpicture}[scale=1]
            \draw[->] (0,0) -- (3.5,0) node[right] {$\lambda$};
            \draw[->] (0,0) -- (0,3.5) node[above] {$f_{d^1}(\lambda)$};

            \draw[thick, domain=0.1:1,samples=100] plot (\x,{(\x- 1)^2 +  1});
            \draw[thick,red, domain=1:2,samples=100] plot (\x,{1});
            \draw[thick,
            domain=2:2.8,samples=100] plot (\x,{(\x- 2)^2 + 1});

            \draw[thick, dotted, domain=0.1:1,samples=100] plot (\x,{(\x- 1)^2 +  1 + 0.2 *\x});
            \draw[thick,dotted, domain=1:2,samples=100] plot (\x,{1+ 0.2 *\x});
            \draw[thick, dotted,
            domain=2:2.8,samples=100] plot (\x,{(\x- 2)^2 +  1 + 0.2 *\x});

               \draw[thick, dotted, domain=0.1:1,samples=100] plot (\x,{(\x- 1)^2 +  1 - 0.2 *\x});
            \draw[thick,dotted, domain=1:2,samples=100] plot (\x,{1- 0.2 *\x});
            \draw[thick, dotted,
            domain=2:2.8,samples=100] plot (\x,{(\x- 2)^2 +  1  - 0.2 *\x});
            \filldraw[blue] (0.9 , .01 + 1 + 0.18) circle (2pt);  
            \filldraw[blue] (2.1, .01 + 1 -0.42) circle (2pt);
            \filldraw[black] (1,1) circle (2pt); 
            \filldraw[black] (2, 1) circle (2pt);
        \end{tikzpicture}
        \caption{Dual function at $d^1$. Multiple dual solutions and instability.}
        \label{Fig: gap-dege-f}
    \end{subfigure}
    \hfill
        \begin{subfigure}[b]{0.3\textwidth}
        \centering
        \begin{tikzpicture}[scale=1]
            \draw[->] (0,0) -- (3.5,0) node[right] {$\lambda$};
            \draw[->] (0,0) -- (0,3.5) node[above] {$f_{d^2}(\lambda)$};

            \draw[thick, domain=0.0:2.5,samples=100] plot (\x,{(\x)^2/2 + 1.5 - \x});
            \draw[thick, dotted, domain=-0.5:2,samples=100] plot (\x,{(\x)^2/2 + 1.5 + 0.4 * \x - \x});
            \draw[thick, dotted, domain=0:2,samples=100] plot (\x,{(\x)^2/2 + 1.5 - 0.5 * \x - \x});

            \filldraw[blue] (1,1) circle (2pt);  
            \filldraw[blue] (0.6,0.36/2 + 1.5 - 0.36) circle (2pt);  
            \filldraw[blue] (1.5 ,1.5^2/2 + 1.5 - 1.5^2) circle (2pt);  
        \end{tikzpicture}
        \caption{Dual function at $d^2$. Unique non-degenerate optimal dual solution.}
        \label{Fig: gap-nondege-f}
    \end{subfigure}
    \hfill
    \caption{Function $V^{\rm fluid}$ and dual function $f_{d}$ for multisecretary instances with gap.}
    \label{Fig: nondiscrete-gap}
\end{figure}

Proposition \ref{prop: necessity of dual-unique in multisec} highlights the significance of dual uniqueness (Assumption \ref{def: unique-dual}) on the performance of {\sf CE}. The failure of {\sf CE} in Proposition \ref{prop: necessity of dual-unique in multisec} shares similarities with how {\sf CE} suffers from degeneracy in discrete settings, as illustrated in Figure \ref{Fig: nondiscrete-gap}. For detailed proof, readers are referred to Proposition 1 of \cite{besbes2024dynamic} or Proposition 2 in \cite{bumpensanti2020re}. Novel algorithmic designs beyond {\sf CE} are required for such instances, such as the {\sf CwG} principle \citep{besbes2024dynamic} or the {\sf boundary-attracted algorithm} \citep{jiang2022degeneracy}, to achieve improved ($o(\sqrt{T})$) regret. In general, OLP instances violating Assumption \ref{def: unique-dual} can be significantly more complex than the examples mentioned, to the extent that even describing such distributions may be challenging. Developing a unified algorithm and analysis to systematically handle these instances is beyond the scope of the current work. However, we hope the methodology introduced in this paper provides a foundation for addressing such challenges in future research.

The above discussion allows us to justify one specific requirement in our Assumption \ref{assum: starting-from-zero}, that the conditional reward distributions are all supported on intervals starting from zero. In fact, without such a restriction, it is possible to construct simple OLP instances that violates Assumption \ref{def: unique-dual} for which {\sf CE} fails.
\begin{exmp}\label{exmp:violating unique-dual}
$\PP\left(\ba = 1 \right)= \PP\left(\ba = 4\right) = \frac{1}{2}$. $r|\ba \sim {\sf Unif}[1, 2]$ for both values of $\ba$. The initial inventory for the resource is $\bb = \bd T = \frac{1}{2} T$.    
\end{exmp}
The distribution in Example \ref{exmp:violating unique-dual} satisfies all the conditions in Assumption \ref{assum: starting-from-zero} with $\beta = 0$, except for the requirement \textit{(ii)}. A straightforward calculation yields multiple dual optimal solutions $\blambda^{\star} \in \left[\frac{1}{2}, 1\right]$ that violates Assumption \ref{def: unique-dual}, and leads to $\Omega(\sqrt{T})$ regret of {\sf CE}. Note that the additional regularity imposed on $F$ in Assumption \ref{assum: small-probability-starting-from-zero} \textit{(ii)} allows us to relax the condition that all conditional reward distribution should start from zero. However, we still require this condition to hold locally, to avoid tricky corner cases. We defer more discussion to Appendix \ref{appendix: second-order} (cf. Example \ref{exmp:violating 2-order}).

The previous discussion provides a fairly comprehensive picture regarding how fluid degeneracy affects the performance of {\sf CE}. Combining both the discrete and non-discrete perspectives, it seems more accurate to describe the structural challenge faced by {\sf CE} as the ``curse of dual non-uniqueness'' rather than the broader and less precise ``curse of degeneracy.''

\section{Proof Sketch}\label{sec:sketch}

We rely on the concentration analysis of the hindsight dual optimal solution in each time period $t$ to conduct our performance analysis. In period $t$, the primal hindsight relaxation takes the form of a multi-knapsack LP similar to Problem (\ref{program: offline primal}). We consider its dual, which can be thought of as the empirical version of Problem (\ref{program: to go fluid dual}). In particular, 
\begin{align}
    \blambda^{\star}_t &= \argmin_{\blambda \in \RR^{m}_{\geq 0}} \bb^{\top}_{t-1} \blambda + \sum_{j = t + 1}^T (r_j - \ba^{\top}_j \blambda)^+,\label{program: to-go-offline-dual-1}\\
    \bar \blambda^{\star}_t &= \argmin_{\blambda \in \RR^{m}_{\geq 0}} (\bb_{t - 1} - \ba_t)^{\top} \blambda + \sum_{j = t + 1}^T (r_j - \ba^{\top}_j \blambda)^+, \label{program: to-go-offline-dual-2}
\end{align}
where we note that $\dtt$ is constructed only when feasibility is satisfied at period $t$, namely, $\bb_{t - 1} \geq \ba_t$ element-wise.  Note that $\blambda^{\star}_t$ and ${\bar \blambda^{\star}_t}$ depend both on $\bb_{t-1}$ (and $\ba_t$)  and $\cI_t$. They are not non-anticipatory, and the introduction of them serves only for the purpose of analyzing the performance of algorithms. For notational simplicity, we write $\blambda^{\star}_t$ and ${\bar \blambda^{\star}_t}$ with the understanding that their dependence on $\bb_{t-1}$ ($\ba_t$) and $\cI_t$ is implicit but clear from the context.
\subsection{Regret decomposition}
\begin{lemma}[Regret Decomposition of {\sf CE}]\label{lem: ce-regret-decomp}
The expected regret of the {\sf CE} heuristic is bounded by
\begin{align}
    \textsc{reg}_T(\pi^{\sf CE}) \ &\leq \ \cC_0\log{T} + \sum_{t = 1}^T  \E\left[ \left(\ba^{\top}_t {\dtt} - r_t\right)\II_{\left\{\ba^{\top}_t \ft \leq r_t \leq \ba^{\top}_t {\dtt}\right\}}\right]\NNN\\
    &\quad\quad\quad\quad\quad\quad\quad\quad\quad\quad\quad +  \sum_{t = 1}^T  \E\left[ \left(r_t - \ba^{\top}_t {\dt} \right)\II_{\left\{\ba^{\top}_t {\dt} \leq r_t \leq \ba^{\top}_t \ft \right\}}\right],
\end{align}
where $\cC_0 = 2\left(1 + \frac{m\Ab}{\Al} \right)m {\bar r}$, and the expectation is taken over $\cI_0$.
\end{lemma}
Lemma \ref{lem: ce-regret-decomp} follows from a careful decomposition analysis similar to \cite{jiang2022degeneracy} within the \textit{compensated coupling} framework of \cite{vera2021bayesian}. Particularly, let $V^{\rm off}_{t} = V^{\rm off}_{t, \bb_t}(\cI_t)$ denote the optimal value of the multi-knapsack LP
\begin{align*}
    V^{{\rm off}}_{t} = \max &\quad  \sum_{j = t + 1}^T r_j x_j\\
    \mbox{s.t.} & \quad \sum_{j = t + 1}^T a_{i j} x_j \leq b_{i t}, \quad i = 1, \dots, m,
    \\& x_j \in [0, 1], \quad j = t + 1,\dots, T.
    \end{align*}
We consider the following decomposition of regret 
\begin{align}
    \textsc{reg}_T(\pi^{\sf CE}) &= \E\left[V^{\rm off}\right] - \E\left[V^{\pi^{\sf CE}}\right]\NNN\\
    &= \E\left[V^{\rm off}_0\right] - \E\left[\sum_{t = 1}^T r_t x^{\pi^{\sf CE}}_t\right] \NNN\\
    &= \E\left[\sum_{t = 1}^T \left(V^{\rm off}_{t-1} - V^{\rm off}_{t}\right)\right] - \E\left[\sum_{t = 1}^T r_t x^{\pi^{\sf CE}}_t\right]  \quad\quad\quad\quad(\textrm{since $V^{\rm off}_T = 0$})\NNN\\
    &= \sum_{t = 1}^T \E\left[ V^{\rm off}_{t-1} - V^{\rm off}_{t} - r_t x^{\pi^{\sf CE}}_t \right]. \label{eq: myopic regret}
\end{align}
Lemma~\ref{lem: ce-regret-decomp} follows from further carefully treating \eqref{eq: myopic regret}. In particular, we show that when $x^{\pi^{\sf CE}}_t = x^*_t$ for a hindsight optimal solution $\{x^*_j\}_{j = t}^T$, then $V^{\rm off}_{t-1} - V^{\rm off}_{t} - r_t x^{\pi^{\sf CE}}_t = 0$ and no regret will be incurred. Otherwise, we bound the myopic regret $V^{\rm off}_{t-1} - V^{\rm off}_{t} - r_t x^{\pi^{\sf CE}}_t$ in one of the three cases: \textit{(i)} $x^{\pi^{\sf CE}}_t \in (0, 1)$, \textit{(ii)} $x^{\pi^{\sf CE}}_t = 1, x^*_t = 0$, and  \textit{(iii)} $x^{\pi^{\sf CE}}_t = 0, x^*_t = 1$. We argue that \textit{(i)} occurs with $O(\frac{1}{T - t})$ probability, and leverage the optimality of $\dt$ and $\dtt$ to bound the myopic regret in \textit{(ii)} and \textit{(iii)}. We defer a detailed proof to \Cref{appendix: proof of decomposition lemma}.

\subsection{Concentration analysis}
Lemma~\ref{lem: ce-regret-decomp} implies that the key to bounding the regret under {\sf CE} is to control the deviation of $\ba_t^{\top} \dt$ (and $\ba_t^{\top} \dtt$) from $\ba_t^{\top} \ft$, In particular, the concentration of $\dt$ and $\dtt$. In the prior works, strong non-degeneracy conditions are typically imposed to regulate the concentration of $\dt$ and $\dtt$. Different from the standard approach, we avoid imposing regularity conditions directly on $\ft$. Rather, conditions on the problem primitives, i.e. Assumptions \ref{assum: starting-from-zero} and \ref{assum: small-probability-starting-from-zero} suffice to guarantee nice concentration of $\dt$ and $\dtt$ for us. 

\begin{lemma}\label{lem: concentration results of dual optimum}
Under either Assumption \ref{assum: starting-from-zero} or Assumption \ref{assum: small-probability-starting-from-zero}, we have
\begin{align*}
    & \E_{\ba \sim F_{\ba}}\bigg[\bigg(F_{\ba}(\ba^{\top}{\ft}) - F_{\ba}(\ba^{\top}\dt ) \bigg)\bigg(\ba^{\top}{\ft} - \ba^{\top}\dt\bigg)\bigg]
    \leq \  \cC_4\left(\frac{\log(T - t)}{T - t}\right)^{\frac{2 + \beta}{2 + 2 \beta}},\\
    & \E_{\ba \sim F_{\ba}}\bigg[\bigg(F_{\ba}(\ba^{\top}{\ft}) - F_{\ba}(\ba^{\top}\dtt ) \bigg)\bigg(\ba^{\top}{\ft} - \ba^{\top}\dtt\bigg)\bigg]    \leq \  \cC_4\left(\frac{\log(T - t)}{T - t}\right)^{\frac{2 + \beta}{2 + 2 \beta}},
\end{align*}
both with probability at least $1-\frac{9 C\log(T - t)}{(T - t)^2}$, for $T-t\geq e^{\frac{\tilde C_1}{32} + 1}$, where $\cC_4$ and $\tilde C_1$ are constants (whose forms can be found in Remark \ref{rem_constantsforms}) that depend on model primitives regarding $F$ but independent of $T - t$ and $\bb$, and $C$ is a universal constant (cf. Lemma \ref{lem:helper-concentration-geer}).
\end{lemma} 

The expectation on the LHS in \Cref{lem: concentration results of dual optimum} is taken with respect to $\ba$. Recall that $\dt$ and $\dtt$ are both random variables, therefore the LHS in \Cref{lem: concentration results of dual optimum} are also random variables. \Cref{lem: concentration results of dual optimum} establishes a high-probability concentration bound of $ \dt, \dtt$ in a specific form. It is not hard to observe that the LHS of \Cref{lem: concentration results of dual optimum} serve as upper bounds on the terms in the regret decomposition of \Cref{lem: ce-regret-decomp}, explaining why we set up bounds in such a particular form. 

The proof of Lemma~\ref{lem: concentration results of dual optimum} needs an intermediate lemma of a similar form, that is 
\begin{lemma}\label{lem:concentration-proof-part-1}
Denote by
\begin{align*}
    M_1 &\ = \ \sqrt{\E_{\ba \sim F^{\ba}}\bigg[\big(\ba^{\top}\ft - \ba^{\top}\dt\big)^2\max\bigg(\bar F^r_{\ba}(\ba^\top \dt), \bar F^r_{\ba}(\ba^\top \ft)\bigg)\bigg]}, \\
    \bar M_1 &\ = \ \sqrt{\E_{\ba \sim F^{\ba}}\bigg[\big(\ba^{\top}\ft - \ba^{\top}\dt\big)^2\max\bigg(\bar F^r_{\ba}(\ba^\top \dt), \bar F^r_{\ba}(\ba^\top \ft)\bigg)\bigg]}, 
\end{align*}
where $\bar F^r_{\ba}(\cdot) = 1 - F^r_{\ba}(\cdot)$. Then under either Assumption \ref{assum: starting-from-zero} or Assumption \ref{assum: small-probability-starting-from-zero},
\begin{align}
    & \E_{\ba \sim F^{\ba}}\bigg[\bigg(F^r_{\ba}(\ba^{\top}{\ft}) - F^r_{\ba}(\ba^{\top}\dt ) \bigg)\bigg(\ba^{\top}{\ft} - \ba^{\top}\dt\bigg)\bigg]\leq \tilde C_3\sqrt{\frac{\log(T-t)}{T-t}}\bigg(M_1 +\sqrt{\frac{\log(T-t)}{T-t}}\bigg), \label{eq: concentration_lemma_1} \\
    & \E_{\ba \sim F^{\ba}}\bigg[\bigg(F^r_{\ba}(\ba^{\top}{\ft}) - F^r_{\ba}(\ba^{\top}\dtt ) \bigg)\bigg(\ba^{\top}{\ft} - \ba^{\top}\dtt\bigg)\bigg] \leq  \tilde C_3\sqrt{\frac{\log(T-t)}{T-t}}\bigg(\bar M_1 +\sqrt{\frac{\log(T-t)}{T-t}}\bigg),\label{eq: concentration_lemma_2}
\end{align}
both with probability at least $1-\frac{9C \log(T - t)}{(T - t)^2}$ for any $T-t\geq e^{\frac{\tilde C_1}{32} + 1}$, where  $\tilde C_1, \tilde C_3$ depend on model primitives regarding $F$ but independent of $T - t$ and $\bb$, whose particular form can be found in Remark \ref{rem_constantsforms}.
\end{lemma}
Lemma~\ref{lem:concentration-proof-part-1} is different from Lemma~\ref{lem: concentration results of dual optimum} in that, the RHS of the bound in Lemma~\ref{lem:concentration-proof-part-1} contains the terms $M_1$ and $\bar M_1$. When both $\ba^{\top}{\ft}, \ba^{\top}\dt$ (or $\ba^{\top}\dtt$) belong to the support of $F^r_{\ba}$, we have $|\ba^{\top}{\ft} - \ba^{\top}\dt| \propto \left(|F^r_{\ba}(\ba^{\top}{\ft}) - F^r_{\ba}(\ba^{\top}{\ft})|\right)^{\frac{1}{1 + \beta}}$ under both Assumption \ref{assum: starting-from-zero} and Assumption \ref{assum: small-probability-starting-from-zero}. This further allows us to bound the terms $M_1$ (and $\bar M_1$) using functions of the LHS in the bound, namely  $\E_{\ba \sim F^{\ba}}\bigg[\bigg(F^r_{\ba}(\ba^{\top}{\ft}) - F^r_{\ba}(\ba^{\top}\dt ) \bigg)\bigg(\ba^{\top}{\ft} - \ba^{\top}\dt\bigg)\bigg]$. Solving the recursive inquality then leads to the desired Lemma~\ref{lem: concentration results of dual optimum}. We formalize this argument and provide a proof of Lemma~\ref{lem: concentration results of dual optimum} under Assumption~\ref{assum: starting-from-zero}.

\begin{proof}[Proof of Lemma \ref{lem: concentration results of dual optimum} under Assumption \ref{assum: starting-from-zero}]
WLOG we only provide the proof regarding $\dt$, as the proof for $\dtt$ is nearly identical. 

Consider the partition  $[\Al, \Au]^m = \mathcal{E}_1 \cup \mathcal{E}_2$, where $\mathcal{E}_1 = \{\ba \in [\Al, \Au]^m: F^r_{\ba}(\ba^{\top} \ft) = 1 \ \textrm{or} \ F^r_{\ba}(\ba^{\top} \dt) = 1\},$ and $\mathcal{E}_2 = \{\ba \in [\Al, \Au]^m: F^r_{\ba}(\ba^{\top} \dt), F^r_{\ba}(\ba^{\top} \ft) < 1 \}.$
\\

\textbf{Case 1. $\mathcal{E}_1$ happens.\ } In this case, $F^r_{\ba}(\ba^{\top} \ft) = 1$ or $F^r_{\ba}(\ba^{\top} \dt) = 1$. WLOG we assume $F^r_{\ba}(\ba^{\top} \dt) = 1,$ then 
\begin{align*}
    & \big(\ba^{\top}\ft - \ba^{\top}\dt\big)^2\max\bigg(\bar F^r_{\ba}(\ba^\top \dt), \bar F^r_{\ba}(\ba^\top \ft)\bigg)  \nonumber\\
    & \ = \ \big(\ba^{\top}\ft - \ba^{\top}\dt\big)^2 \bar F^r_{\ba}(\ba^\top \ft)\nonumber\\
    & \ \leq \ \max\left(\ba^{\top}\ft,\ba^{\top}\dt\right) \left|\ba^{\top}\ft - \ba^{\top}\dt\right| \bar F^r_{\ba}(\ba^\top \ft)\nonumber\\
    & \ \leq \ m\frac{{\bar r } \Au}{\Al} \left|\ba^{\top}\ft - \ba^{\top}\dt\right| \bar F^r_{\ba}(\ba^\top \ft)\ \ \ \ \ \ \ \ \ \ \textrm{(by Corollary \ref{coro: helper-bounded-norm-lambda})}\nonumber\\
    &\ = \ m \frac{{\bar r} \Au}{\Al} \bigg(F^r_{\ba}(\ba^{\top}\ft) - F^r_{\ba}(\ba^{\top}\dt ) \bigg)\bigg(\ba^{\top}\ft - \ba^{\top}\dt\bigg),
\end{align*}
where the last equality follows from $F^r_{\ba}\left(\ba^{\top} \dt\right) = 1$ and that $\ba^{\top}\ft - \ba^{\top}\dt$ and $F^r_{\ba}(\ba^{\top}\ft) - F^r_{\ba}(\ba^{\top}\dt )$ share the same sign. If $F^r_{\ba}(\ba^{\top} \ft) = 1,$ the argument is nearly identical. We further bound the above RHS as follows
\begin{align}
    \label{eq_pfconrescase1}
    & m \frac{{\bar r} \Au}{\Al} \bigg(F^r_{\ba}(\ba^{\top}\ft) - F^r_{\ba}(\ba^{\top}\dt ) \bigg)\bigg(\ba^{\top}\ft - \ba^{\top}\dt\bigg)\NNN\\
    & \ \leq\ \left(2 m \frac{{\bar r} \Au}{\Al}\right)^{\frac{2 + 2\beta }{2 + \beta}}\left(\big(F^r_{\ba}(\ba^{\top}\ft) - F^r_{\ba}(\ba^{\top}\dt ) \big)\big(\ba^{\top}\ft - \ba^{\top}\dt\big)\right)^{\frac{2}{2 + \beta}}, 
\end{align}\\
where we use the fact that $\frac{2}{2 + \beta} < 1$, $|F^r_{\ba}(\ba^{\top}\ft) - F^r_{\ba}(\ba^{\top}\dt )| \leq 1$ and $|\ba^{\top}\ft - \ba^{\top}\dt| < 2 m \frac{{\bar r} \Au}{\Al}$ by Corollary \ref{coro: helper-bounded-norm-lambda}.

\textbf{Case 2. $\mathcal{E}_2$ happens.\ } In this case, both $F^r_{\ba}(\ba^{\top} \ft)$ and $F^r_{\ba}(\ba^{\top} \dt)$ are strictly bounded away from $1.$ WLOG, assume $F^r_{\ba}(\ba^{\top} \ft) \geq F^r_{\ba}(\ba^{\top} \dt)$, while the other case $F^r_{\ba}(\ba^{\top} \ft) < F^r_{\ba}(\ba^{\top} \dt)$ is almost identical. Hence, by the H\"older condition in Assumption \ref{assum: starting-from-zero}, we have
\begin{align*}
    c_\beta(\ba^{\top} \ft - \ba^{\top}\dt)^{ 1 + \beta}\leq \   F^r_{\ba}(\ba^{\top} \ft) - F^r_{\ba}(\ba^{\top} \dt),
\end{align*}
which leads to
\begin{align}\label{eq_pfconrescase2}
    &\big(\ba^{\top}\ft - \ba^{\top}\dt\big)^2\max\bigg(\bar F^r_{\ba}(\ba^\top \dt), \bar F^r_{\ba}(\ba^\top \ft)\bigg) \NNN\\
    &\leq \big(\ba^{\top}\ft - \ba^{\top}\dt\big)^2 \NNN\\
    &= c_\beta^{-\frac{2}{2 + \beta}}\big(\ba^{\top}\ft - \ba^{\top}\dt\big)^{\frac{2}{2 + \beta}} c_\beta^{\frac{2}{2 + \beta}}\big(\ba^{\top}\ft - \ba^{\top}\dt\big)^{\frac{2 + 2 \beta}{2 + \beta}},\NNN\\
    &\leq c_\beta^{-\frac{2}{2 + \beta}}\big(\ba^{\top}\ft - \ba^{\top}\dt\big)^{\frac{2}{2 + \beta}} \left(F^r_{\ba}(\ba^{\top} \ft) - F^r_{\ba}(\ba^{\top} \dt)\right)^{\frac{2}{2 + \beta}}.
\end{align}
Combining \eqref{eq_pfconrescase1} with \eqref{eq_pfconrescase2}, we conclude that 
\begin{align}\label{eq:finalM1}
    M_1 &\ = \ \sqrt{\E_{\ba \sim F^{\ba}}\bigg[\big(\ba^{\top}\ft - \ba^{\top}\dt\big)^2\max\bigg(\bar F^r_{\ba}(\ba^\top \dt), \bar F^r_{\ba}(\ba^\top \ft)\bigg)\bigg]} \nonumber\\
    & \ = \ \sqrt{\E_{\ba \sim F^{\ba}}\bigg[\big(\ba^{\top}\ft - \ba^{\top}\dt\big)^2\max\bigg(\bar F^r_{\ba}(\ba^\top \dt), \bar F^r_{\ba}(\ba^\top \ft)\bigg)\left(I_{\cE_1} + I_{\cE_2}\right)\bigg]}\nonumber\\
&\ \leq \ \max\left(\left(2 m \frac{{\bar r} \Au}{\Al}\right)^{\frac{1 + \beta }{2 + \beta}}, c_\beta^{-\frac{1}{2 + \beta}}\right)\sqrt{\E_{\ba \sim F^{\ba}}\bigg[\big(\ba^{\top}\ft - \ba^{\top}\dt\big)^{\frac{2}{2 + \beta}} \left(F^r_{\ba}(\ba^{\top} \ft) - F^r_{\ba}(\ba^{\top} \dt)\right)^{\frac{2}{2 + \beta}}\bigg]}\nonumber\\
&\ \leq \ \max\left(\left(2 m \frac{{\bar r} \Au}{\Al}\right)^{\frac{1 + \beta }{2 + \beta}}, c_\beta^{-\frac{1}{2 + \beta}}\right)\left(\E_{\ba \sim F^{\ba}}\bigg[\big(\ba^{\top}\ft - \ba^{\top}\dt\big) \left(F^r_{\ba}(\ba^{\top} \ft) - F^r_{\ba}(\ba^{\top} \dt)\right)\bigg]\right)^{\frac{1}{2 + \beta}}\nonumber\\
& \quad\quad\quad\quad (\textrm{by Jensen's inequality since $x^{\frac{2}{2 + \beta}}$ is concave} ).
\end{align}
Plugging \eqref{eq:finalM1} into Lemma \ref{lem:concentration-proof-part-1}, 
\begin{align*}
    & \E_{\ba \sim F^{\ba}}\bigg[\bigg(F^r_{\ba}(\ba^{\top}{\ft}) - F^r_{\ba}(\ba^{\top}\dt ) \bigg)\bigg(\ba^{\top}{\ft} - \ba^{\top}\dt\bigg)\bigg]\\
    \ \leq\  & \tilde C_3\sqrt{\frac{\log(T-t)}{T-t}}\bigg(M_1+\sqrt{\frac{\log(T-t)}{T-t}}\bigg)\\
    \ \leq\  &\tilde C_3\frac{\log(T-t)}{T-t} + \tilde C_3 \max\left(\left(2 m \frac{{\bar r} \Au}{\Al}\right)^{\frac{1 + \beta }{2 + \beta}}, c_\beta^{-\frac{1}{2 + \beta}}\right) \sqrt{\frac{\log(T-t)}{T-t}}\\ &\qquad\qquad\qquad\qquad \times \left(\E_{\ba \sim F^{\ba}}\bigg[\big(\ba^{\top}\ft - \ba^{\top}\dt\big) \left(F^r_{\ba}(\ba^{\top} \ft) - F^r_{\ba}(\ba^{\top} \dt)\right)\bigg]\right)^{\frac{1}{2 + \beta}}\nonumber,
\end{align*}
which, through straightforward algebra, implies that
\begin{align*}
    & \E_{\ba \sim F^{\ba}}\bigg[\bigg(F^r_{\ba}(\ba^{\top}{\ft}) - F^r_{\ba}(\ba^{\top}\dt ) \bigg)\bigg(\ba^{\top}{\ft} - \ba^{\top}\dt\bigg)\bigg]\nonumber\\
    \leq & \  \tilde C_3\left(1 + \max\left(\left(2 m \frac{{\bar r} \Au}{\Al}\right)^{\frac{1 + \beta }{2 + \beta}}, c_\beta^{-\frac{1}{2 + \beta}}\right)\right)^{\frac{2 + \beta}{1 + \beta}} \left(\frac{\log(T - t)}{T - t}\right)^{\frac{2 + \beta}{2 + 2 \beta}},
    \triangleq \tilde C_4 \left(\frac{\log(T - t)}{T - t}\right)^{\frac{2 + \beta}{2 + 2 \beta}},
\end{align*}
with probability at least  $1 - \frac{9 C \log (T - t)}{(T - t)^2}$ for any $T - t \geq e^{\frac{\tilde C_1}{32} + 1}$. Note that we have chosen
\[
 \tilde C_4 \triangleq \tilde C_3\left(1 + \max\left(\left(2 m \frac{{\bar r} \Au}{\Al}\right)^{\frac{1 + \beta }{2 + \beta}}, c_\beta^{-\frac{1}{2 + \beta}}\right)\right)^{\frac{2 + \beta}{1 + \beta}},
\]
thus concluding the proof of the lemma under Assumption \ref{assum: starting-from-zero}. 
\end{proof}

The proof under Assumption \ref{assum: small-probability-starting-from-zero} requires more effort but follows essentially the same high level idea, and we defer to Appendix \ref{appendix: proof of concentration lemma}.

\subsubsection{Proof sketch of Lemma~\ref{lem:concentration-proof-part-1} } 
Let
\begin{align*}
    h_{t,\bb}(\blambda, \ba, r) &= \frac{1}{T-t}\bb^\top \blambda + \left(r - \ba^\top \blambda\right)^+,\\
    \phi_{t,\bb}(\blambda, \ba, r) &= \frac{\partial h_{t,\bb}(\blambda, \ba)}{\partial \blambda} = \frac{1}{T-t}\bb - \ba\II_{\left\{r > \ba^\top \blambda\right\}}.
\end{align*}
Furthermore, let $\Omega \triangleq [0, \frac{\bar r}{\underline{A}}]^m.$

The proof of Lemma~\ref{lem:concentration-proof-part-1} relies on the following concentrations bounds. 

\begin{lemma}\label{lemma: concentration_gradient}
With probability at least $1 - \frac{3C\log(T - t)}{(T - t)^2}$,
\begin{align*}
    &\left|\E_{(\ba,r)\sim F}\phi_{t,\bb}(\lama, \ba, r)^\top(\lamb -\lama) - \frac{1}{T-t}\sum_{j=t+1}^T \phi_{t,\bb}(\lama, \ba_j, r_j)^\top(\lamb -\lama)\right|\nonumber\\
    \leq & \tilde C\sqrt{\frac{\log(T - t)}{T - t}}\left(\sqrt{\E_{\ba\sim F^{\ba}}\left\{(\ba^\top(\lamb -\lama))^2(1-F^r_{\ba}(\ba^\top \lama))\right\}}+\sqrt{\frac{\log(T - t)}{T - t}}\right)
\end{align*}
holds for any $\bb \in \RR^m_{\geq 0}$, $\lama, \lamb \in \Omega$, as long as $T - t \geq e^{\frac{\tilde C^2}{32} + 1}$, where $\tilde C_1$ is a constant independent of  $T-t$, $\bb$ and $\blambda_1, \blambda_2$, whose specific form can be found in Remark~\ref{rem_constantsforms}, and $C$ is a universal constant. (cf. Lemma~\ref{lem:helper-concentration-geer})
and $C$ is a universal constant.
\end{lemma}

\begin{lemma}\label{lemma:concentration_int}
Suppose $T - t \geq e^{\frac{\tilde C_1^2}{32} + 1}$.
We have with probability at least $1 - \frac{3C\log(T - t)}{(T - t)^2}$, for all $\lama, \lamb \in \Omega$,
\begin{align*}
    & \left|\E_{(\ba,r)\sim F}\int_{\ba^\top\lama}^{\ba^\top\lamb}(\II_{\left\{r > u\right\}} - \II_{\left\{r > \ba^\top  \lamb\right\}}){\rm d}u - \frac{1}{T-t}\sum_{j=t+1}^T \int_{\ba_j^\top\lama}^{\ba_j^\top\lamb}(\II_{\left\{r_j > u\right\}} - \II_{\left\{r_j > \ba^\top \lamb\right\}}){\rm d}u\right|\nonumber\\
    \leq & \tilde C_1 \sqrt{\frac{\log(T - t)}{T - t}}\left(\left(\E_{(\ba,r)\sim F}\left(\int_{\ba^\top\lama}^{\ba^\top\lamb}(\II_{\left\{r > u\right\}} - \II_{\left\{r > \ba^\top  \lamb\right\}}){\rm d}u\right)^2\right)^{1/2}+\sqrt{\frac{\log(T - t)}{T - t}}\right),
\end{align*}
where $\tilde C_1$ is a constant independent of  $T-t$, $\bb$ and $\blambda_1, \blambda_2$, whose specific form can be found in Remark~\ref{rem_constantsforms}, and $C$ is a universal constant. (cf. Lemma~\ref{lem:helper-concentration-geer})
\end{lemma}

Both Lemma~\ref{lemma: concentration_gradient} and Lemma~\ref{lemma:concentration_int} are proved via applying the powerful peeling device as well as classic concentration results in the empirical process theory (cf.  Lemma~\ref{lem:helper-concentration-geer} in \Cref{appendix:concentration-helper}). These techniques are key to the removal of non-degeneracy conditions. We defer a detailed proof of Lemma~\ref{lemma: concentration_gradient} and Lemma~\ref{lemma:concentration_int}  to \Cref{appendix: proof of concentration helper}. Lemma~\ref{lem:concentration-proof-part-1} follows from applying these two concentration bounds at the special values $\dt, \dtt$ and $\ft$ which are the optimal solutions to three closely related problems. We defer a detailed proof of Lemma~\ref{lem:concentration-proof-part-1} to \Cref{appendix: proof of concentration part 1}.



\subsection{Proof of Theorem~\ref{thm: regret of CE}}
With the help of Lemma \ref{lem: ce-regret-decomp} and Lemma \ref{lem: concentration results of dual optimum}, the proof of Theorem \ref{thm: regret of CE} is straightforward. 
\begin{proof}[Proof of Theorem \ref{thm: regret of CE}]
In fact, by Lemma \ref{lem: ce-regret-decomp}, 
\begin{align}
    \textsc{Reg}_{\bb, T}(\pi^{\sf CE}) \ &\leq \ \cC_0\log{T} + \sum_{t = 1}^T  \E\left[ \left(\ba^{\top}_t {\dtt} - r_t\right)\II_{\left\{\ba^{\top}_t \ft \leq r_t \leq \ba^{\top}_t {\dtt}\right\}}\right]\NNN\\
    &\quad\quad\quad\quad\quad +  \sum_{t = 1}^T  \E\left[ \left(r_t - \ba^{\top}_t {\dt} \right)\II_{\left\{\ba^{\top}_t {\dt} \leq r_t \leq \ba^{\top}_t \ft \right\}}\right]\nonumber\NNN\\
    &\leq \ \cC_0\log{T} + \sum_{t = 1}^T  \E\left[ \left(\ba^{\top}_t {\dtt} - \ba^{\top}_t \ft\right)\II_{\left\{\ba^{\top}_t \ft \leq r_t \leq \ba^{\top}_t {\dtt}\right\}}\right]\NNN\\
    &\quad\quad\quad\quad\quad+  \sum_{t = 1}^T  \E\left[ \left(\ba^{\top}_t \ft - \ba^{\top}_t {\dt} \right)\II_{\left\{\ba^{\top}_t {\dt} \leq r_t \leq \ba^{\top}_t \ft \right\}}\right]\nonumber.
    \end{align}
    Observe that the above RHS takes the form of the LHS in Lemma \ref{lem: concentration results of dual optimum}, namely, 
    \begin{align}\label{eq:reg-bound-ce}
    & \cC_0\log{T} + \sum_{t = 1}^T  \E\left[ \left(\ba^{\top}_t {\dtt} - \ba^{\top}_t \ft\right)\II_{\left\{\ba^{\top}_t \ft \leq r_t \leq \ba^{\top}_t {\dtt}\right\}}\right] + \sum_{t = 1}^T  \E\left[ \left(\ba^{\top}_t \ft - \ba^{\top}_t {\dt} \right)\II_{\left\{\ba^{\top}_t {\dt} \leq r_t \leq \ba^{\top}_t \ft \right\}}\right]\nonumber\\
    = & \ \cC_0\log{T} + \sum_{t = 1}^T  \E_{\ba \sim F^{\ba}}\bigg[\bigg(F^r_{\ba}(\ba^{\top}{\ft}) - F^r_{\ba}(\ba^{\top}\dt ) \bigg)\bigg(\ba^{\top}{\ft} - \ba^{\top}\dt\bigg)\bigg]\NNN\\
    &\quad\quad\quad\quad\quad+  \sum_{t = 1}^T  \E_{\ba \sim F^{\ba}}\bigg[\bigg(F^r_{\ba}(\ba^{\top}{\ft}) - F^r_{\ba}(\ba^{\top}\dtt ) \bigg)\bigg(\ba^{\top}{\ft} - \ba^{\top}\dtt\bigg)\bigg].
\end{align}
Note that
\begin{align}\label{eq:reg-bound-ce-2nd}
    & \sum_{t = 1}^T  \E_{\ba \sim F^{\ba}}\bigg[\bigg(F^r_{\ba}(\ba^{\top}{\ft}) - F^r_{\ba}(\ba^{\top}\dt ) \bigg)\bigg(\ba^{\top}{\ft} - \ba^{\top}\dt\bigg)\bigg] \nonumber\\
    \leq & \sum_{t = 1}^{(T - \cC_5)_+}  \E_{\ba \sim F^{\ba}}\bigg[\bigg(F^r_{\ba}(\ba^{\top}{\ft}) - F^r_{\ba}(\ba^{\top}\dt ) \bigg)\bigg(\ba^{\top}{\ft} - \ba^{\top}\dt\bigg)\II_{\cA_t}\bigg]\nonumber\\
    & + \sum_{t = 1}^{(T - \cC_5)_+}  \E_{\ba \sim F^{\ba}}\bigg[\bigg(F^r_{\ba}(\ba^{\top}{\ft}) - F^r_{\ba}(\ba^{\top}\dt ) \bigg)\bigg(\ba^{\top}{\ft} - \ba^{\top}\dt\bigg)\II_{\cA_t}^c\bigg] \nonumber\\
    & + \sum_{t = (T - \cC_5)_+ + 1}^T \E_{\ba \sim F^{\ba}}\bigg[\bigg(F^r_{\ba}(\ba^{\top}{\ft}) - F^r_{\ba}(\ba^{\top}\dt ) \bigg)\bigg(\ba^{\top}{\ft} - \ba^{\top}\dt\bigg)\II_{\left\{T>\cC_5\right\}}\bigg]\nonumber\\
    \leq & \tilde \cC_4\sum_{t = 1}^{(T - \cC_5)_+} \left(\frac{\log(T - t)}{T - t}\right)^{\frac{2 + \beta}{2 + 2 \beta}} + \sum_{t = 1}^{(T - \cC_5)_+}\PP(\cA_t^c) + \cC_5 \frac{\Au}{\Al}m {\bar r}\nonumber\\
    \leq & \tilde \cC_4\sum_{t = 1}^{(T - \cC_5)_+} \left(\frac{\log(T - t)}{T - t}\right)^{\frac{2 + \beta}{2 + 2 \beta}} + \sum_{t = 1}^{(T - \cC_5)_+} \frac{9 C\log(T - t) \frac{\Au}{\Al}m {\bar r}}{(T - t)^2} + \cC_5 \frac{\Au}{\Al}m {\bar r},
\end{align}
where $\cC_5 =e^{\frac{\tilde C_1}{32} + 1}$, and $\cA_t$ denotes the event that the first inequality in Lemma \ref{lem: concentration results of dual optimum} occurs. In \eqref{eq:reg-bound-ce-2nd}, the first term is by the high probability concentration bound of Lemma \ref{lem: concentration results of dual optimum}, the second term is by the small probability that $\cA_t^c$ occurs, and the last term is by $\|\dt\|, \|\ft\|$ are both upper bounded.
Similarly, 
\begin{align}\label{eq:reg-bound-ce-3rd}
    & \sum_{t = 1}^T  \E_{\ba \sim F^{\ba}}\bigg[\bigg(F^r_{\ba}(\ba^{\top}{\ft}) - F^r_{\ba}(\ba^{\top}\dtt ) \bigg)\bigg(\ba^{\top}{\ft} - \ba^{\top}\dtt\bigg)\bigg] \nonumber\\
    \leq & \tilde \cC_4\sum_{t = 1}^{(T - \cC_5)_+} \left(\frac{\log(T - t)}{T - t}\right)^{\frac{2 + \beta}{2 + 2 \beta}} + \sum_{t = 1}^{(T - \cC_5)_+} \frac{9 C\log(T - t) \frac{\Au}{\Al}m {\bar r}}{(T - t)^2} + \cC_5 \frac{\Au}{\Al}m {\bar r}.
\end{align}
Observe that 
\begin{align*}
    \int_1^T \left(\frac{\log t}{t}\right)^\alpha {\rm d } t  \leq (\log T)^{\alpha} \int_1^T t^{-\alpha} {\rm d } t \leq \begin{cases}
        (\log T)^2 \quad\quad & \alpha = 1, \\
         \frac{1}{1-\alpha}(\log T)^\alpha T^{1 - \alpha} \quad\quad & 0 
 < \alpha < 1. 
    \end{cases}
\end{align*}
Thus, by the basic inequality $\log x \leq x$, combining \eqref{eq:reg-bound-ce}, \eqref{eq:reg-bound-ce-2nd}, and \eqref{eq:reg-bound-ce-3rd} gives us
\begin{align*}
    \textsc{Reg}_{\bb, T}(\pi^{\sf CE}) \ \leq \begin{cases}
        \cC (\log T)^2 \quad\quad &\beta = 0,\\
        \tilde \cC T^{\frac{1}{2} - \frac{1}{2(1 + \beta)}}  (\log T)^{\frac{2 + \beta}{2 + 2 \beta}} \quad\quad &\beta > 0,
    \end{cases}
\end{align*}
where $\cC = \cC_0 + 4\tilde \cC_4 + 36C\frac{\Au}{\Al}m {\bar r} + 2 \cC_5 \frac{\Au}{\Al}m {\bar r}$, and $\tilde \cC = \cC_0 + \frac{8+8\beta}{\beta}\tilde \cC_4 + 36C\frac{\Au}{\Al}m {\bar r} + 2 \cC_5 \frac{\Au}{\Al}m {\bar r}$. This finishes the proof of Theorem \ref{thm: regret of CE}, where the explicit form of constants $\cC$ and $\tilde \cC$ are provided in remark \ref{rem_constantsforms} in Appendix \ref{appendix: proof of ce regret thm}.
\end{proof}

\section{Concluding Remarks}\label{sec:conclusion}
In this work, we provided near-optimal regret guarantee of the classical {\sf CE} algorithm for the general OLP problem, under mild assumptions on the underlying request distributions. Our result extends the state-of-the-art understanding of {\sf CE}'s range of effectiveness, revealing that the commonly imposed non-degeneracy conditions are overly restrictive and not necessary for {\sf CE} to achieve low regret. We developed new algorithmic analytical techniques based on empirical processes theory, potentially applicable to a broader range of dynamic optimization problems with non-discrete distributions. Our work leave open many interesting questions.

One direction is the design and analysis of a unified algorithm that achieves near optimal performance for all OLP instances. Example \ref{exmp:violating unique-dual} demonstrates that our performance guarantee of {\sf CE} does not extend beyond the distribution classes that we identify, and novel algorithmic innovations are needed. Algorithms with uniformly near-optimal performance have previously been designed in the discrete setting (cf. \cite{vera2021bayesian}) and in the multisecretary problem (cf. \cite{besbes2024dynamic}). The simulation-based {\sf RAMS} algorithm proposed in \cite{besbes2024dynamic} is a potential candidate that attains uniform near-optimal performance across discrete and non-discrete settings, though the existing regret analysis of this algorithm still relies on some other reference algorithms (cf. Theorem 3 in \cite{besbes2024dynamic}).

Another direction is to relax several assumptions made in the current work. For instance, one may consider scenarios where arriving requests are not \emph{i.i.d.}, but with certain inter-dependent structure. Also, one may attempt to establish similar results in more general settings of dynamic resource allocation beyond OLP, such as the unified framework of $DRC^2$ proposed in \cite{balseiro2023survey}.

\section*{Acknowledgements}
We thank Dave Goldberg and Pengyu Qian for a number of valuable comments that improved the paper.



\newpage
\appendix


\section{Analysis of Example \ref{exmp: linear}}\label{appendix: example and assumption}
We analyze Example \ref{exmp: linear} in this section. Recall that in this example, reward and resource consumption follow a generalized linear model:
$r = g\left(\ba^{\top} \bz \right) + \epsilon$.

\begin{lemma}\label{lem: linear-example-satisfy-our-assumption}
Assumption \ref{assum: small-probability-starting-from-zero} \textit{(i)}, \textit{(ii)} and \textit{(iv)} hold for Example \ref{exmp: linear} with $\beta = 0$ and $\eta = 1$. Furthermore, a relaxed version of \textit{(iii)} holds: there exist $\ba_o \in {\rm supp}\left(F^{\ba}\right)$ and $r_0 > 0$ s.t. $l(\ba)\leq 0$ for all $\ba \in \cB(\ba_o, r_0)\cap {\rm supp}\left(F^{\ba}\right)$. 
\end{lemma}

\begin{proof}
That \textit{(i)} holds follows from the conditions on $F^{\ba}$ in Example \ref{exmp: linear}, and that both function $g$ and random variable $\epsilon$ is bounded. That \textit{(ii)} holds follows from the conditions on $F^{\ba}$ in Example \ref{exmp: linear}, and the Lipschitz continuity of $g$. \textit{(iv)} holds because $F^r_{\ba}(z_2) - F^r_{\ba}(z_1) = \PP\left(z_1 < r \leq z_2 | \ba \right) = \PP\left(z_1 - g\left(\ba^{\top}\bz \right) < \epsilon \leq z_2 - g\left(\ba^{\top}\bz \right) | \ba \right)$, which is bounded by $\underline{f_\epsilon}(z_2 - z_1)$ and ${\bar f_{\epsilon}}(z_2 - z_1)$, where $0 < \underline{f_\epsilon} \leq {\bar f_{\epsilon}} $ are the upper and lower bounds on the density of $\epsilon$, as is assumed in Example \ref{exmp: linear}. Finally, we argue that the relaxed version of \textit{(iii)} holds. Recall that $g\left((\ba')^{\top}\bz \right) \leq \cL - \eta.$ Thus the support of $r|\ba'$ is an interval $[-\eta, M]$, where $M < \infty$ since $g$ is Lipschitz, $\ba$ is bounded and $\eta$ is bounded. Namely, $l(\ba') = -\eta < 0.$ \textit{(iii)} thus holds with a ball centered at $\ba'$, where we use the Lipshcitzness of $l(\ba)$.
\end{proof}


The possibly negative rewards can be ignored without affecting the regret. Indeed, the hindsight optimal decision when facing a request with negative reward is to reject it. Thus $F$ can be WLOG restricted to the part with non-negative rewards. For that restricted probability distribution, Lemma \ref{lem: linear-example-satisfy-our-assumption} implies that Assumption \ref{assum: small-probability-starting-from-zero} holds.

\section{Explicit constants appearing in Theorem \ref{thm: regret of CE}}\label{appendix: proof of ce regret thm}
We specify the concrete constants appearing in the regret bound. 
\begin{rmk}\label{rem_constantsforms}
 The constants $\cC$ and $\tilde \cC$ in Theorem \ref{thm: regret of CE} have the specific form of 
    \begin{align*}
        \cC = & \cC_0 + 4 \cC_4 + 36C\frac{\Au}{\Al}m {\bar r} + 2 \cC_5 \frac{\Au}{\Al}m {\bar r}, \\
        \tilde \cC = & \cC_0 + \frac{8+8\beta}{\beta} \cC_4 + 36C\frac{\Au}{\Al}m {\bar r} + 2 \cC_5 \frac{\Au}{\Al}m {\bar r}.
    \end{align*}
    Recall from Lemma \ref{lem:helper-concentration-geer}, $C$ is an absolute constant. Furthermore,
    \begin{align*}
        \cC_0 = & 2\left(1 + \frac{m\Ab}{\Al} \right)m {\bar r}, \ \cC_5 = e^{\frac{\tilde C_1}{32} + 1},\\
        \tilde C_1 = & 2\left(\frac{4 m \bar r \Au}{\Al}+ 1\right) \left(10 C^2 \left(1 + 2\left(\frac{10 m}{\min\{\nu/2, 1\}}\right)^{1/2}\right)^2\left( 4c_{\nu}^{\frac{1}{2}}\frac{\bar r}{\Al}\left( 2 m \Au\right)^{\frac{\nu}{2}}+ 9\right)^2 + 2\right).
    \end{align*} 
 $\cC_4$ is specified from the proof of Lemma \ref{lem: concentration results of dual optimum}. Let 
 \[
 \tilde C_3 =  3\tilde C_1 + \frac{4 m \Au {\bar r}}{\Al} + \frac{2 m^2 {\bar A} {\bar r}}{\underline A}.
 \]
Then under Assumption \ref{assum: starting-from-zero}, 
    \[\cC_4 = \tilde C_3\left(1 + \max\left(\left(2 m \frac{{\bar r} \Au}{\Al}\right)^{\frac{1 + \beta }{2 + \beta}}, c_\beta^{-\frac{1}{2 + \beta}}\right)\right)^{\frac{2 + \beta}{1 + \beta}}.
    \]
   Under Assumption \ref{assum: small-probability-starting-from-zero}, 
    $\cC_4 = \max\left(c_1,c_2,c_3,c_4\right)$  where
\begin{align*}
    c_1 = & \tilde C_3 + \tilde C_3^2\Au \left(\min\left(c_\beta, 4^{- 1- \beta}c_\beta c_\nu^{\frac{1 + \beta}{\nu}} \left(2 \sqrt{m}\frac{\bar r \Au}{\Al}\right)^{-1 - \beta}\right) C_f^{\frac{2 + \beta}{2}}\left(\frac{r_1\wedge w}{6} \sin{(\theta)}\right)^{\frac{(2 + \beta)(m + 2)}{2}}\right)^{-1}  \max\left(1, \Au\right),\\
    c_2 = & \tilde C_3 + \tilde C_3\left(\frac{2\Au\sqrt{m} \bar  r}{\Al}\max\left(\tilde C_34^{1+\beta}c_\beta^{-1} c_\nu^{-\frac{1 + \beta}{\nu}}, \tilde C_3^2\frac{2\Au\sqrt{m} \bar  r}{\Al}4^{2+2\beta}c_\beta^{-2} c_\nu^{-\frac{2(1 + \beta)}{\nu}}\right)\right)^{1/2},\\
    c_3 = & \tilde C_3\left(1 + \max\left(\left(2 m \frac{{\bar r} \Au}{\Al}\right)^{\frac{1 + \beta }{2 + \beta}}, c_\beta^{-\frac{1}{2 + \beta}}\right)\right)^{\frac{2 + \beta}{1 + \beta}},\\
    c_4 = & \tilde C_3 + \tilde C_3^2\Au \left(c_\beta C_f^{\frac{2 + \beta}{2}}\left(\frac{r_0\wedge w}{6} \sin{(\theta)}\right)^{\frac{(2 + \beta)(m + 2)}{2}}\right)^{-1}  \max\left(1, \Au\right),
\end{align*}
and $\theta, w$ are the parameters specified by the uniform cone condition (cf. Lemma \ref{lem: ball-beats-orthogonality}), and
\begin{align*}
    C_f = l_f\frac{\pi^{m/2}}{\Gamma\left(\frac{m}{2}+1\right)}, \ \ r_1 = \frac{1}{8}\left(\left(c_L + \sqrt{m}\frac{\bar r}{\Al}\right)^{-1}\wedge 1\right) c_\nu^{\frac{1}{\nu}},
\end{align*}
with $\Gamma(\cdot)$ denoting the Gamma function.
\end{rmk}

\section{Proof in Section \ref{sec:implications}}\label{appendix: degeneracy}
\subsection{Proof of Lemma \ref{lem: non essential non degeneracy}}

The proof is divided into two parts. We first provide a proof of Lemma \ref{lem: smoothness of dual objective under our assumption}, then use this lemma to show the desired reuslts.
\begin{proof}[Proof of Lemma \ref{lem: smoothness of dual objective under our assumption}]
    Direct computation shows that
    \begin{align}\label{eq_lem_fbcsmooth}
        f_{\bc}(\blambda) = \bc^\top\blambda - \EE_{\ba \sim F^{\ba}} \ba^\top\blambda + \EE_{\ba \sim F^{\ba}}\int_0^{\ba^\top\blambda} F_{\ba}^r (v) {\rm d}v.
    \end{align}
    Both Assumptions \ref{assum: starting-from-zero} and Assumption \ref{assum: small-probability-starting-from-zero} imply that $F_{\ba}^r (v)$ is continuous with respect to $v$. Hence, by the Leibniz integral rule, $\nabla f_{\bc}(\blambda)$ exists. The statement about $V^{\rm fluid}_{\bc}$ is obvious by the expression \eqref{eq_lem_fbcsmooth} and the fact that $F_{\ba}^r (v)$ is continuous.
\end{proof}
With Lemma \ref{lem: smoothness of dual objective under our assumption}, we proceed with the proof of Lemma \ref{lem: non essential non degeneracy}.
\begin{proof}[Proof of Lemma \ref{lem: non essential non degeneracy}]
By Lemma \ref{lem: smoothness of dual objective under our assumption},  $\nabla f_{\bc}$ always exists for any $\bc$. We now show that for any nonzero
$\blambda^{o} \in \mathbb{R}^m_{\geq 0}$ satisfying $\lambda^{o}_i = 0$ for some $1 \leq i \leq m$, there exists $\bd$ such that $\blambda^{o}$ is the optimal solution to the dual problem $\min_{\blambda \in \RR^m_{\geq 0}} f_{\bd}(\blambda)$. Without loss of generality, let $i=1$, and set $\lambda_i^o > 0$ for $i=2,...,m$. Such a $\bd$ can be constructed via $\nabla f_{\bd}(\blambda^{o}) = \mathbf{0},$ or more precisely $\bd = \EE_{(\ba, r) \sim F}\left[\ba \mathbbm{1}\left(r > \ba^{\top} \blambda^{o}\right)\right]$. The resulting instance violates Assumption \ref{def: dual non-degeneracy} since $\lambda^{o}_1 = 0$ and $d_1 = \EE_{(\ba, r) \sim F}\left[a_1 \mathbbm{1}\left(r > \ba^{\top} \blambda^{o}\right)\right]$. Consider two perturbation $\bd' = (d_1+\epsilon,d_2,...,d_m)$ and $\bd'' = (d_1-\epsilon,d_2,...,d_m)$ of $\bd = (d_1,d_2,...,d_m)$, where $\epsilon>0$, and the optimal solutions for $\min_{\blambda \in \RR^m_{\geq 0}} f_{\bd'}(\blambda)$ and $\min_{\blambda \in \RR^m_{\geq 0}} f_{\bd''}(\blambda)$ are $\blambda'$ and $\blambda''$, respectively. It is clear that $\lambda_1'$ is zero while $d_1+\epsilon > \EE_{(\ba, r) \sim F}\left[a_1 \mathbbm{1}\left(r > \ba^{\top} \blambda'\right)\right]$, and $\lambda_1''>0$ while  $d_1-\epsilon = \EE_{(\ba, r) \sim F}\left[a_1 \mathbbm{1}\left(r > \ba^{\top} \blambda''\right)\right]$. Hence, it violates Assumption \ref{def: dual stability}. Notably, Example 2 in \cite{jiang2022degeneracy} is a specific instance within this degenerate class.
\end{proof}

\subsection{DLP and the proof of Lemma \ref{lem: non-degeneracy discrete} } 
We write down the dual fluid problem in its LP form:
\begin{align}\label{program: fluid dual-DLP}
    \min_{\blambda} &\quad \sum_{i = 1}^m d_i \lambda_i + \sum_{j = 1}^n \eta_j\\
    \mbox{s.t.} & \quad \sum_{i = 1}^m  a^j_i p^j \lambda_i + \eta_j \geq p^j r^j, \ \ j \in [1, n], \NNN\\
    & \quad \blambda, \bbbeta \geq 0. \nonumber
\end{align}

\begin{proof}[Proof of Lemma \ref{lem: non-degeneracy discrete}]
\textit{(i) Assumption \ref{def: unique-dual} $\Longleftrightarrow$ Assumption \ref{def: dual non-degeneracy}.} Since Assumption \ref{def: dual non-degeneracy} requires dual uniqueness, it suffices to show dual uniqueness implies Assumption \ref{def: dual non-degeneracy}, namely, strict complementary slackness. By standard LP theory (cf. Exercise 4.20 in \cite{bertsimas1997introduction}), there exists a pair of primal and dual optimal solution that the strict complementary slackness condition is satisfied. Premised on $\bx^{\star}$ is the unique solution to DLP (\ref{program: fluid primal-DLP}), both primal and dual have unique optimal solution and therefore they must satisfy strict complementary slackness.

\textit{(ii) Assumption \ref{def: unique-dual} $\Longleftrightarrow$ Assumption \ref{def: DLP non-degeneracy}.} We first show Assumption \ref{def: unique-dual} $\Longleftarrow$ Assumption \ref{def: DLP non-degeneracy}. Premised on primal uniqueness and under Assumption \ref{def: DLP non-degeneracy}, we have $\bx^{\star}$ is a unique, non-degenerate optimal solution to the DLP (\ref{program: fluid primal-DLP}). Standard LP theory thus implies the dual problem (\ref{program: fluid dual-DLP}) also has a non-degenerate and unique optimal solution $(\blambda^{\star}, \bbbeta^{\star})$, or equivalently, a unique solution $\blambda^{\star}$ to problem (\ref{eq: dual-formulation}) and Assumption \ref{def: unique-dual} is true. We now in turn prove Assumption \ref{def: unique-dual} $\Longrightarrow$ Assumption \ref{def: DLP non-degeneracy}. We prove by contradiction. Assume $\bx^{\star}$ is degenerate. Then by standard LP theory (cf. Theorem 4.5 in \cite{sierksma2001linear}), since both DLP (\ref{program: fluid primal-DLP}) and the dual problem (\ref{program: fluid dual-DLP}) are in standard (inequality) form, the dual must have multiple optimal solutions. Since $\bbbeta^{\star}$ is uniquely determined by $\eta_j = p^j\left(r^j - \sum_{i = 1}^m  a^j_i \lambda_i \right)^+$ for $j = 1, \dots, n$, the non-uniqueness of the dual solution thus must imply the non-uniqueness of optimal $\blambda^{\star}$ to the problem (\ref{eq: dual-formulation}), completing the proof.

\textit{(iii) Assumption \ref{def: DLP non-degeneracy}
$\Longrightarrow$ Assumption \ref{def: dual stability LP}.} In this case, $\bx^{\star}$ is the unique and non-degenerate optimal solution to DLP (\ref{program: fluid primal-DLP}). Since the DLP has a bounded feasible region, $\bx^{\star}$ must be a unique, non-degenerate optimal basic feasible solution (BFS). The primal stability in this case is immediate.

\textit{(iv) Assumption \ref{def: dual stability LP}
$\Longrightarrow$ Assumption \ref{def: unique-dual}.} We argue by contraction. If the dual has multiple optimal solutions, then by standard LP theory (cf. Theorem 4.5 in \cite{sierksma2001linear}), the primal DLP must have a degenerate solution. Since we further assume $\bx^{\star}$ is unique, it must be a BFS since the DLP has a bounded feasible region. Thus, $\bx^{\star}$ is a unique, degenerate optimal BFS. There must be more than $n$ binding constraints at $\bx^{\star}$, and $n$ of them are linearly independent. For any neighborhood of $\bd$, there always exists $\bd'$ in the neighborhood such that the binding constraints at $\bx^{\star}$ cannot be binding simultaneously. Therefore, any optimal solution to the DLP with RHS $\bd'$ can not share the same set of binding resources as $\bx^{\star}$, violating Assumption \ref{def: dual stability LP}, thus completing the proof.

Combining the above completes the proof.
\end{proof}

\section{Second-Order Growth Conditions}\label{appendix: second-order}




In addition to non-degeneracy, another class of fluid regularity conditions frequently imposed in prior literature is \textit{second-order growth} (on dual objectives). A fixed OLP instance specified by $(F, \bb, T)$ corresponds to a fluid instance $(F, \bd)$.  The second-order growth conditions characterize the curvature of the function $f_{\bd}(\blambda)$, as is determined by $F$ and $\bd$. Noticeably they are specific to the non-discrete settings, since in the discrete case, $f_{\bd}(\blambda)$ is piecewise linear, and the notion of curvature is trivial. Second-order growth conditions in the literature can be broadly classified into two categories: local and global conditions. Building on Theorem \ref{thm: regret of CE}, we demonstrate that local conditions capture key structural features of the problem that fundamentally determine the achievable regret scaling of {\sf CE} (and, in fact, any algorithm). By contrast, global conditions, often introduced for technical analytical purposes, are typically overly restrictive and not reflective of the actual determinants of algorithmic performance.

\subsection{Local conditions}
Various forms of (local) second-order growth conditions exist in the literature. Here, we present a typical one. For ease of connecting to our own assumptions, we state it in a more general form that incorporates high-order growth. Fix a fluid instance $(F, \bd)$ such that $\bd \neq \bm{0}$. 
\begin{assumption}[Local higher-order growth condition]\label{def: 2-order-growth}
There exist a constant $\gamma \geq 0$ and a neighborhood $\mathcal{N}$ of the dual optimal solution $\blambda^{\star}$ to $\min_{\blambda \in \RR^{m}_{\geq 0}} f_{\bd}(\blambda)$, and positive constant $\underline{\alpha}$, such that for any $\blambda \in \mathcal{N}$,  it holds that 
\[
 f_{\bd}(\blambda) - f_{\bd}(\blambda^\star) - \nabla f_{\bd}(\blambda^{\star})^{\top}\left(\blambda  - \blambda^{\star} \right)
\geq \underline{\alpha} \left|\blambda  - \blambda^{\star} \right|^{2 + \gamma}.
\]   
Typically, $\gamma = 0$ corresponds to the standard second-order growth condition.
\end{assumption}
Assumption \ref{def: 2-order-growth}, with $\gamma = 0$, appears in \cite{balseiro2023survey} (cf. SC 7). We note that other notions of (local) second-order growth conditions exist. For instance, \cite{li2022online} imposes a local Lipschitz continuity condition on the conditional reward CDF (Assumption 2(b) in \cite{li2022online}), \cite{bray2024logarithmic} imposes a positive definiteness condition on the Hessian matrix of $f_{\bd}(\blambda)$ at $\blambda^{\star}$ (cf. Assumption 6), while \cite{balseiro2023survey} also introduces a lower downward quadratic condition on $V^{\rm fluid}_{\bd}$ (cf. Assumption 2.1). For a comprehensive overview of these assumptions and their interrelations, we refer readers to Appendix A of \cite{jiang2022degeneracy}. The existing regret guarantee of an OLP instance $(F, \bb, T)$ typically requires the corresponding fluid instance $(F, \bd)$ to satisfy both the second-order growth condition, \textbf{and} the non-degenerate condition (cf. Assumption \ref{def: dual stability} or \ref{def: dual non-degeneracy}).

 We state here an observation that Assumption \ref{def: 2-order-growth} implies Assumption \ref{def: unique-dual}. In particular,
\begin{prop}\label{prop:2-order imply unique dual}
    Under the second-order growth condition of Assumption \ref{def: 2-order-growth} the dual solution is unique. Namely, Assumption \ref{def: unique-dual} holds.
\end{prop}
\begin{proof}
    Suppose there are two dual solutions $\blambda_1$ and $\blambda_2$ with $\blambda_1\neq \blambda_2$. For any $\blambda_0 = \theta \blambda_1 + (1-\theta)\blambda_2$, by the convexity of $f_{\bd}(\blambda)$, we have 
    \begin{align*}
        f_{\bd}(\blambda_\theta)\leq \theta f_{\bd}(\blambda_1) + (1-\theta)f_{\bd}(\blambda_2),
    \end{align*}
    which implies $\blambda_\theta$ is also a dual solution for any $\theta \in [0,1]$. Hence, we can take $\blambda_\theta$ in the neighborhood of $\blambda_2$ stated in Assumption \ref{def: 2-order-growth} and $\blambda_\theta\neq \blambda_2$. 
    By Assumption \ref{def: 2-order-growth}, we have
    \begin{align*}
        f_{\bd}(\blambda_2) - f_{\bd}(\blambda_{\frac{\theta}{2}}) - \nabla f_{\bd}(\blambda_{\frac{\theta}{2}})^{\top}\left(\blambda_2 - \blambda_{\frac{\theta}{2}} \right)
\geq \underline{\alpha} \left|\blambda_2 - \blambda_{\frac{\theta}{2}}\right|^{2 + \gamma},
    \end{align*}
    which implies
    \begin{align*}
        -\frac{1}{2} \nabla f_{\bd}(\blambda_{\frac{\theta}{2}})^{\top}\left(\blambda_2 - \blambda_\theta \right)
\geq \underline{\alpha} \left|\frac{1}{2}\left(\blambda_2 - \blambda_\theta \right) \right|^{2 + \gamma}.
    \end{align*}
    Similarly, it can be seen that 
    \begin{align*}
        f_{\bd}(\blambda_\theta) - f_{\bd}(\blambda_{\frac{\theta}{2}}) - \nabla f_{\bd}(\blambda_{\frac{\theta}{2}})^{\top}\left(\blambda_\theta - \blambda_{\frac{\theta}{2}} \right)
\geq \underline{\alpha} \left|\blambda_\theta - \blambda_{\frac{\theta}{2}}\right|^{2 + \gamma},
    \end{align*}
    which implies
    \begin{align*}
        \frac{1}{2} \nabla f_{\bd}(\blambda_{\frac{\theta}{2}})^{\top}\left(\blambda_2 - \blambda_\theta \right)
\geq \underline{\alpha} \left|\frac{1}{2}\left(\blambda_\theta - \blambda_2 \right) \right|^{2 + \gamma}.
    \end{align*}
    Hence, we must have $\blambda_\theta = \blambda_2$, which contradicts the condition $\blambda_\theta\neq \blambda_2$, and finish the proof.
\end{proof}
Proposition \ref{prop:2-order imply unique dual} shows that Assumption \ref{def: unique-dual} is in fact implicitly assumed under the sufficient conditions typically made in the literature. In addition to forcing $\blambda^{\star}$ to be a unique solution to $\min_{\blambda \in \RR^{m}_{\geq 0}} f_{\bd}(\blambda)$, Assumption \ref{def: 2-order-growth} imposes curvature requirements of $f_{\bd}(\blambda)$ at $\blambda^{\star}$, captured by parameter $\gamma$. This parameter $\gamma$ appears in the fundamental regret lower bound and determines the best achievable regret scaling (cf. Proposition \ref{thm: fundamental lower bound} and Theorem 1 in \cite{besbes2024dynamic}). In that sense, Assumption \ref{def: 2-order-growth} is an essential condition. In our Assumptions \ref{assum: starting-from-zero} and \ref{assum: small-probability-starting-from-zero}, we impose a reverse H{\"o}lder condition on the distributions with parameter $\beta$, which effectively captures the $\gamma$ as appeared in Assumption \ref{def: 2-order-growth}.

Recall that our Assumptions \ref{assum: starting-from-zero} and \ref{assum: small-probability-starting-from-zero} rely purely on the properties of the underlying $F$ and are independent of $\bd$. This offers several advantages over directly imposing Assumption \ref{def: 2-order-growth}. Beyond being more intuitive and easy-to-check, one notable benefit of our assumptions is their ability to handle cases involving scarce resources. In fact, the subtle scaling regime of a sequence of OLP instances $(F, \bb_T, T)$ such that $\bb_T \propto T^\eta$ for $\eta \in (0, 1)$ cannot be simply described by its ``fluid'' instance, in which $\bd $ approaches $ \bm{0}$. In such cases, Assumption \ref{def: 2-order-growth} is no longer suitable. In contrast, our assumptions uniformly cover these regimes, for which the regret guarantee in Theorem \ref{thm: regret of CE} remains valid.



Through the lens of Assumption \ref{def: 2-order-growth}, we can justify the requirement in our Assumption \ref{assum: small-probability-starting-from-zero}, in particular, that there exists $\ba_o$ such that the conditional reward distributions in a neighborhood of $\ba_o$ are all supported on intervals starting from zero. In fact, without such a restriction, it is possible to construct simple OLP instances with $\beta = 0$, yet the second-order growth condition ( Assumption \ref{def: 2-order-growth} with $\gamma = 0$) fails to hold.

\begin{exmp}\label{exmp:violating 2-order}
$(\ba, r) \sim {\sf Unif}[1,2]^2.$ The initial inventory for the resource is $\bb = \bd T = 1.5 T$.    
\end{exmp}
The distribution in Example \ref{exmp:violating 2-order} satisfies all the conditions in Assumption \ref{assum: small-probability-starting-from-zero} with $\beta = 0$, except for the requirement (iii). A straightforward calculation reveals that $\blambda^{\star} = 0$, and Assumption \ref{def: 2-order-growth} is satisfied only with $\gamma = 1$, suggesting that instead of polylogarithmic in $T$ regret, only polynomial in $T$ regret is possible in this specific case. We note that the choice of $\bb$ is a boundary case. Generally, a precise characterization of the achievable regret scaling for OLP instances requires accurately tracking all such boundary cases, which can be complicated depending on the specific problem structure. In this work, we aim to strike a balance between generality and clarity. Hence we present our main results under the current set of assumptions that exclude tricky boundary cases. A further refined regret analysis  is left for future research.


\subsection{Global conditions}
Some prior work imposes stronger, global-version of Assumption \ref{def: 2-order-growth}, which essentially requires the dual objective $f_{\bd}(\blambda)$ to be quadratically lower bounded not only at the dual optimal point $\blambda^{\star}$, but at all points in a given set. We here provide a typical formulation of such global condition.
\begin{assumption}[Uniform second-order growth condition]\label{def: uniform-2-order-growth}
There exists a compact convex set $\Omega \subseteq \mathbb{R}_{\geq 0}^m$ such that for any $t \in [T]$, any $\bb'$, and any problem instance $\cI_{T - t + 1}$, the relaxed offline optimum $V^{\text{hind}}_{\bb', t}(\cI_{T - t + 1})$ (cf. problem (\ref{program: offline primal}) starting from time $t$) possesses one optimal dual solution $\tilde{\blambda}\in \Omega $. Moreover, there exist two positive constants $\underline{\alpha}, \bar{\alpha}$ such that for any $\blambda', \blambda'' \in \Omega$, it holds that
\begin{align*}
    \underline{\alpha} \mathbb{E}_{\ba \sim F^{\ba}} \left( \ba^\top \blambda' - \ba^\top \blambda'' \right)^2 &\leq \EE_{\ba \sim F^{\ba}} \left[ \left( F_{\ba}^r(\ba^\top \blambda') - F_{\ba}^r(\ba^\top \blambda'') \right) \left( \ba^\top \blambda' - \ba^\top \blambda'' \right) \right] \nonumber \\
    &\leq \bar{\alpha} \mathbb{E}_{\ba \sim F^{\ba}}\left( \ba^\top \blambda' - \ba^\top \blambda'' \right)^2.
\end{align*}
\end{assumption}
 Assumption \ref{def: uniform-2-order-growth} appears as Assumption 2 in \cite{jiang2022degeneracy}. Other examples of global second-order conditions include Assumption 3(b) in \cite{li2022online}. We note that \cite{jiang2022degeneracy} establishes an $\mathcal{O}(\log T)$ regret guarantee of {\sf CE} under Assumption \ref{def: uniform-2-order-growth}, also without imposing non-degeneracy conditions (i.e. Assumption \ref{def: dual stability} or \ref{def: dual non-degeneracy}). We here provide a comparison between their results and ours.

First, the $\mathcal{O}(\log T)$ regret guarantee in \cite{jiang2022degeneracy} is established in their ``semi-discrete'' setting, where the number of different types of resource consumption vectors is finite. We note that this discreteness is essential to their analysis, as the inverse of the probability mass of each resource consumption type enters their regret bound. In contrast, our Assumption \ref{assum: starting-from-zero} allows the distribution of $\ba$ to be arbitrary, capturing discrete, continuous and mixed-type distributions.

Second, Assumption \ref{def: uniform-2-order-growth} is hard to check. In fact, to check whether Assumption \ref{def: uniform-2-order-growth} holds, we need to first determine $\Omega$, then verify the second-order growth property for every pair of vectors in $\Omega$. We note that the uniformity of Assumption \ref{def: uniform-2-order-growth} also makes it more restrictive than the standard local second-order conditions such as Assumption \ref{def: 2-order-growth}. In contrast, our Assumptions \ref{assum: starting-from-zero} and \ref{assum: small-probability-starting-from-zero} are distributional assumptions that are straightforward to verify.

\section{Proof of Lemma \ref{lem: ce-regret-decomp}}\label{appendix: proof of decomposition lemma}

\begin{proof}[Proof of Lemma \ref{lem: ce-regret-decomp}]
Recall that 
\begin{align}
    \textsc{reg}_T(\pi) = \sum_{t = 1}^T \E\left[ V^{\rm off}_{t-1} - V^{\rm off}_{t} - r_t x^{\pi}_t \right], \label{eq: myopic regret-app}
\end{align}
where 
\begin{align*}
    V^{{\rm off}}_{t}= V^{\rm off}_{t, \bb_t}(\cI_t) = \max &\quad  \sum_{j = t + 1}^T r_j x_j\\
    \mbox{s.t.} & \quad \sum_{j = t + 1}^T a_{i j} x_j \leq b_{i t}, \quad i = 1, \dots, m,
    \\& x_j \in [0, 1], \quad j = t + 1,\dots, T.
\end{align*}
By LP duality theory (Lemma \ref{lem: helper-dual LP simplified}), we also have
\begin{equation*}
    V^{{\rm off}}_{t} = \min_{\blambda \in \RR^{m}_{\geq 0}} \bb^{\top}_t \blambda + \sum_{j = t + 1}^T(r_j - \ba^{\top}_j \blambda)^+.
\end{equation*}
By Lemma \ref{lem: helper-offline-opt-induction}, we have the equality
\[
V^{\rm off}_{t-1, \bb_{t - 1}}(\cI_{t-1}) = r_t x^*_t +  V^{\rm off}_{t, \bb_{t - 1} - \ba_t x^*_t}(\cI_t),
\]
where $\{x^*_j\}_{j = t}^T$ is any primal optimal solution. The equality implies that $V^{\rm off}_{t-1} - V^{\rm off}_{t} - r_t x^{\pi}_t = 0$ if $x^{\pi}_t = x^*_t$. Indeed, if $x^{\pi}_t = x^*_t = 1$, then $\bb_t = \bb_{t-1} - \ba_t x^{\pi}_t = \bb_{t - 1} - \ba_t$, and 
\[
V^{\rm off}_{t-1} - V^{\rm off}_{t} - r_t x^{\pi}_t = V^{\rm off}_{t-1, \bb_{t - 1}} - V^{\rm off}_{t, \bb_t} - r_t = V^{\rm off}_{t, \bb_{t - 1} - \ba_t} -  V^{\rm off}_{t, \bb_t} = 0.  
\]
The case $x^{\pi}_t = x^*_t = 0$ is similar. In other words, if the decision under policy $\pi$ coincides with the optimal offline solution, then we have zero regret. Thus it suffices to consider cases where $x^*_t \neq x^{\pi}_t$. To proceed, we suppose $\bx_t^*$ is basic feasible optimal solution for each $\cI_{t-1}$, and consider the following three cases.

\noindent $\bullet\quad$\textbf{Case 1. $x^*_t \in (0, 1).$\ } In this case $x^*_t$ is fractional. By Lemma \ref{lem: helper-non-integer}, $\bx_t^*$ has at most $m$ fractional variables. Since $\{(\ba_j, r_j)\}_{j = t}^T$ are \emph{i.i.d.}, the chance that $x^*_t$ is fractional is at most $\frac{m}{T - t + 1}.$  In such cases that $x^*_t$ is indeed fractional, we bound the per period regret as follows
    \[
    V^{\rm off}_{t-1} - V^{\rm off}_{t} - r_t x^{\pi}_t  \leq V^{\rm off}_{t-1} - V^{\rm off}_{t}   \leq V^{\rm off}_{t-1, \bb_{t - 1}} - V^{\rm off}_{t, \bb_{t - 1} - \ba_t} \leq V^{\rm off}_{t, \bb_{t - 1}} - V^{\rm off}_{t, \bb_{t - 1} - \ba_t} + {\bar r},
    \]
where the first inequality follows from the non-negativity of $x^{\pi}_t$, the second follows from the monotonicity of $V^{\rm off}_{t, \bc}$ in $\bc$, and the last inequality follows from Lemma \ref{lem: helper-offline-opt-induction} and the monotonicity of $V^{\rm off}_{t, \bc}$ in $\bc$. By the boundedness of $\ba_t$, we further have
\begin{align*}
    V^{\rm off}_{t, \bb_{t - 1}} - V^{\rm off}_{t, \bb_{t - 1} - \ba_t} + {\bar r} = & (T-t)\left(g_{t,\bb_{t-1}}(\dt) - g_{t,\bb_{t-1}-\ba_t}(\dtt)\right) + {\bar r}\\
    \leq & (T-t)\left(g_{t,\bb_{t-1}}(\dtt) - g_{t,\bb_{t-1}-\ba_t}(\dtt)\right) + {\bar r}\\
    = & \ba_t^\top \dtt + \bar r \leq \frac{m\bar A\bar r}{\underline A } + \bar r,
\end{align*}
where the last inequality follows from Lemma \ref{lem: helper-boundedness-lambda}. We then combine the above arguments and conclude that 
\begin{equation}\label{eq: decomposition-non-integer-case}
    \E\left[ \left(V^{\rm off}_{t-1} - V^{\rm off}_{t} - r_t x^{\pi}_t \right) \II_{\left\{x^*_t \textrm{\ is fractional}\right\}}\right] \ \leq \  \frac{\left(1 + \frac{m\Ab}{\Al} \right)m {\bar r}}{T - t + 1}.
\end{equation}

\noindent$\bullet\quad$\textbf{Case 2. $x^{\pi}_t = 1, x^*_t = 0.$\ } By definition of $\pi$, $x^{\pi}_t = 1$ implies $r_t \geq \ba_t^{\top} \ft$. 
By Lemma \ref{lem: helper-offline-opt-induction}, $x^*_t = 0$ implies that
$V^{\rm off}_{t - 1, \bb_{t - 1}} = V^{\rm off}_{t, \bb_{t - 1}} \geq V^{\rm off}_{t, \bb_{t - 1} - \ba_t} + r_t$, more specifically
\[
\bb^{\top}_{t - 1} \dt + \sum_{j = t + 1 }^T(r_j - \ba^{\top}_j \dt)^+ \geq (\bb^{\top}_{t - 1} - \ba_t) {\dtt} + \sum_{j = t + 1 }^T(r_j - \ba^{\top}_j {\dtt})^+ + r_t,
\]
where we recall that $\dt$ and ${\dtt}$ are optimal dual solutions to Problem (\ref{program: to-go-offline-dual-1}) and Problem (\ref{program: to-go-offline-dual-2}), respectively. Since $ \dt$ minimizes the above left hand side, the above further implies 
\begin{align*}
    & \bb^{\top}_{t - 1} {\dtt} + \sum_{j = t + 1 }^T(r_j - \ba^{\top}_j {\dtt})^+ \geq (\bb^{\top}_{t - 1} - \ba_t) {\dtt} + \sum_{j = t + 1 }^T(r_j - \ba^{\top}_j {\dtt})^+ + r_t,\\
    & \Leftrightarrow \ \ba^{\top}_t {\dtt} \geq r_t.
\end{align*}
 Hence $x^{\pi}_t = 1, x^*_t = 0$ implies $\ba_t^{\top} \ft \leq r_t \leq \ba^{\top}_t {\dtt}.$

Since $x^{\pi}_t = 1$, we have $\bb_t = \bb_{t - 1} - \ba_t x^{\pi}_t = \bb_{t - 1} - \ba_t.$  Therefore we may bound the per period regret as follows
\begin{align*}
    &V^{\rm off}_{t-1} - V^{\rm off}_{t} - r_t x^{\pi}_t  \\&\quad\quad= \quad V^{\rm off}_{t-1,\bb_{t - 1}} - V^{\rm off}_{t, \bb_{t - 1} - \ba_t } - r_t 
    \\&\quad\quad =\quad V^{\rm off}_{t, \bb_{t - 1}} - V^{\rm off}_{t, \bb_{t - 1} - \ba_t } - r_t, 
    \\&\quad\quad =\quad \bb^{\top}_{t - 1} \dt + \sum_{j = t + 1 }^T(r_j - \ba^{\top}_j \dt)^+ - \bigg((\bb_{t - 1} - \ba_t)^{\top} {\dtt} + \sum_{j = t + 1 }^T(r_j - \ba^{\top}_j {\dtt})^+\bigg) - r_t,
     \\&\quad\quad \leq\quad \bb^{\top}_{t - 1}  {\dtt} + \sum_{j = t + 1 }^T(r_j - \ba^{\top}_j  {\dtt})^+ - \bigg((\bb_{t - 1}- \ba_t)^{\top} {\dtt} + \sum_{j = t + 1 }^T(r_j - \ba^{\top}_j {\dtt})^+\bigg) - r_t,
     \\&\quad\quad =\quad \ba^{\top}_t {\dtt} - r_t,
\end{align*}
where the inequality follows from the fact that $\dt$ is the minimizer of Problem (\ref{program: to-go-offline-dual-1}) and that ${\dtt}$ is feasible (non-negative) and hence must achieve a larger objective value when plugging into the objective function of Problem (\ref{program: to-go-offline-dual-1}). Combining the above, we have  
\begin{equation}\label{eq: decomposition-first-type-mistake}
    \E\left[ \left(V^{\rm off}_{t-1} - V^{\rm off}_{t} - r_t x^{\pi}_t \right) \II_{\left\{x^*_t = 0, x^{\pi}_t = 1\right\}}\right] \ \leq \  \E\left[ \left(\ba^{\top}_t {\dtt} - r_t\right)\II_{\left\{\ba_t^{\top} \ft \leq r_t \leq \ba^{\top}_t {\dtt}\right\}}\right],
\end{equation}
where the expectation is taken over $\cI_{t-1}.$ 

\noindent $\bullet\quad$\textbf{Case 3. $x^{\pi}_t = 0, x^*_t = 1.$\ } This case is very similar to Case 2. We omit the detailed arguments. The conclusion can be summarized as
\begin{equation}\label{eq: decomposition-second-type-mistake}
   \E\left[ \left(V^{\rm off}_{t-1} - V^{\rm off}_{t} - r_t x^{\pi}_t \right) \II_{\left\{x^*_t = 1, x^{\pi}_t = 0\right\}}\right] \ \leq \  \E\left[ \left(r_t - \ba^{\top}_t {\dt}\right)\II_{\left\{  \ba^{\top}_t {\dt} \leq r_t \leq \ba_t^{\top} \ft\right\}}\right].
\end{equation}

Now plugging (\ref{eq: decomposition-non-integer-case}), (\ref{eq: decomposition-first-type-mistake}) and (\ref{eq: decomposition-second-type-mistake}) into (\ref{eq: myopic regret-app}), we conclude that 
\begin{align*}
    \textsc{reg}_T(\pi) &= \sum_{t = 1}^T \E\left[ V^{\rm off}_{t-1} - V^{\rm off}_{t} - r_t x^{\pi}_t \right], \\
    &= \sum_{t = 1}^T \E\left[ \left(V^{\rm off}_{t-1} - V^{\rm off}_{t} - r_t x^{\pi}_t\right)\II_{\left\{ x^*_t \neq x^{\pi}_t \right\}} \right],
     \\&= \ \ \sum_{t = 1}^T \E\left[ \left(V^{\rm off}_{t-1} - V^{\rm off}_{t} - r_t x^{\pi}_t \right) \II_{\left\{x^*_t \textrm{\ is fractional}\right\}}\right] \ \ \ \ \  (\textrm{Case 1})
     \\& \quad\quad\quad + \ \ \sum_{t = 1}^T \E\left[ \left(V^{\rm off}_{t-1} - V^{\rm off}_{t} - r_t x^{\pi}_t \right) \II_{\left\{x^*_t = 0, x^{\pi}_t = 1\right\}}\right]\ \ \ \ \ (\textrm{Case 2})
     \\& \quad\quad\quad\quad\quad + \ \ \sum_{t = 1}^T \E\left[ \left(V^{\rm off}_{t-1} - V^{\rm off}_{t} - r_t x^{\pi}_t \right) \II_{\left\{x^*_t = 1, x^{\pi}_t = 0\right\}}\right]\ \ \ \ \ (\textrm{Case 3})
     \\& \leq \ \  \sum_{t = 1}^T \frac{\left(1 + \frac{m\Ab}{\Al} \right)m {\bar r}}{T - t + 1} + \sum_{t = 1}^T  \E\left[ \left(\ba^{\top}_t {\dtt} - r_t\right)\II_{\left\{\ba_t^{\top} \ft \leq r_t \leq \ba^{\top}_t {\dtt}\right\}}\right] \\&\quad\quad\quad\quad\quad\quad\quad\quad\quad\quad\quad  + \ \sum_{t = 1}^T  \E\left[ \left(r_t - \ba^{\top}_t {\dt}\right)\II_{\left\{  \ba^{\top}_t {\dt} \leq r_t \leq \ba_t^{\top} \ft\right\}}\right]
     \\& \leq \ \ \ \cC\log{T} + \sum_{t = 1}^T  \E\left[ \left(\ba^{\top}_t {\dtt} - r_t\right)\II_{\left\{\ba_t^{\top} \ft \leq r_t \leq \ba^{\top}_t {\dtt}\right\}}\right]\NNN\\
    &\quad\quad\quad\quad\quad\quad\quad\quad\quad\quad\quad +  \ \sum_{t = 1}^T  \E\left[ \left(r_t - \ba^{\top}_t {\dt} \right)\II_{\left\{\ba^{\top}_t {\dt} \leq r_t \leq \ba_t^{\top} \ft \right\}}\right],
\end{align*}
where $\cC \triangleq 2\left(1 + \frac{m\Ab}{\Al} \right)m {\bar r} $.
\end{proof}

\section{Complete proof of Lemma \ref{lem: concentration results of dual optimum}}\label{appendix: proof of concentration lemma}

In this section, we complete the proof of Lemma \ref{lem: concentration results of dual optimum} under Assumption \ref{assum: small-probability-starting-from-zero}.

\begin{proof}[Proof of Lemma \ref{lem: concentration results of dual optimum} under Assumption~\ref{assum: small-probability-starting-from-zero}]
Observe that under Assumption \ref{assum: small-probability-starting-from-zero}, for any $\ba \in \textrm{supp}(F^{\ba})$, $c_\beta (r_{\ba} - l_{\ba})^{1 + \beta} \leq F^r_{\ba}(r_{\ba}) - F^r_{\ba}(l_{\ba}) = 1 \leq c_\nu(r_{\ba} - l_{\ba})^\nu.$ Thus $c_\nu^{\frac{1}{\nu}} \leq r_{\ba} - l_{\ba} \leq c_\beta^{-\frac{1}{1 + \beta}}$ for any $\ba$. 


We divide the proof of Lemma \ref{lem: concentration results of dual optimum} under Assumption \ref{assum: small-probability-starting-from-zero} into several cases.

\noindent\textbf{Case 1.\ } There exists $\ba \in {\rm supp}(F^{\ba})$ such that either one of the following is satisfied
\begin{align}
    &\frac{l_{\ba} + r_{\ba}}{2} - \frac{1}{8}c_\nu^{\frac{1}{\nu}} \leq \ba^{\top} \dt \leq \frac{l_{\ba} + r_{\ba}}{2} + \frac{1}{8}c_\nu^{\frac{1}{\nu}},
    \label{Lemc_53_as22_c11}\\
    &\frac{l_{\ba} + r_{\ba}}{2} - \frac{1}{8}c_\nu^{\frac{1}{\nu}} \leq \ba^{\top} \ft \leq \frac{l_{\ba} + r_{\ba}}{2} + \frac{1}{8}c_\nu^{\frac{1}{\nu}},\label{Lemc_53_as22_c12}
\end{align} 
where without loss of generality, we assume that \eqref{Lemc_53_as22_c11} holds, and note that the other case, i.e., \eqref{Lemc_53_as22_c12} holds, is nearly identical. Denote by $r_1 \triangleq \frac{1}{8}\left(\left(c_L + \sqrt{m}\frac{\bar r}{\Al}\right)^{-1}\wedge 1\right) c_\nu^{\frac{1}{\nu}}.$ Consider a ball $\cB(\ba, r_1)$ centered at $\ba$ with radius $r_1$, where we recall from Assumption \ref{assum: small-probability-starting-from-zero} that $c_L$ is the Lipschitz constant of $l_\ba$ and $r_\ba$.
By Lemma \ref{lem: helper-continuity-under-assumption-2-3}, for all $\ba' \in \cB(\ba, r_1)$, 
\begin{align}\label{eq: small-probability-non-zero-proof-case-1}
    \frac{l_{\ba'} + r_{\ba'}}{2} - \frac{1}{4}c_\nu^{\frac{1}{\nu}} \leq \ba'^{\top} \dt \leq \frac{l_{\ba'} + r_{\ba'}}{2} + \frac{1}{4}c_\nu^{\frac{1}{\nu}}.
\end{align}
Because $r_{\ba} - l_{\ba}\geq c_\nu^{\frac{1}{\nu}}$, for all $\ba' \in \cB(\ba, r_1)$, $\ba'^{\top}\dt\in [l_{\ba'}, r_{\ba'}]$. If $\ba'^{\top}\ft \in [l_{\ba'}, r_{\ba'}]$, then since both $\ba'^{\top}\ft$ and $\ba'^{\top}\dt$ belong to the interval $[l_{\ba'}, r_{\ba'}]$, we have 
\begin{align}\label{eq: small-probability-non-zero-proof-case-1-ff1}
     \left|F^r_{\ba'}(\ba'^{\top}\dt) - F^r_{\ba'}(\ba'^{\top}\ft)\right| \geq c_\beta \left(\ba'^{\top}\dt - \ba'^{\top}\ft\right)^{1 + \beta}.
\end{align}
If $\ba'^{\top}\ft \notin [l_{\ba'}, r_{\ba'}]$, then by \eqref{eq: small-probability-non-zero-proof-case-1}, $\ba'^{\top} \dt$ is at least $\frac{1}{4}c_\nu^{\frac{1}{\nu}}$ far away from $l_{\ba'}$ and $r_{\ba'}$, which implies 
\begin{align}\label{eq: small-probability-non-zero-proof-case-1-ff2}
     \left|F^r_{\ba'}(\ba'^{\top}\dt) - F^r_{\ba'}(\ba'^{\top}\ft)\right| \geq & \min\left(\left|F^r_{\ba'}(\ba'^{\top}\dt) - F^r_{\ba'}(l_{\ba'})\right|,\left|F^r_{\ba'}(\ba'^{\top}\dt) - F^r_{\ba'}(r_{\ba'})\right|\right)\nonumber\\
     & \geq c_\beta \left(\frac{1}{4}c_\nu^{\frac{1}{\nu}}\right)^{1 + \beta} = 4^{- 1- \beta}c_\beta c_\nu^{\frac{1 + \beta}{\nu}}.
\end{align}
Applying Lemma \ref{lem: ball-beats-orthogonality}, we have
\begin{align}\label{eq: ball-beats-orthogonality-case-1}
    \E_{\ba' \sim F^{\ba}, \ba' \in \cB(\ba, r)} (\ba'^\top \blambda_1 - \ba'^{\top} \blambda_2)^2 \geq C_f\left(\frac{r_1\wedge w}{6} \sin{(\theta)}\right)^{m + 2}\|\blambda_1 - \blambda_2\|^2,
\end{align}
where $C_f = l_f\frac{\pi^{m/2}}{\Gamma\left(\frac{m}{2}+1\right)}$. We thus have 
\begin{align}\label{eq: ball-beats-orthogonality-case-1-eb1}
    &\E_{\ba' \sim F^{\ba}}\bigg[\bigg(F^r_{\ba'}(\ba'^{\top}{\ft}) - F^r_{\ba'}(\ba'^{\top}\dt ) \bigg)\bigg(\ba'^{\top}{\ft} - \ba'^{\top}\dt\bigg)\bigg]\nonumber\\
    \geq  & \ \E_{\ba' \sim F^{\ba}, \ba' \in \cB(\ba, r_1)}\bigg[\bigg(F^r_{\ba'}(\ba'^{\top}{\ft}) - F^r_{\ba'}(\ba'^{\top}\dt ) \bigg)\bigg(\ba'^{\top}{\ft} - \ba'^{\top}\dt\bigg)\bigg]\nonumber\\
    \geq & \   \E_{\ba' \sim F^{\ba}, \ba' \in \cB(\ba, r_1)}\bigg[\min\left(c_\beta \left|\ba'^{\top}\dt - \ba'^{\top}\ft\right|^{1 + \beta}, 4^{- 1- \beta}c_\beta c_\nu^{\frac{1 + \beta}{\nu}}\right)\left|\ba'^{\top}{\ft} - \ba'^{\top}\dt\right|\bigg]\quad \textrm{(by \eqref{eq: small-probability-non-zero-proof-case-1-ff1} and \eqref{eq: small-probability-non-zero-proof-case-1-ff2})}\nonumber\\
    = & \   \E_{\ba' \sim F^{\ba}, \ba' \in \cB(\ba, r_1)}\bigg[\left|\ba'^{\top}\dt - \ba'^{\top}\ft\right|^{2 + \beta}  \min\left(c_\beta, 4^{- 1- \beta}c_\beta c_\nu^{\frac{1 + \beta}{\nu}}\left|\ba'^{\top}{\ft} - \ba'^{\top}\dt\right|^{-1 - \beta}\right)\bigg]\nonumber\\
    \geq & \    \min\left(c_\beta, 4^{- 1- \beta}c_\beta c_\nu^{\frac{1 + \beta}{\nu}} \left(2 \sqrt{m}\frac{\bar r \Au}{\Al}\right)^{-1 - \beta}\right) \E_{\ba' \sim F^{\ba}, \ba' \in \cB(\ba, r_1)}\bigg[\left|\ba'^{\top}\dt - \ba'^{\top}\ft\right|^{2 + \beta} \bigg]\quad \textrm{(by Lemma \ref{lem: helper-boundedness-lambda})}\nonumber\\
    \geq & \  c_1 \left(\E_{\ba' \sim F^{\ba}, \ba' \in \cB(\ba, r_1)}\left[\left(\ba'^{\top}\dt - \ba'^{\top}\ft\right)^{2} \right]\right)^{\frac{2 + \beta}{2}}\quad\textrm{(by Jensen's inequality)}\nonumber\\
    \geq & \  c_1 C_f^{\frac{2 + \beta}{2}}\left(\frac{r_1\wedge w}{6} \sin{(\theta)}\right)^{\frac{(2 + \beta)(m + 2)}{2}}\|\dt - \ft\|^{2 + \beta} \quad \textrm{(by \eqref{eq: ball-beats-orthogonality-case-1})},
\end{align}
where $c_1 = \min\left(c_\beta, 4^{- 1- \beta}c_\beta c_\nu^{\frac{1 + \beta}{\nu}} \left(2 \sqrt{m}\frac{\bar r \Au}{\Al}\right)^{-1 - \beta}\right)$. We set $c_2 = c_1C_f^{\frac{2 + \beta}{2}}\left(\frac{r_1\wedge w}{6} \sin{(\theta)}\right)^{\frac{(2 + \beta)(m + 2)}{2}}$. Note that 
\begin{align*}
    M_1 &\ = \ \sqrt{\E_{\ba \sim F^{\ba}}\bigg[\big(\ba^{\top}\ft - \ba^{\top}\dt\big)^2\max\bigg(\bar F^r_{\ba}(\ba^\top \dt), \bar F^r_{\ba}(\ba^\top \ft)\bigg)\bigg]} \\
    & \leq \sqrt{\E_{\ba \sim F^{\ba}}\bigg[\big(\ba^{\top}\ft - \ba^{\top}\dt\big)^2\bigg]} \leq \Au \|\dt - \ft\|,
\end{align*}
which, together with \eqref{eq: ball-beats-orthogonality-case-1-eb1} and Lemma \ref{lem:concentration-proof-part-1}, implies 
\begin{align}\label{eq: small-probability-zero-proof-case-1-regret-decompose-bound-1}
    & c_2\|\dt - \ft\|^{2 + \beta} \leq \E_{\ba' \sim F^{\ba}}\bigg[\bigg(F^r_{\ba'}(\ba'^{\top}{\ft}) - F^r_{\ba'}(\ba'^{\top}\dt ) \bigg)\bigg(\ba'^{\top}{\ft} - \ba'^{\top}\dt\bigg)\bigg]\NNN\\
    & \leq \tilde C_3\sqrt{\frac{\log(T-t)}{T-t}}\bigg(M_1+\sqrt{\frac{\log(T-t)}{T-t}}\bigg) \leq \tilde C_3 \frac{\log(T-t)}{T-t} + \tilde C_3 \Au \sqrt{\frac{\log(T-t)}{T-t}}  \|\dt - \ft\|.
\end{align}
Solving \eqref{eq: small-probability-zero-proof-case-1-regret-decompose-bound-1} yields
\begin{align}\label{eq: small-probability-zero-proof-case-1-regret-decompose-bound-finalb1}
    \|\dt - \ft\| \leq & \max\left(\frac{\tilde C_3}{c_2}\left(\frac{\log(T-t)}{T-t}\right)^{\frac{1}{1 + \beta}}, \frac{\tilde C_3\Au}{c_2}\left(\frac{\log(T-t)}{T-t}\right)^{\frac{1}{2(1 + \beta)}}\right)\leq c_3\left(\frac{\log(T-t)}{T-t}\right)^{\frac{1}{2(1 + \beta)}},
\end{align}
where $c_3 = \max\left(\frac{\tilde C_3}{c_2}, \frac{\tilde C_3\Au}{c_2}\right)$, and we use the fact that $\frac{\log(T-t)}{T-t} \leq 1$ and $\frac{1}{2 + \beta} \geq \frac{1}{2(1 + \beta)}$. By \eqref{eq: small-probability-zero-proof-case-1-regret-decompose-bound-1} and \eqref{eq: small-probability-zero-proof-case-1-regret-decompose-bound-finalb1}, we obtain 
\begin{align}\label{eq: small-probability-starting-from-zero-case-1-proof}
    &\E_{\ba' \sim F^{\ba}}\bigg[\bigg(F^r_{\ba'}(\ba'^{\top}{\ft}) - F^r_{\ba'}(\ba'^{\top}\dt ) \bigg)\bigg(\ba'^{\top}{\ft} - \ba'^{\top}\dt\bigg)\bigg] \NNN\\
    & \leq \tilde C_3 \frac{\log(T-t)}{T-t} + \tilde C_3 \Au \sqrt{\frac{\log(T-t)}{T-t}}  \|\dt - \ft\| \leq c_4 \left(\frac{\log(T-t)}{T-t}\right)^{\frac{2 + \beta}{2(1 + \beta)}},
\end{align}
with probability at least $1 - \frac{9 C\log(T - t)}{(T - t)^2}$ and for $T - t \geq e^{\frac{\tilde C_1}{32} + 1}$, with $c_4 = \tilde C_3 + \tilde C_3\Au c_3$.

\noindent\textbf{Case 2.\ } For any $\ba \in \textrm{supp}(F^{\ba})$, the following holds
\begin{align*}
    &\ba^{\top} \dt  \notin \left[\frac{l_{\ba} + r_{\ba}}{2} - \frac{1}{8}c_\nu^{\frac{1}{\nu}} , \frac{l_{\ba} + r_{\ba}}{2} + \frac{1}{8}c_\nu^{\frac{1}{\nu}}\right],\\
    &\ba^{\top} \ft  \notin \left[\frac{l_{\ba} + r_{\ba}}{2} - \frac{1}{8}c_\nu^{\frac{1}{\nu}} , \frac{l_{\ba} + r_{\ba}}{2} + \frac{1}{8}c_\nu^{\frac{1}{\nu}}\right].
\end{align*}
We first show that there does not exist $\ba', \ba'' \in \textrm{supp}(F^{\ba})$ such that 
\begin{align*}
    \ba'^{\top} \dt  \leq \frac{l_{\ba'} + r_{\ba'}}{2} - \frac{1}{8}c_\nu^{\frac{1}{\nu}}, \ \  \ba''^{\top} \dt  \geq \frac{l_{\ba''} + r_{\ba''}}{2} + \frac{1}{8}c_\nu^{\frac{1}{\nu}}
\end{align*}
hold. For otherwise, consider $\ba(p) \triangleq p \ba' + (1 - p)\ba''$ for $p \in [0 , 1]$, and  function
\begin{align*}
    g_{\textrm{gap}}(p) \triangleq \frac{l_{\ba(p)} + r_{\ba(p)}}{2} - \ba(p)^{\top} \dt.
\end{align*}
Then $g_{\textrm{gap}}(0) \leq - \frac{1}{8}c_\nu^{\frac{1}{\nu}}$ and $g_{\textrm{gap}}(1) \geq  \frac{1}{8}c_\nu^{\frac{1}{\nu}}$. By the continuity of $l_{\ba}$ and $r_{\ba}$ in $\ba$, $g_{\textrm{gap}}(p)$ is a continuous function in $p$. By the intermediate value theorem, there must exist $p' \in (0, 1)$ such that $g_{\textrm{gap}}(p') = 0. $ Such $p'$ corresponds to $p' \ba' + (1 - p')\ba''$, which belongs to $\textrm{supp}(F^{\ba})$ by its own convexity. However, the existence of $p' \ba' + (1 - p')\ba''$ violates the condition of Case 2. Therefore by contradiction, it must be that either 
\begin{align*}
    \ba^{\top} \dt  < \frac{l_{\ba} + r_{\ba}}{2} - \frac{1}{8}c_\nu^{\frac{1}{\nu}},
\end{align*}
for all $\ba \in \textrm{supp}(F^{\ba})$, or 
\begin{align*}
    \ba^{\top} \dt  > \frac{l_{\ba} + r_{\ba}}{2} + \frac{1}{8}c_\nu^{\frac{1}{\nu}},
\end{align*}
for all $\ba \in \textrm{supp}(F^{\ba})$. We note that the same should be true with respect to $\ft$. We thus consider three sub-cases.

\noindent \textbf{Sub-case 1.\ } For all $\ba, \ba^{\top} \dt  < \frac{l_{\ba} + r_{\ba}}{2} - \frac{1}{8}c_\nu^{\frac{1}{\nu}},$ yet $\ba^{\top} \ft  > \frac{l_{\ba} + r_{\ba}}{2} + \frac{1}{8}c_\nu^{\frac{1}{\nu}}$, or the other way round. In this case, for all $\ba' \in \textrm{supp}(F^\ba)$, 
\begin{align*}
    \left|F^r_{\ba'}(\ba'^{\top}{\ft}) - F^r_{\ba'}(\ba'^{\top}\dt ) \right|  &\geq  F^r_{\ba'}\left(\frac{l_{\ba} + r_{\ba}}{2} + \frac{1}{8}c_\nu^{\frac{1}{\nu}}\right) - F^r_{\ba'}\left(\frac{l_{\ba} + r_{\ba}}{2} - \frac{1}{8}c_\nu^{\frac{1}{\nu}}\right) \geq c_\beta \left(\frac{1}{4}c_\nu^{\frac{1}{\nu}}\right)^{1 + \beta},
\end{align*}
where the second inequality follows from the H{\"o}lder condition in Assumption \ref{assum: small-probability-starting-from-zero} and the fact that both $\frac{l_{\ba} + r_{\ba}}{2} - \frac{1}{8}c_\nu^{\frac{1}{\nu}}$ and $\frac{l_{\ba} + r_{\ba}}{2} + \frac{1}{8}c_\nu^{\frac{1}{\nu}}$ belong to the support, i.e. $[l_{\ba}, r_{\ba}]$. Thus we have 
\begin{align}\label{eq: small-probability-zero-proof-case-2-regret-decompose-bound-21111}
    &\E_{\ba' \sim F^{\ba}}\bigg[\bigg(F^r_{\ba'}(\ba'^{\top}{\ft}) - F^r_{\ba'}(\ba'^{\top}\dt ) \bigg)\bigg(\ba'^{\top}{\ft} - \ba'^{\top}\dt\bigg)\bigg] \ \geq 4^{- 1- \beta}c_\beta c_\nu^{\frac{1 + \beta}{\nu}} \E_{\ba' \sim F^{\ba}}\bigg|\ba'^{\top}{\ft} - \ba'^{\top}\dt\bigg|.
\end{align}
Note that 
\begin{align*}
    M_1 &\ = \ \sqrt{\E_{\ba \sim F^{\ba}}\bigg[\big(\ba^{\top}\ft - \ba^{\top}\dt\big)^2\max\bigg(\bar F^r_{\ba}(\ba^\top \dt), \bar F^r_{\ba}(\ba^\top \ft)\bigg)\bigg]} \\
    & \leq \sqrt{\E_{\ba \sim F^{\ba}}\bigg[\big(\ba^{\top}\ft - \ba^{\top}\dt\big)^2\bigg]} \leq \sqrt{\frac{2\Au\sqrt{m} \bar  r}{\Al}\E_{\ba \sim F^{\ba}}\bigg|\ba^{\top}\ft - \ba^{\top}\dt\bigg|}, \quad (\textrm{by Lemma \ref{lem: helper-boundedness-lambda}})
\end{align*}
which, together with \eqref{eq: small-probability-zero-proof-case-2-regret-decompose-bound-21111} and Lemma \ref{lem:concentration-proof-part-1}, implies
\begin{align}\label{eq: small-probability-zero-proof-case-2-1-regret-decompose-bound-1}
    & 4^{- 1- \beta}c_\beta c_\nu^{\frac{1 + \beta}{\nu}} \E_{\ba' \sim F^{\ba}}\bigg|\ba'^{\top}{\ft} - \ba'^{\top}\dt\bigg| \NNN\\
    & \leq \E_{\ba' \sim F^{\ba}}\bigg[\bigg(F^r_{\ba'}(\ba'^{\top}{\ft}) - F^r_{\ba'}(\ba'^{\top}\dt ) \bigg)\bigg(\ba'^{\top}{\ft} - \ba'^{\top}\dt\bigg)\bigg]\NNN\\
    & \leq \tilde C_3\sqrt{\frac{\log(T-t)}{T-t}}\bigg(M_1+\sqrt{\frac{\log(T-t)}{T-t}}\bigg)\NNN\\
    & \leq \tilde C_3 \frac{\log(T-t)}{T-t} + \tilde C_3 \sqrt{\frac{\log(T-t)}{T-t}} \sqrt{\frac{2\Au\sqrt{m} \bar  r}{\Al}\E_{\ba \sim F^{\ba}}\bigg|\ba^{\top}\ft - \ba^{\top}\dt\bigg|}.
\end{align}
Solving \eqref{eq: small-probability-zero-proof-case-2-1-regret-decompose-bound-1}, we obtain 
\begin{align}
     &\E_{\ba' \sim F^{\ba}}\bigg|\ba'^{\top}{\ft} - \ba'^{\top}\dt\bigg| \leq \max\left(\tilde C_34^{1+\beta}c_\beta^{-1} c_\nu^{-\frac{1 + \beta}{\nu}}, \tilde C_3^2\frac{2\Au\sqrt{m} \bar  r}{\Al}4^{2+2\beta}c_\beta^{-2} c_\nu^{-\frac{2(1 + \beta)}{\nu}}\right)\frac{\log(T-t)}{T-t},
\end{align}
Plugging back into \cref{eq: small-probability-zero-proof-case-2-1-regret-decompose-bound-1} yields
\begin{align}\label{eq: small-probability-starting-from-zero-case-2-1-proof}
    &\E_{\ba' \sim F^{\ba}}\bigg[\bigg(F^r_{\ba'}(\ba'^{\top}{\ft}) - F^r_{\ba'}(\ba'^{\top}\dt ) \bigg)\bigg(\ba'^{\top}{\ft} - \ba'^{\top}\dt\bigg)\bigg]  \leq c_5 \frac{\log(T-t)}{T-t},
\end{align}
where $c_5 = \tilde C_3 + \tilde C_3\left(\frac{2\Au\sqrt{m} \bar  r}{\Al}\max\left(\tilde C_34^{1+\beta}c_\beta^{-1} c_\nu^{-\frac{1 + \beta}{\nu}}, \tilde C_3^2\frac{2\Au\sqrt{m} \bar  r}{\Al}4^{2+2\beta}c_\beta^{-2} c_\nu^{-\frac{2(1 + \beta)}{\nu}}\right)\right)^{1/2}$. Recall that $\frac{2 + \beta}{2(1 + \beta)} < 1$ and $\frac{\log(T - t)}{T - t} \leq 1$, we have $\frac{\log(T - t)}{T - t} \leq \left(\frac{\log(T - t)}{T - t}\right)^{\frac{2 + \beta}{2(1 + \beta)}}$. Further we note that $c_5 \leq \tilde C_4$. We thus complete the proof of Lemma \ref{lem: concentration results of dual optimum} under Assumption \ref{assum: small-probability-starting-from-zero} in the Sub-case 1 of Case 2. 

\noindent \textbf{Sub-case 2.\ } For all $\ba, \ba^{\top} \dt  >  \frac{l_{\ba} + r_{\ba}}{2} + \frac{1}{8}c_\nu^{\frac{1}{\nu}},$ and $\ba^{\top} \ft  > \frac{l_{\ba} + r_{\ba}}{2} + \frac{1}{8}c_\nu^{\frac{1}{\nu}}$. The proof in this case is identical to the proof under Assumption \ref{assum: starting-from-zero}, because both $\ba^\top \dt$ and $\ba^\top \ft$ never fall below the lower bound of the support of $F^r_{\ba}$ for all $\ba \in \textrm{supp}(F^{\ba})$, and thus Lemma \ref{lem: concentration results of dual optimum} is true in this sub-case with the same constant as chosen in the previous proof under Assumption \ref{assum: starting-from-zero}. 

\noindent \textbf{Sub-case 3.\ }For all $\ba, \ba^{\top} \dt  <  \frac{l_{\ba} + r_{\ba}}{2} - \frac{1}{8}c_\nu^{\frac{1}{\nu}},$ and $\ba^{\top} \ft  < \frac{l_{\ba} + r_{\ba}}{2} - \frac{1}{8}c_\nu^{\frac{1}{\nu}}$. In this case, we focus on $\cB(\ba_o, r_0) \cap \textrm{supp}(F^{\ba})$, where we recall that by Assumption \ref{assum: small-probability-starting-from-zero}, for $\ba \in \cB(\ba_o, r_0) \cap \textrm{supp}(F^{\ba})$, $l_{\ba} = 0.$ Hence for all such $\ba$, we have by the H{\"o}lder condition of Assumption \ref{assum: small-probability-starting-from-zero} that 
\begin{align*}
    \left|F^r_{\ba}(\ba^{\top}\dt) - F^r_{\ba}(\ba^{\top}\ft) \right| \geq c_\beta \left|\ba^{\top}\dt - \ba^{\top}\ft\right|^{1+\beta}.
\end{align*}
Therefore, we have
\begin{align*}
    &\E_{\ba' \sim F^{\ba}}\bigg[\bigg(F^r_{\ba'}(\ba'^{\top}{\ft}) - F^r_{\ba'}(\ba'^{\top}\dt ) \bigg)\bigg(\ba'^{\top}{\ft} - \ba'^{\top}\dt\bigg)\bigg]\\
    & \ \geq  \E_{\ba' \sim F^{\ba}, \ba' \in \cB(\ba_o, r_0)}\bigg[\bigg(F^r_{\ba'}(\ba'^{\top}{\ft}) - F^r_{\ba'}(\ba'^{\top}\dt ) \bigg)\bigg(\ba'^{\top}{\ft} - \ba'^{\top}\dt\bigg)\bigg]\\
     & \ \geq  \E_{\ba' \sim F^{\ba}, \ba' \in \cB(\ba_o, r_0)}\bigg[c_\beta \left|\ba'^{\top}\dt - \ba'^{\top}\ft\right|^{2 + \beta}\bigg]\\
     & \ \geq  c_\beta \left(\E_{\ba' \sim F^{\ba}, \ba' \in \cB(\ba_o, r_0)}\left[\left(\ba'^{\top}\dt - \ba'^{\top}\ft\right)^{2} \right]\right)^{\frac{2 + \beta}{2}}\quad \textrm{(by Jensen's inequality)}\\
    & \ \geq  c_\beta
    C_f^{\frac{2 + \beta}{2}}\left(\frac{r_0\wedge w}{6} \sin{(\theta)}\right)^{\frac{(2 + \beta)(m + 2)}{2}}\|\dt - \ft\|^{2 + \beta}\quad \textrm{(by Lemma \ref{lem: ball-beats-orthogonality}).}
\end{align*}
The rest of the proof is almost identical to Case 1, and we omit the details. This gives us 
\begin{align}\label{eq: small-probability-starting-from-zero-case-23-proof}
    &\E_{\ba' \sim F^{\ba}}\bigg[\bigg(F^r_{\ba'}(\ba'^{\top}{\ft}) - F^r_{\ba'}(\ba'^{\top}\dt ) \bigg)\bigg(\ba'^{\top}{\ft} - \ba'^{\top}\dt\bigg)\bigg] \leq c_6 \left(\frac{\log(T-t)}{T-t}\right)^{\frac{2 + \beta}{2(1 + \beta)}},
\end{align}
where $c_6 = \tilde C_3 + \tilde C_3^2\Au \left(c_\beta C_f^{\frac{2 + \beta}{2}}\left(\frac{r_0\wedge w}{6} \sin{(\theta)}\right)^{\frac{(2 + \beta)(m + 2)}{2}}\right)^{-1}  \max\left(1, \Au\right).$

Combining the above completes the proof of  Lemma \ref{lem: concentration results of dual optimum}, where we summarize the concrete constants appearing in the bound in Remark \ref{rem_constantsforms}.
\end{proof}

\section{Proof of Lemma \ref{lem:concentration-proof-part-1}}\label{appendix: proof of concentration part 1}

We introduce some additional notation. Recall that
\begin{align*}
    h_{t,\bb}(\blambda, \ba, r) &= \frac{1}{T-t}\bb^\top \blambda + \left(r - \ba^\top \blambda\right)^+,\\
    \phi_{t,\bb}(\blambda, \ba, r) &= \frac{\partial h_{t,\bb}(\blambda, \ba)}{\partial \blambda} = \frac{1}{T-t}\bb - \ba\II_{\left\{r > \ba^\top \blambda\right\}}.
\end{align*}
We further introduce functions $f_{t, \bb}(\cdot)$ and $g_{t, \bb}(\cdot)$ that are defined accordingly:
\begin{align*}
    & f_{t, \bb}(\blambda) = \E_{(\ba, r) \sim F}\left[h_{t, \bb}(\blambda, \ba, r)\right] = \frac{1}{T-t}\bb^\top \blambda +  \E_{(\ba, r) \sim F}\left[\left(r - \ba^\top \blambda\right)^+\right],\\
    & g_{t, \bb}(\blambda) = \frac{1}{T-t} \sum_{j = t + 1}^T h_{t, \bb}(\blambda, \ba_j, r_j) = \frac{1}{T-t}\bb^\top \blambda +  \frac{1}{T-t} \sum_{j = t + 1}^T \left(r_j - \ba^\top_j \blambda\right)^+,
\end{align*}
and by our definition, $\ft, \dt, \dtt $ are optimal to problems $\min_{\blambda \in \RR^{m}_{\geq 0}} f_{t, \bb_{t-1}}(\blambda), \min_{\blambda \in \RR^{m}_{\geq 0}} g_{t, \bb_{t-1}}(\blambda)$ and $\min_{\blambda \in \RR^{m}_{\geq 0}} g_{t, \bb_{t-1}-\ba_t}(\blambda)$, respectively. (We generalize the definition of function $f_{\bc}(\cdot)$ in Section \ref{sec:algo} to $f_{t, \bb}(\cdot)$ to be consistent with function $g_{t, \bb}(\cdot)$.)  
We now prove Lemma \ref{lem:concentration-proof-part-1}.
\begin{proof}[Proof of Lemma \ref{lem:concentration-proof-part-1}]
By Lemma \ref{lem: helper-line-integration}, for any $\blambda_1,\blambda_2 \in \RR^{m}_{\geq 0}$,
\begin{align*}
    h_{t,\bb_{t - 1}}(\blambda_1, \ba, r) - h_{t,\bb_{t - 1}}(\blambda_2, \ba, r) = & \phi_{t,\bb_{t - 1}}(\blambda_2, \ba, r)^\top(\blambda_1-\blambda_2) + \int_{\ba^\top\blambda_1}^{\ba^\top\blambda_2}\left(\II_{\left\{ r > v \right\}} - \II_{\left\{ r > \ba^\top \blambda_2 \right\}} \right){\rm d}v,
\end{align*}
applying to functions $f_{t, \bb_{t - 1}}$ and $g_{t, \bb_{t - 1}}$ evaluated at $\ft, \dt$,
\begin{align}
    &f_{t,\bb_{t - 1}}(\dt) - f_{t,\bb_{t - 1}}(\ft)\nonumber\\
    = & \E_{(\ba,r)\sim F}\phi_{t,\bb_{t - 1}}(\ft, \ba, r)^\top(\dt-\ft) + \E_{(\ba,r)\sim F}\int_{\ba^\top\dt}^{\ba^\top\ft}\left(\II_{\left\{ r > v \right\}} - \II_{\left\{ r > \ba^\top \ft \right\}} \right){\rm d}v,\label{eq_pflemfnear_ffd1}\\
    & f_{t,\bb_{t - 1}}(\ft) - f_{t,\bb_{t - 1}}(\dt)\nonumber\\
    = & \E_{(\ba,r)\sim F}\phi_{t,\bb_{t - 1}}(\dt, \ba, r)^\top(\ft-\dt) + \E_{(\ba,r)\sim F}\int_{\ba^\top\ft}^{\ba^\top\dt}\left(\II_{\left\{ r > v \right\}} - \II_{\left\{ r > \ba^\top \dt \right\}} \right) {\rm d}v,\label{eq_pflemfnear_ffd2}
\end{align}
and 
\begin{align}
    &\ \ \  g_{t,\bb_{t - 1}}(\dt) -  g_{t,\bb_{t - 1}}(\ft)\nonumber\\
    = &\ \ \frac{1}{T-t}\sum_{j=t+1}^T \phi_{t,\bb_{t - 1}}(\ft, \ba_j, r_j)^\top(\dt-\ft) + 
    \frac{1}{T-t}\sum_{j=t+1}^T \int_{\ba_j^\top\dt}^{\ba_j^\top\ft}\left(\II_{\left\{ r > v \right\}} - \II_{\left\{ r > \ba_j^\top \ft \right\}} \right){\rm d}v,\label{eq_pflemfnear_ggd1}\\
    &\ \ \  g_{t,\bb_{t - 1}}(\ft) - g_{t,\bb_{t - 1}}(\dt)\nonumber\\
    = &\ \  \frac{1}{T-t}\sum_{j=t+1}^T \phi_{t,\bb_{t - 1}}(\dt, \ba_j, r_j)^\top(\ft-\dt) + 
    \frac{1}{T-t}\sum_{j=t+1}^T \int_{\ba_j^\top\ft}^{\ba_j^\top\dt}\left(\II_{\left\{ r > v \right\}} - \II_{\left\{ r > \ba^\top \dt \right\}} \right){\rm d}v.\label{eq_pflemfnear_ggd2}
\end{align}
We first prove the first bound. To this end, we observe that
\begin{align*}
    &\E_{\ba \sim F^{\ba}}\bigg[\bigg(F^r_{\ba}(\ba^{\top}\ft) - F^r_{\ba}(\ba^{\top}\dt ) \bigg)\bigg(\ba^{\top}\ft - \ba^{\top}\dt\bigg)\bigg]\nonumber\\
    & = \ \E_{\ba\sim F^{\ba}}\int_{\ba^\top\dt}^{\ba^\top\ft}\left(F^r_{\ba}\left(\ba^{\top}\ft \right) - F^r_{\ba}(v ) \right){\rm d}v 
  + \E_{\ba\sim F^{\ba}}\int_{\ba^\top\dt}^{\ba^\top\ft}\left( F^r_{\ba}(v ) - F^r_{\ba}\left(\ba^{\top}\dt \right) \right){\rm d}v.
\end{align*}
We bound the two integrals respectively. Because $\ft$ minimizes $f_{t,\bb_{t - 1}}(\cdot)$ and $\dt$ minimizes $g_{t,\bb_{t - 1}}(\cdot)$, by \eqref{eq_pflemfnear_ffd1} and  \eqref{eq_pflemfnear_ggd1}, we have 
\begin{align}\label{eq_pflembf1}
    0 \ \leq\ & f_{t,\bb_{t - 1}}(\dt) - f_{t,\bb_{t - 1}}(\ft) \nonumber\\
    \leq &\ \ f_{t,\bb_{t - 1}}(\dt) - f_{t,\bb_{t - 1}}(\ft) - (g_{t,\bb_{t - 1}}(\dt) -  g_{t,\bb_{t - 1}}(\ft))\NNN\\
    = &\ \ \E_{(\ba,r)\sim F}\phi_{t,\bb_{t - 1}}(\ft, \ba, r) ^\top(\dt-\ft) -\frac{1}{T-t}\sum_{j=t+1}^T \phi_{t,\bb_{t - 1}}(\ft, \ba_j, r_j)^\top(\dt-\ft)\nonumber\\
    & +\E_{(\ba,r)\sim F}\int_{\ba^\top\dt}^{\ba^\top\ft}\left(\II_{\left\{ r > v \right\}} - \II_{\left\{ r > \ba^\top \ft \right\}} \right){\rm d}v -  
    \frac{1}{T-t}\sum_{j=t+1}^T \int_{\ba_j^\top\dt}^{\ba_j^\top\ft}\left(\II_{\left\{ r > v \right\}} - \II_{\left\{ r > \ba_j^\top \ft \right\}} \right){\rm d}v.
\end{align}
Applying Lemma \ref{lemma: concentration_gradient} and Lemma \ref{lemma:concentration_int} to \eqref{eq_pflembf1}, together with Lemma \ref{lem: helper-cauchy-schwarz}, we have as long as $T-t\geq e^{\frac{\tilde C_1}{32} + 1}$, with probability at least $1-6C \log(T - t)(T-t)^{-2}$ with $C$ a universal constant,
\begin{align}\label{eq_pflembf1X}
    0 \leq & f_{t,\bb_{t - 1}}(\dt) - f_{t,\bb_{t - 1}}(\ft) \nonumber\\
    \leq &  \tilde C_1\sqrt{\frac{\log(T-t)}{T-t}}\bigg(\sqrt{\E_{\ba\sim F^{\ba}}(\ba^\top(\ft-\dt))^2(1-F^r_{\ba}(\ba^\top \ft))}+2\sqrt{\frac{\log(T-t)}{T-t}}\nonumber\\
    & +\sqrt{\E_{\ba \sim F^{\ba}}(\ba^\top\ft - \ba^\top\dt)^2\left|F^r_{\ba}(\ba^\top \ft) - F^r_{\ba}(\ba^\top  \dt)\right|}\bigg)\nonumber\\
    \leq & 2\tilde C_1\sqrt{\frac{\log(T-t)}{T-t}}\bigg(M_1 +\sqrt{\frac{\log(T-t)}{T-t}}\bigg),
\end{align}
where the first inequality is also because $\tilde C_1 \geq \tilde C$, where $\tilde C$ and $\tilde C_1$ are as in Lemma \ref{lemma: concentration_gradient} and Lemma \ref{lemma:concentration_int}, respectively. In \eqref{eq_pflembf1X},
\begin{align*}
    M_1 = \max\left\{\sqrt{\E_{\ba\sim F^{\ba}}(\ba^\top(\ft-\dt))^2\left(1-F^r_{\ba}(\ba^\top \dt)\right)},\sqrt{\E_{\ba\sim F^{\ba}}(\ba^\top(\ft-\dt))^2\left(1-F^r_{\ba}(\ba^\top \ft)\right)}\right\}.
\end{align*}
By Lemma \ref{lem: helper-gradient-non-negativity} with $\blambda = \dt$, we have 
\begin{align}
    &\E_{(\ba,r)\sim F}\phi_{t,\bb_{t - 1}}(\ft, \ba, r)^\top(\dt -\ft)\geq 0.\label{eq_pflemfnear_flarger01}
\end{align}
Plugging \eqref{eq_pflemfnear_flarger01} into \eqref{eq_pflemfnear_ffd1}, we obtain 
\begin{align*}
f_{t,\bb_{t - 1}}(\dt) - f_{t,\bb_{t - 1}}(\ft)&\ \geq\  \E_{(\ba,r)\sim F}\int_{\ba^\top\dt}^{\ba^\top\ft}\left(\II_{\left\{ r > v \right\}} - \II_{\left\{ r > \ba^\top \ft \right\}} \right){\rm d}v\\
& \ = \ \E_{\ba\sim F^{\ba}}\int_{\ba^\top\dt}^{\ba^\top\ft}\left(F^r_{\ba}\left(\ba^{\top}\ft \right) - F^r_{\ba}(v ) \right){\rm d}v,
\end{align*}
which, together with \eqref{eq_pflembf1X}, gives us that with probability at least $1 - \frac{6C \log(T - t)}{(T  - t)^2}$,
\begin{align}\label{eq_pflemfnear_fubfinal1}
    & \E_{\ba\sim F^{\ba}}\int_{\ba^\top\dt}^{\ba^\top\ft}\left(F^r_{\ba}\left(\ba^{\top}\ft \right) - F^r_{\ba}(v ) \right){\rm d}v \leq 2\tilde C_1\sqrt{\frac{\log(T-t)}{T-t}}\bigg(M_1+\sqrt{\frac{\log(T-t)}{T-t}}\bigg).
\end{align}

We next obtain a bound on $\E_{\ba\sim F^{\ba}}\int_{\ba^\top\dt}^{\ba^\top\ft}\left( F^r_{\ba}(v ) - F^r_{\ba}\left(\ba^{\top}\dt \right) \right){\rm d}v.$
Note that for any $\blambda_1,\blambda_2\geq \mathbf{0}$, we have
\begin{align*}
    \int_{\ba^\top\blambda_1}^{\ba^\top\blambda_2}\left(\II_{\left\{r > v\right\}} - \II_{\left\{r > \ba^\top \blambda_2\right\}}\right){\rm d}v \ \geq\  0,
\end{align*}
which, by \eqref{eq_pflemfnear_ffd2}, yields
\begin{align}\label{eq_lem2gradfup1}
& \E_{(\ba,r)\sim F}\phi_{t,\bb_{t - 1}}(\dt, \ba, r)^\top(\ft-\dt)\nonumber\\
\leq &  \E_{(\ba,r)\sim F}\phi_{t,\bb_{t - 1}}(\dt, \ba, r)^\top(\ft-\dt) + \E_{(\ba,r)\sim F}\int_{\ba^\top\ft}^{\ba^\top\dt}\left(\II_{\left\{r > v\right\}} - \II_{\left\{r > \ba^\top \dt\right\}}\right){\rm d}v \nonumber\\
= & f_{t,\bb_{t - 1}}(\ft) - f_{t,\bb_{t - 1}}(\dt) \leq 0.
\end{align}
Lemma \ref{lemma: concentration_gradient} implies that with probability at least $1- \frac{3C\log (T - t)}{(T - t)^2}$, 
\begin{align}
    & \E_{(\ba,r)\sim F}\phi_{t,\bb_{t - 1}}(\dt, \ba, r)^\top(\dt - \ft) - \frac{1}{T-t}\sum_{j=t+1}^T \phi_{t,\bb_{t - 1}}(\dt, \ba_j, r_j)^\top(\dt - \ft)  \nonumber\\
    & \ \ \ \ \ \quad\quad\quad + \frac{1}{T-t}\sum_{j=t+1}^T \phi_{t,\bb_{t - 1}}(\dt, \ba_j, r_j)^\top(\dt - \ft)\nonumber\\
   \leq  & \tilde C_1 \sqrt{\frac{\log(T-t)}{T-t}}\bigg(M_1+\sqrt{\frac{\log(T-t)}{T-t}}\bigg) + \frac{1}{T-t}\sum_{j=t+1}^T \phi_{t,\bb_{t - 1}}(\dt, \ba_j, r_j)^\top(\dt - \ft)\nonumber\\
    \leq & \tilde C_1 \sqrt{\frac{\log(T-t)}{T-t}}\bigg(M_1+\sqrt{\frac{\log(T-t)}{T-t}}\bigg) + \frac{2 m^2 {\bar A} {\bar r}}{(T-t)\underline A}, \label{eq_pflemfnear_granear1}
\end{align}
where the last inequality is by the second inequality in Lemma \ref{lem: helper-gradient-non-negativity} with $\blambda = \ft$. 
 
Therefore, with probability at least $1- \frac{3C\log (T - t)}{(T - t)^2}$,
\begin{align}
    0 \ \leq \  \E_{(\ba,r)\sim F}\phi_{t,\bb_{t - 1}}(\dt, \ba, r)^\top(\dt - \ft)\ \leq \ C_3\sqrt{\frac{\log(T-t)}{T-t}}\bigg(M_1 +\sqrt{\frac{\log(T-t)}{T-t}}\bigg),\label{eq_pflemfnear_graubb1}
\end{align}
where $C_3 = \tilde C_1 + \frac{2 m^2 {\bar A} {\bar r}}{\underline A}$.
By \eqref{eq_lem2gradfup1} and \eqref{eq_pflemfnear_graubb1}, we have with probability at least $1- \frac{3C\log (T - t)}{(T - t)^2}$,
\begin{align}\label{eq_pflemfnear_fubbef21}
    \E_{\ba\sim F^{\ba}}\int_{\ba^\top\dt}^{\ba^\top\ft}\left( F^r_{\ba}(v ) - F^r_{\ba}\left(\ba^{\top}\dt \right) \right){\rm d}v = & \  \E_{(\ba,r)\sim F}\int_{\ba^\top\ft}^{\ba^\top\dt}\left(\II_{\left\{r > v\right\}} - \II_{\left\{r > \ba^\top \dt\right\}}\right){\rm d}v\nonumber\\
    &\leq \ \E_{(\ba,r)\sim F}\phi_{t,\bb_{t - 1}}(\dt, \ba, r)^\top(\dt - \ft)\nonumber\\
    &\leq \  C_3\sqrt{\frac{\log(T-t)}{T-t}}\bigg(M_1+\sqrt{\frac{\log(T-t)}{T-t}}\bigg).
\end{align}
Combining \eqref{eq_pflemfnear_fubfinal1} and \eqref{eq_pflemfnear_fubbef21}, we obtain with probability at least $1- \frac{9C \log (T - t)}{(T - t)^2}$,
\begin{align}\label{eq_pflemfnear_fubfinal}
    & \E_{\ba \sim F^{\ba}}\bigg[\bigg(F^r_{\ba}(\ba^{\top}\ft) - F^r_{\ba}(\ba^{\top}\dt ) \bigg)\bigg(\ba^{\top}\ft - \ba^{\top}\dt\bigg)\bigg]\nonumber\\
    & = \ \E_{\ba\sim F^{\ba}}\int_{\ba^\top\dt}^{\ba^\top\ft}\left(F^r_{\ba}\left(\ba^{\top}\ft \right) - F^r_{\ba}(v ) \right){\rm d}v 
  + \E_{\ba\sim F^{\ba}}\int_{\ba^\top\dt}^{\ba^\top\ft}\left( F^r_{\ba}(v ) - F^r_{\ba}\left(\ba^{\top}\dt \right) \right){\rm d}v\nonumber\\
   & \leq \ (2\tilde C_1+C_3)\sqrt{\frac{\log(T-t)}{T-t}}\bigg(M_1 +\sqrt{\frac{\log(T-t)}{T-t}}\bigg).
\end{align}
Note that 
\begin{align*}
    M_1 \ = \ \sqrt{\E_{\ba \sim F^{\ba}}\bigg[\big(\ba^{\top}\ft - \ba^{\top}\dt\big)^2\max\bigg(\bar F^r_{\ba}(\ba^\top \dt), \bar F^r_{\ba}(\ba^\top \ft)\bigg)\bigg]}.
\end{align*}
We have proven \eqref{eq: concentration_lemma_1} regarding $\dt.$ 

We now turn to  \eqref{eq: concentration_lemma_2}  that involves $\dtt$. The proof is very similar. Similar to \eqref{eq_pflembf1}, we have 
\begin{align}\label{eq_pflembf1_dtt}
    0 \ \leq\ & f_{t,\bb_{t - 1}}(\dtt) - f_{t,\bb_{t - 1}}(\ft) \nonumber\\
    \leq &\ \ f_{t,\bb_{t - 1}}(\dtt) - f_{t,\bb_{t - 1}}(\ft) - (g_{t,\bb_{t - 1}-\ba_t}(\dtt) -  g_{t,\bb_{t - 1}-\ba_t}(\ft))\NNN\\
     &= \ \ \E_{(\ba,r)\sim F}\phi_{t,\bb_{t - 1}}(\ft, \ba, r) ^\top(\dtt-\ft) -\frac{1}{T-t}\sum_{j=t+1}^T \phi_{t,\bb_{t - 1}-\ba_t}(\ft, \ba_j, r_j)^\top(\dtt-\ft)\nonumber\\
    & +\E_{(\ba,r)\sim F}\int_{\ba^\top\dtt}^{\ba^\top\ft}\left(\II_{\left\{ r > v \right\}} - \II_{\left\{ r > \ba^\top \ft \right\}} \right){\rm d}v -  
    \frac{1}{T-t}\sum_{j=t+1}^T \int_{\ba_j^\top\dtt}^{\ba_j^\top\ft}\left(\II_{\left\{ r > v \right\}} - \II_{\left\{ r > \ba_j^\top \ft \right\}} \right){\rm d}v \NNN\\
     & =\ \ \E_{(\ba,r)\sim F}\phi_{t,\bb_{t - 1}}(\ft, \ba, r) ^\top(\dtt-\ft) -\frac{1}{T-t}\sum_{j=t+1}^T \phi_{t,\bb_{t - 1}}(\ft, \ba_j, r_j)^\top(\dtt-\ft)\nonumber\\
    & +\ \frac{1}{T-t}\sum_{j=t+1}^T \phi_{t,\bb_{t - 1}}(\ft, \ba_j, r_j)^\top(\dtt-\ft) - \frac{1}{T-t}\sum_{j=t+1}^T \phi_{t,\bb_{t - 1}-\ba_t}(\ft, \ba_j, r_j)^\top(\dtt-\ft)\nonumber\\
    & +\ \E_{(\ba,r)\sim F}\int_{\ba^\top\dtt}^{\ba^\top\ft}\left(\II_{\left\{ r > v \right\}} - \II_{\left\{ r > \ba^\top \ft \right\}} \right){\rm d}v -  
    \frac{1}{T-t}\sum_{j=t+1}^T \int_{\ba_j^\top\dtt}^{\ba_j^\top\ft}\left(\II_{\left\{ r > v \right\}} - \II_{\left\{ r > \ba_j^\top \ft \right\}} \right){\rm d}v \NNN\\
    &= \ \ \E_{(\ba,r)\sim F}\phi_{t,\bb_{t - 1}}(\ft, \ba, r) ^\top(\dtt-\ft) -\frac{1}{T-t}\sum_{j=t+1}^T \phi_{t,\bb_{t - 1}}(\ft, \ba_j, r_j)^\top(\dtt-\ft)\nonumber\\
    & +\ \E_{(\ba,r)\sim F}\int_{\ba^\top\dtt}^{\ba^\top\ft}\left(\II_{\left\{ r > v \right\}} - \II_{\left\{ r > \ba^\top \ft \right\}} \right){\rm d}v -  
    \frac{1}{T-t}\sum_{j=t+1}^T \int_{\ba_j^\top\dtt}^{\ba_j^\top\ft}\left(\II_{\left\{ r > v \right\}} - \II_{\left\{ r > \ba_j^\top \ft \right\}} \right){\rm d}v \NNN\\
    & \quad\quad\quad\quad + \frac{\ba_t^{\top}(\dtt- \ft)}{T - t}\NNN\\
    & \leq \ \ 2\tilde C_1\bigg(\bar M_1 +\sqrt{\frac{\log(T-t)}{T-t}}\bigg) + \frac{2 m \Au {\bar r}}{\Al(T - t)}\NNN\\
    & \leq\ \  C_4\sqrt{\frac{\log(T-t)}{T-t}}\bigg(\bar M_1 +\sqrt{\frac{\log(T-t)}{T-t}}\bigg),
\end{align}
where the third inequality follows from Lemma \ref{lemma: concentration_gradient} and Lemma \ref{lemma:concentration_int} similar to \eqref{eq_pflembf1X}, and that $\frac{\ba_t^{\top}(\dtt- \ft)}{T - t} \leq \frac{2 m \Au {\bar r}}{\Al(T - t)}$ since $\|\ba_t\|_{\infty} \leq \Au$ and $\|\dtt\|_{\infty}, \|\ft\|_{\infty} \leq \frac{\bar r}{\Al}$ by Corollary \ref{coro: helper-bounded-norm-lambda}. Here $C_4 = 2\tilde C_1 + \frac{2 m \Au {\bar r}}{\Al}$, and 
\begin{align*}
    \bar M_1 = \max\left\{\sqrt{\E_{\ba\sim F^{\ba}}(\ba^\top(\ft-\dtt))^2\left(1-F^r_{\ba}(\ba^\top \dtt)\right)},\sqrt{\E_{\ba\sim F^{\ba}}(\ba^\top(\ft-\dtt))^2\left(1-F^r_{\ba}(\ba^\top \ft)\right)}\right\}.
\end{align*}

Now following nearly identical arguments as in \eqref{eq_pflemfnear_flarger01}, we derive a similar bound as \eqref{eq_pflemfnear_fubfinal1}, that 
\begin{align}\label{eq_pflemfnear_fubfinal1_dtt}
    & \E_{\ba\sim F^{\ba}}\int_{\ba^\top\dtt}^{\ba^\top\ft}\left(F^r_{\ba}\left(\ba^{\top}\ft \right) - F^r_{\ba}(v ) \right){\rm d}v \leq C_4\sqrt{\frac{\log(T-t)}{T-t}}\bigg(\bar M_1+\sqrt{\frac{\log(T-t)}{T-t}}\bigg)
\end{align}
with probability at least $1 - \frac{6C \log(T - t)}{(T  - t)^2}.$

To obtain a bound on $\E_{\ba\sim F^{\ba}}\int_{\ba^\top\dtt}^{\ba^\top\ft}\left( F^r_{\ba}(v ) - F^r_{\ba}\left(\ba^{\top}\dtt \right) \right){\rm d}v$, we follow the argument similar to \eqref{eq_lem2gradfup1}, which yields 
\begin{align}\label{eq_lem2gradfup1_dtt}
&  \E_{(\ba,r)\sim F}\phi_{t,\bb_{t - 1}}(\dtt, \ba, r)^\top(\ft-\dtt)\nonumber\\
& \leq\    \E_{(\ba,r)\sim F}\phi_{t,\bb_{t - 1}}(\dtt, \ba, r)^\top(\ft-\dtt) + \E_{(\ba,r)\sim F}\int_{\ba^\top\ft}^{\ba^\top\dtt}\left(\II_{\left\{r > v\right\}} - \II_{\left\{r > \ba^\top \dtt\right\}}\right){\rm d}v \nonumber\\
&= \ f_{t,\bb_{t - 1}}(\ft) - f_{t,\bb_{t - 1}}(\dtt) \leq 0,
\end{align}
where we used the fact that  $\ft$ is the minimizer of $f_{t, \bb_{t - 1}}(\cdot)$.
Lemma \ref{lemma: concentration_gradient} implies that with probability at least $1- \frac{3C\log (T - t)}{(T - t)^2}$, 
\begin{align}
    & \E_{(\ba,r)\sim F}\phi_{t,\bb_{t - 1}}(\dtt, \ba, r)^\top(\dtt - \ft) - \frac{1}{T-t}\sum_{j=t+1}^T \phi_{t,\bb_{t - 1} - \ba_t}(\dtt, \ba_j, r_j)^\top(\dtt - \ft)  \nonumber\\
    & \ \ \ \  + \frac{1}{T-t}\sum_{j=t+1}^T \phi_{t,\bb_{t - 1} - \ba_t}(\dtt, \ba_j, r_j)^\top(\dtt - \ft)\nonumber\\
    & = \ \  \E_{(\ba,r)\sim F}\phi_{t,\bb_{t - 1}}(\dtt, \ba, r)^\top(\dtt - \ft) - \frac{1}{T-t}\sum_{j=t+1}^T \phi_{t,\bb_{t - 1}}(\dtt, \ba_j, r_j)^\top(\dtt - \ft)  \nonumber\\
    & \ \ \ + \ \ \frac{1}{T-t}\sum_{j=t+1}^T \phi_{t,\bb_{t - 1} }(\dtt, \ba_j, r_j)^\top(\dtt - \ft)  - \frac{1}{T-t}\sum_{j=t+1}^T \phi_{t,\bb_{t - 1} - \ba_t}(\dtt, \ba_j, r_j)^\top(\dtt - \ft) \nonumber\\
    & \ \ \ \  + \frac{1}{T-t}\sum_{j=t+1}^T \phi_{t,\bb_{t - 1} - \ba_t}(\dtt, \ba_j, r_j)^\top(\dtt - \ft)\nonumber\\
    & \leq \ \  \E_{(\ba,r)\sim F}\phi_{t,\bb_{t - 1}}(\dtt, \ba, r)^\top(\dtt - \ft) - \frac{1}{T-t}\sum_{j=t+1}^T \phi_{t,\bb_{t - 1}}(\dtt, \ba_j, r_j)^\top(\dtt - \ft)\nonumber\\
    & \ \ \ \   + \frac{\ba_t^{\top}(\dtt - \ft)}{T - t} + \frac{2 m^2 {\bar A} {\bar r}}{\underline A (T-t)}\nonumber\\ 
    & \ \ \ \quad\quad \leq \ \left(\tilde C_1 + \frac{2 m \Au {\bar r}}{\Al} + \frac{2 m^2 {\bar A} {\bar r}}{\underline A}\right)\sqrt{\frac{\log(T-t)}{T-t}}\bigg(\bar M_1+\sqrt{\frac{\log(T-t)}{T-t}}\bigg),\label{eq_pflemfnear_granear1_dtt}
\end{align}
where the first inequality is by the third inequality in Lemma \ref{lem: helper-gradient-non-negativity} with $\blambda = \ft$, and the last inequality is because $\frac{\ba_t^{\top}(\dtt- \ft)}{T - t} \leq \frac{2 m \Au {\bar r}}{\Al(T - t)}$. 

Combining \eqref{eq_lem2gradfup1_dtt} with \eqref{eq_pflemfnear_granear1_dtt}, we have with probability at least $1 - \frac{3 \log(T - t)}{(T - t)^2}$,
\begin{align}\label{eq_pflemfnear_graubb1_dtt}
    0 \ \leq \ \E_{(\ba,r)\sim F}\phi_{t,\bb_{t - 1}}(\dtt, \ba, r)^\top(\dtt-\ft) \ \leq \ C_5\sqrt{\frac{\log(T-t)}{T-t}}\bigg(\bar M_1+\sqrt{\frac{\log(T-t)}{T-t}}\bigg),
\end{align}
where $C_5 = \tilde C_1 + \frac{2 m \Au {\bar r}}{\Al} + \frac{2 m^2 {\bar A} {\bar r}}{\underline A}$.
By \eqref{eq_lem2gradfup1_dtt} and \eqref{eq_pflemfnear_graubb1_dtt}, we have with probability at least $1- \frac{3C\log (T - t)}{(T - t)^2}$,
\begin{align}\label{eq_pflemfnear_fubbef21_dtt}
    \E_{\ba\sim F^{\ba}}\int_{\ba^\top\dtt}^{\ba^\top\ft}\left( F^r_{\ba}(v ) - F^r_{\ba}\left(\ba^{\top}\dtt \right) \right){\rm d}v & = \  \E_{(\ba,r)\sim F}\int_{\ba^\top\ft}^{\ba^\top\dtt}\left(\II_{\left\{r > v\right\}} - \II_{\left\{r > \ba^\top \dtt\right\}}\right){\rm d}v\nonumber\\
    &\leq \ \E_{(\ba,r)\sim F}\phi_{t,\bb_{t - 1}}(\dtt, \ba, r)^\top(\dtt - \ft)\nonumber\\
    &\leq \  C_5\sqrt{\frac{\log(T-t)}{T-t}}\bigg(\bar M_1+\sqrt{\frac{\log(T-t)}{T-t}}\bigg).
\end{align}
Combining \eqref{eq_pflemfnear_fubfinal1_dtt} and \eqref{eq_pflemfnear_fubbef21_dtt}, we obtain with probability at least $1- \frac{9C \log (T - t)}{(T - t)^2}$,
\begin{align}\label{eq_pflemfnear_fubfinal_dtt}
    & \E_{\ba \sim F^{\ba}}\bigg[\bigg(F^r_{\ba}(\ba^{\top}\ft) - F^r_{\ba}(\ba^{\top}\dtt ) \bigg)\bigg(\ba^{\top}\ft - \ba^{\top}\dtt\bigg)\bigg]\nonumber\\
    & = \ \E_{\ba\sim F^{\ba}}\int_{\ba^\top\dtt}^{\ba^\top\ft}\left(F^r_{\ba}\left(\ba^{\top}\ft \right) - F^r_{\ba}(v ) \right){\rm d}v 
  + \E_{\ba\sim F^{\ba}}\int_{\ba^\top\dtt}^{\ba^\top\ft}\left( F^r_{\ba}(v ) - F^r_{\ba}\left(\ba^{\top}\dtt \right) \right){\rm d}v\nonumber\\
   & \leq \ (C_4+C_5)\sqrt{\frac{\log(T-t)}{T-t}}\bigg(\bar M_1 +\sqrt{\frac{\log(T-t)}{T-t}}\bigg),
\end{align}
where we note
\begin{align*}
    \bar M_1 \ = \ \sqrt{\E_{\ba \sim F^{\ba}}\bigg[\big(\ba^{\top}\ft - \ba^{\top}\dtt\big)^2\max\bigg(\bar F^r_{\ba}(\ba^\top \dtt), \bar F^r_{\ba}(\ba^\top \ft)\bigg)\bigg]},
\end{align*}
$C_4 = 2\tilde C_1 + \frac{2 m \Au {\bar r}}{\Al}$, and $C_5 = \tilde C_1 + \frac{2 m \Au {\bar r}}{\Al} + \frac{2 m^2 {\bar A} {\bar r}}{\underline A}$. We have proven  \eqref{eq: concentration_lemma_2}  regarding $\dtt.$ 

Now combining \eqref{eq: concentration_lemma_1} and \eqref{eq: concentration_lemma_2}, we complete the proof of Lemma \ref{lem:concentration-proof-part-1}. 
\end{proof}

\section{Proof of Lemma~\ref{lemma: concentration_gradient} and Lemma~\ref{lemma:concentration_int}}\label{appendix: proof of concentration helper}
To begin with, we introduce additional notation. For any function $g$ defined on set $\cX$, let $\|g\|_{L_\infty(\cX)} \triangleq \sup_{x \in \cX} |g(x)|$ denote its supremum norm. In cases where $\cX$ is clear from the context, we use the shorthand notation $\|g\|_{\infty}$ instead. For any measure $\cQ$ supported on $\cX$, let $\|g\|_{\cQ} \triangleq \left(\int |g|^2 d\cQ\right)^{1/2}$ denote its $L_2(\cQ)$ norm. 
Following the convention in the theory of empirical processes, we define the \textit{entropy with bracketing} for $L_2(\cQ)$ for a function class $\cG$ as follows (cf. Definition 2.2 of \cite{geer2000empirical}). Let $N_{B}(\delta, \cG, \cQ)$ be the smallest value of $N$ such that there exits pairs of functions $\{(g^L_j, g^U_j)\}_{j = 1}^N$ for which $\|g^U_j - g^L_j\|_{\cQ} \leq \delta$ for all $j = 1, \dots, N$, and such that for any $g \in \cG$ there exits a $j = j(g)$ such that $g^L_j \leq g \leq g^U_j.$ Then $H_{B}(\delta, \cG, \cQ) = \log N_{B}(\delta, \cG, \cQ)$ is called the $\delta$-entropy with bracketing for $\cG$ (for $L_2(\cQ)$-metric). The proof of both Lemma~\ref{lemma: concentration_gradient} and Lemma~\ref{lemma:concentration_int} uses a classic result in the empirical process theory, which we state as Lemma~\ref{lem:helper-concentration-geer} in \Cref{appendix:concentration-helper}.

\begin{proof}[Proof of Lemma \ref{lemma: concentration_gradient}]
    Recall that 
\begin{align*}
    \phi_{t,\bb}(\blambda, \ba, r) = \frac{\partial h_{t,\bb}(\blambda, \ba)}{\partial \blambda} = \frac{1}{T-t}\bb - \ba\II_{\left\{r > \ba^\top \blambda\right\}}.
\end{align*}
Therefore, we have
\begin{align*}
    & \E_{(\ba,r)\sim F}\phi_{t,\bb}(\lama, \ba, r)^\top(\lamb-\lama) - \frac{1}{T-t}\sum_{j=t+1}^T \phi_{t,\bb}(\lama, \ba_j, r_j)^\top(\lamb-\lama)\nonumber\\
    = & \E_{(\ba,r)\sim F}\ba^\top(\lamb-\lama)\II_{\left\{ r > \ba^\top \lama\right\}} - \frac{1}{T-t}\sum_{j=t+1}^T \ba_j^\top(\lamb-\lama)\II_{\left\{r_j > \ba_j^\top \lama\right\}}.
\end{align*}
Define the function class $\cG$ parameterized by $\bp_1, \bp_2, \bp_3, \bp_4$:
\begin{align*}
    \cG = \bigg\{g: g(\ba,r) = \frac{\Al}{4m\bar r \Au} & \bigg(\ba^\top\bp_4~\II_{\left\{r > \ba^\top \bp_2\right\}} -\ba^\top\bp_1\II_{\left\{r > \ba^\top \bp_3\right\}}\bigg), \bp_j\in \Omega, j=1,2,3,4\bigg\}.
\end{align*}
For all $g \in \cG$ and $(\ba, r) \in \cS$, we have 
\begin{align*}
    \left|g(\ba, r)\right| & = \left|\frac{\Al}{4m\bar r \Au}  \bigg(\ba^\top\bp_4~\II_{\left\{r > \ba^\top \bp_2\right\}} -\ba^\top\bp_1\II_{\left\{r > \ba^\top \bp_3\right\}}\bigg)\right|
    \\& \leq \frac{\Al}{4m\bar r \Au}  \bigg(\left|\ba^\top\bp_4~\II_{\left\{r > \ba^\top \bp_2\right\}}\right| + \left|\ba^\top\bp_1\II_{\left\{r > \ba^\top \bp_3\right\}}\right|\bigg)
     \\& \leq \frac{\Al}{4 m\bar r \Au}  \left(\left|\ba^\top\bp_4\right| + \left|\ba^\top\bp_1\right|\right)
     \\& \leq \frac{\Al}{4 m\bar r \Au}  \times 2 m \Au \frac{\bar r}{\Al} \ = \ \frac{1}{2} < 1,
\end{align*}
where in the final inequality we use $\|\blambda\|_{\infty} \leq \frac{\bar r}{\Al}$ for any $\blambda \in \Omega$ and $\|\ba\|_{\infty} \leq \Au, r \leq \bar{r}$ for any $(\ba, r) \in \mathcal{S}.$ In other words, $\|g\|_{\infty} = \|g\|_{L_\infty(\mathcal{S})} \leq \frac{1}{2} < 1$ for any $\bp_j \in \Omega, j = 1,2, 3, 4.$ Our plan is to approximate $\cG$ by its $\delta$-net. For any $\delta > 0$, consider the grid points set $\Omega_0 \subseteq [0,\frac{\bar r}{\Al}]^{4 m}$, where each element of each point $\bq = (\bq^{(1)},\bq^{(2)},\bq^{(3)}, \bq^{(4)})\in \Omega_0$ has the form $k\delta$, with $k = 0,1,...,\lfloor\frac{\bar r}{\delta \Al}\rfloor + 1$, and $\bq^{(1)},\bq^{(2)},\bq^{(3)}, \bq^{(4)}\in \Omega$. Thus, we have $\mathbf{card}(\Omega_0) \leq (2 + \frac{\bar r}{\delta \Al})^{4m}$. Let $N_0 = (2 + \frac{\bar r}{\delta \Al})^{4m}$. 
Let 
\begin{align*}
    \cG_0 = \bigg\{g: g(\ba,r) = \frac{\Al}{4m \bar r \Au} \bigg(\ba^\top\bq^{(4)}~\II_{\left\{r > \ba^\top \bq^{(2)}\right\}}
    -\ba^\top\bq^{(1)}\II_{\left\{r > \ba^\top \bq^{(3)}\right\}}\bigg), \bq^{(j)}\in \Omega_0, j=1,2,3,4\bigg\}.
\end{align*}
For each $g\in \cG$, we are going to find functions $g_L, g_U\in \cG_0$ such that $g_L\leq g \leq g_U$ with $\|g_L-g_U\|_{F}$ bounded by a function of $\delta$ to be determined later, and we recall that $\|\cdot\|_F$ is the $L_2(F)$ norm with respect to distribution $F$ (of $(\ba, r)$). 

Since $g\in \cG$, it has the form 
\begin{align*}
    g(\ba,r) = \frac{\Al}{4m\bar r \Au} & \bigg(\ba^\top\bp_4~\II_{\left\{r > \ba^\top \bp_2\right\}} -\ba^\top\bp_1\II_{\left\{r > \ba^\top \bp_3\right\}}\bigg).
\end{align*}
We choose $\bq_U = (\bq_U^{(1)},\bq_U^{(2)},\bq_U^{(3)}, \bq_U^{(4)})$ such that $\bq_U^{(j)} \leq \bp_j < \bq_U^{(j)} + \delta \mathbf{1}$, for $j=1,2$, and $\max\{\bq_U^{(i)} - \delta \mathbf{1},\mathbf{0}\} < \bp_i \leq \bq_U^{(i)}$, for $ i = 3, 4,$ where the maximum is taken elementwisely. We also choose $\bq_L = (\bq_L^{(1)},\bq_L^{(2)},\bq_L^{(3)}, \bq_L^{(4)})$ such that $\bq_L^{(i)} \leq \bp_i < \bq_L^{(i)} + \delta \mathbf{1}$, and $\max\{\bq_L^{(j)} - \delta \mathbf{1},\mathbf{0}\} < \bp_j\leq \bq_L^{(j)}$, for $j=1,2, i = 3, 4$.

Let 
\begin{align*}
    g_U(\ba,r) &= \frac{\Al}{4m\bar r \Au} \bigg(\ba^\top\bq_U^{(4)}~\II_{\left\{r > \ba^\top \bq_U^{(2)}\right\}}
    -\ba^\top\bq_U^{(1)}\II_{\left\{r > \ba^\top \bq_U^{(3)}\right\}}\bigg),\\
    g_L(\ba,r) &= \frac{\Al}{4m\bar r \Au} \bigg(\ba^\top\bq_L^{(4)}~\II_{\left\{r > \ba^\top \bq_L^{(2)}\right\}}
    -\ba^\top\bq_L^{(1)}\II_{\left\{r > \ba^\top \bq_L^{(3)}\right\}}\bigg).
\end{align*}
By monotonicity and non-negativity of $\ba$, $g_L(\ba,r)\leq g(\ba,r)\leq g_U(\ba,r)$ for all $(\ba, r) \in \cS$ and any $\bp_j \in \Omega, j = 1,2, 3, 4$. We next bound $\|g_U - g_L\|_{F}.$ By definition
\begin{align*}
    \|g_U - g_L\|^2_{F} & = \left(\frac{\Al}{4m\bar r \Au}\right)^2\E_{(\ba, r) \sim F} \left\{ \left(J_1(\ba, r) - J_2(\ba, r)\right)^2\right\}\\
    & \leq 2 \left(\frac{\Al}{4m\bar r \Au}\right)^2 \left(\E_{(\ba, r) \sim F}J_1(\ba, r)^2 + \E_{(\ba, r) \sim F}J_2(\ba, r)^2 \right),
\end{align*}
where 
\begin{align*}
    J_1(\ba,r) = & \ba^\top\bq_U^{(4)} \II_{\left\{r > \ba^\top \bq_U^{(2)}\right\}} - \ba^\top\bq_L^{(4)} \II_{\left\{r > \ba^\top \bq_L^{(2)}\right\}},\\
    J_2(\ba,r) = & \ba^\top\bq_U^{(1)}\II_{\left\{r > \ba^\top \bq_U^{(3)}\right\}} - \ba^\top\bq_L^{(1)}\II_{\left\{r > \ba^\top \bq_L^{(3)}\right\}}.
\end{align*}
For $J_1(\ba, r)$ we have 
\begin{align}
    & \E_{(\ba, r)\sim F} J_1(\ba, r)^2\nonumber\\
   = & \ \E_{(\ba, r) \sim F}\bigg(\ba^\top\bq_U^{(4)}\II_{\left\{r > \ba^\top \bq_U^{(2)}\right\}} - \ba^\top\bq_U^{(4)}\II_{\left\{r > \ba^\top \bq_L^{(2)}\right\}} + \ba^\top\bq_U^{(4)}\II_{\left\{r > \ba^\top \bq_L^{(2)}\right\}} - \ba^\top\bq_L^{(4)}\II_{\left\{r > \ba^\top \bq_L^{(2)}\right\}}\bigg)^2\nonumber\\
    \leq & 2\E_{(\ba, r) \sim F}\bigg(\bigg(\ba^\top\bq_U^{(4)}\II_{\left\{r > \ba^\top \bq_U^{(2)}\right\}} - \ba^\top\bq_U^{(4)}\II_{\left\{r > \ba^\top \bq_L^{(2)}\right\}}\bigg)^2 + \bigg(\ba^\top\bq_U^{(4)}\II_{\left\{r > \ba^\top \bq_L^{(2)}\right\}} - \ba^\top\bq_L^{(4)}\II_{\left\{r > \ba^\top \bq_L^{(2)}\right\}}\bigg)^2\bigg)\nonumber\\
    = & 2\E_{(\ba, r) \sim F} (J_{11}(\ba,r)^2 + J_{12}(\ba,r)^2)\NNN.
\end{align}
By our choice of $\bq_L$ and $\bq_U$, $\|\bq_L - \bq_U\|_{\infty} \leq \|\bq_L - (\bp_1, \bp_2, \bp_3, \bp_4)\|_{\infty} + \|\bq_U - (\bp_1, \bp_2, \bp_3, \bp_4)\|_{\infty} \leq 2 \delta.$ For $J_{11}(\ba, r)$, we have 
\begin{align}
    \E_{(\ba, r)\sim F} J_{11}(\ba, r)^2 & \leq \  \left(m\Au \frac{\bar r}{\Al}\right)^2\E_{(\ba, r) \sim F}\left(\II_{\left\{r > \ba^\top \bq_U^{(2)}\right\}} - \II_{\left\{r > \ba^\top \bq_L^{(2)}\right\}}\right)^2\nonumber\\
    &= \  \left(m\Au \frac{\bar r}{\Al}\right)^2\E_{(\ba, r) \sim F}\left|\II_{\left\{r > \ba^\top \bq_U^{(2)}\right\}} - \II_{\left\{r > \ba^\top \bq_L^{(2)}\right\}}\right|\nonumber\\
    &= \  \left(m\Au \frac{\bar r}{\Al}\right)^2 \E_{\ba\sim F^{\ba}}\left|F^r_{\ba}(\ba^\top \bq_L^{(2)}) - F^r_{\ba}(\ba^\top \bq_U^{(2)})\right|\nonumber\\
    &\leq c_{\nu} \left(m\Au \frac{\bar r}{\Al}\right)^2 \E_{\ba\sim F^{\ba}}|\ba^\top \bq_L^{(2)} - \ba^\top \bq_U^{(2)}|^\nu \quad\quad\quad\quad \textrm{(under either Assumptions)}\NNN\\
    &\leq \ c_{\nu} \left(m\Au \frac{\bar r}{\Al}\right)^2 (m\Au)^\nu 2^\nu \delta^\nu\NNN,
\end{align}
where in the last equality we use $\ba \in [\Al, \Au]^m $ and $\|\bq_U - \bq_L\|_{\infty} \leq 2 \delta.$ Now consider the second term $J_{12}(\ba, r)$, we have 
\[
\E_{(\ba, r) \sim F} J_{12}(\ba, r)^2 \ \leq\ \E_{\ba \sim F^{\ba}}\left(\ba^{\top} \bq_U^{(4)} - \ba^{\top} \bq_L^{(4)}\right)^2 \ \leq \ 4 m^2 \Au^2 \delta^2.
\]
Therefore, the term $\E_{(\ba, r) \sim F}J_1(a, r)^2$ can be bounded by
\[
2\left(c_{\nu}\left(m\Au \frac{\bar r}{\Al}\right)^2 (m\Au)^\nu 2^\nu \delta^\nu + 4 m^2 \Au^2 \delta^2\right).
\]
The term $\E_{(\ba, r) \sim F}J_2(a, r)^2$ can be bounded nearly identically. Combining the above, we have 
\begin{align*}
    \|g_U- g_L\|^2_F & \leq 2 \left(\frac{\Al}{4m\bar r \Au}\right)^2 \left(\E_{(\ba, r) \sim F}J_1(\ba, r)^2 + \E_{(\ba, r) \sim F}J_2(\ba, r)^2 \right)\\
    & \leq 4 \left(\frac{\Al}{4m\bar r \Au}\right)^2  \left( 2c_{\nu}\left(m\Au \frac{\bar r}{\Al}\right)^2 (m\Au)^\nu 2^\nu \delta^\nu +  8 m^2\Au^2\delta^2 \right)\\
    & \leq c_{\nu} (2 m \Au)^{\nu} \delta^{\nu} +   \frac{2 \Al^2}{{\bar r}^2} \delta^2  \  \leq \ \left(c_{\nu}^{\frac{1}{2}}(2 m \Au)^{\frac{\nu}{2}} +  \frac{2 \Al}{\bar r}\right)^2 \delta^{\min\{\nu, 2\}}.
\end{align*}
As a result, $\|g_U- g_L\|_F \leq  \left(c_{\nu}^{\frac{1}{2}}(2 m \Au)^{\frac{\nu}{2}} + \frac{ 2\Al}{\bar r}\right)\delta^{\min\{\nu/2, 1\}}.$  Consider the set of function pairs 
\[
\{(g_L(\bq_1), g_U(\bq_2)), \bq_1, \bq_2 \in \Omega_0\}.
\]
The cardinality of this set is $\mathbf{card}(\Omega_0)^2 \leq N_0^2.$ By the previous argument, this set of function pairs is enough to $\delta'$-cover $\cG$ where $\delta' = \left(c_{\nu}^{\frac{1}{2}}(2 m \Au)^{\frac{\nu}{2}} + \frac{2 \Al}{\bar r}\right)\delta^{\min\{\nu/2, 1\}}$ in the following precise sense. For any $g \in \cG$, we can find a pair of $\bq_L$ and $\bq_U$ in $\Omega_0$ (given precisely earlier), such that the associated $g_L$ and $g_U$ satisfy $g_L \leq g \leq g_U$, and furthermore $\|g_U- g_L\|_F \leq  \left(c_{\nu}^{\frac{1}{2}}(2 m \Au)^{\frac{\nu}{2}} + \frac{2 \Al}{\bar r}\right)\delta^{\min\{\nu/2, 1\}}.$ Denote by $C_7 \triangleq c_{\nu}^{\frac{1}{2}} (2 m \Au)^{\frac{\nu}{2}} + \frac{2 \Al}{\bar r}.$  By definition of the $\delta$-entropy with bracketing, we must have 
\begin{align}
    H_B(C_7\delta^{\min\{\nu/2,1\}}, \cG ,F) \leq  \log(N_0^2) =  8 m \log\left(2 + \frac{\bar r}{\delta \Al}\right),
\end{align}
which implies
\begin{align*}
    H_B(u, \cG ,F) &\leq 8 m \log\left(2 + \frac{\bar r C_7^{\frac{1}{\min\{\nu/2,1\}}}}{u^{\frac{1}{\min\{\nu/2,1\}}}\Al}\right) \leq 8 m \log\left(\left(2 + \frac{{\bar r} C_7}{u \Al}\right)^{\frac{1}{\min\{\nu/2,1\}}}\right) \\
    & = \frac{8 m}{\min\{\nu/2,1\}} \log\left(2 + \frac{\bar{r} C_7}{\Al u}\right).
\end{align*}
Thus we have, when $\delta\leq \frac{\bar{r} C_7}{\Al}$,
\begin{align*}
    \int_{0}^\delta H_B(u,\cG,F)^{1/2}{\rm d}u \leq & \left(\frac{8 m}{\min\{\nu/2,1\}}\right)^{1/2}\int_0^{\delta}\sqrt{\log\left(2 + \frac{\bar{r} C_7}{\Al  u}\right)} {\rm d} u
    \leq  C_3\delta\sqrt{\log \left(1+\frac{C_4}{\delta}\right)},
\end{align*}
where $C_3 = 2\left(\frac{8 m}{\min\{\nu/2,1\}}\right)^{1/2}$ and $C_4 = \frac{\bar{r} C_7}{\Al} = c_{\nu}^{\frac{1}{2}} \frac{\bar{r}}{\Al}(2 m \Au)^{\frac{\nu}{2}} + 2 .$
Here the integral is bounded by
\begin{align*}
    \int_0^{\delta}\sqrt{\log\left(2 + \frac{\bar{r} C_7}{\Al u}\right)} {\rm d} u &= \ \sum_{s = 1}^{\infty} \int_{2^{-s}\delta}^{2^{-(s-1)}\delta }\sqrt{\log\left(2 + \frac{\bar{r} C_7}{\Al u}\right)} {\rm d} u\\
    &\leq  \sum_{s = 1}^{\infty} \int_{2^{-s}\delta}^{2^{-(s-1)}\delta} \sqrt{\log\left(2 + \frac{\bar{r} C_7}{\Al 2^{- s}\delta}\right)} {\rm d} u\\
    &= \  \sum_{s = 1}^{\infty} 2^{-(s-1)}\delta \sqrt{\log\left(2 + \frac{\bar{r} C_7}{\Al 2^{- s}\delta}\right)}\\
    &\leq \  \delta \sum_{s = 1}^{\infty} 2^{-(s-1)}\sqrt{ s\log{2} + \log\left(1 + \frac{\bar{r} C_7}{\Al\delta}\right)}\\
 &\leq \ 2 \delta \sqrt{\log\left(1 + \frac{\bar{r} C_7}{\Al\delta}\right)} 
 \sum_{s = 1}^{\infty} s 2^{-(s-1)},
\end{align*}
where the infinite summation converges to $1$, and the last inequality holds when $\delta\leq \frac{\bar{r} C_7}{\Al}$.

Now we are ready to apply Lemma \ref{lem:helper-concentration-geer}. Recall that $\sup_{g \in \cG} \|g\|_{\infty} \leq \frac{1}{2},$ which implies $\sup_{g \in \cG} \|g\|_{F} \leq \frac{1}{2} < 1.$ Let $\delta_{T - t} \triangleq \sqrt{\frac{\log(T - t)}{T - t}}.$ Consider the following partition of $\cG$:
\begin{align*}
    \cG_s = & \bigg\{g\in \cG: 2^s \delta_{T-t} < \|g\|_F \leq 2^{s+1} \delta_{T-t}\bigg\},\nonumber\\
    \cG_c = & \bigg\{g\in \cG: \|g\|_F \leq \delta_{T-t}\bigg\},
\end{align*}
with $s=1,...,S$ and $S = \lfloor\log(\delta_{T-t}^{-1})/\log 2\rfloor + 1$. Clearly $\cG = (\cup_{s=1}^S \cG_s)\cup \cG_c$. We apply the peeling device. 

Consider $\cG_c$ first. Take $n = T - t, K = 1, R = \delta_{T - t} = \sqrt{\frac{\log(T - t)}{T - t}}$. Denote constant $v = 10 C^2 (1 + C_3)^2(1 + C_4)^2 + 2$ where $C_3, C_4$ are specified above and $C$ is the universal constant as in Lemma \ref{lem:helper-concentration-geer}. We further take $a = v\frac{\log (T - t)}{\sqrt{T - t}}, C_0 = \frac{v}{2(1+ C_3)(1 + C_4)}$ and $ C_1 = 2 v.$ We first check the conditions and then apply Lemma \ref{lem:helper-concentration-geer}. Within the function class $\cG_c$, we have $\sup_{g \in \cG_c} \|g\|_{\infty}\leq \sup_{g \in \cG} \|g\|_{\infty} \leq K = 1$ and $\sup_{g \in \cG_c} \|g\|_{F} \leq \delta_{T - t} = R.$ Next we check the conditions \eqref{eq: Thm511eq1} - \eqref{eq: Thm511eq4}. Condition \eqref{eq: Thm511eq1} is $a \leq C_1 \sqrt{n} R^2/(2 K).$ We have that 
\[
C_1 \frac{\sqrt{T - t} R^2 }{2 K} = v\frac{\log(T - t) }{ \sqrt{T -t}}  = a,
\]
hence condition \eqref{eq: Thm511eq1} holds true.  For condition \eqref{eq: Thm511eq2}, we need $a \leq 8 \sqrt{2(T - t)} R$. We have
\begin{align*}
    8 \sqrt{2(T - t)} R = 8 \sqrt{2\log(T - t)} \geq v \frac{\log(T - t)}{\sqrt{T - t}} = a,
\end{align*}
for all $T - t \geq e^{\frac{ v^2}{32} + 1}.$ Hence condition \eqref{eq: Thm511eq2} holds true. Condition \eqref{eq: Thm511eq3} is 
\[
a \geq C_0 \max\left\{\int_{0}^{\sqrt{2}R} H_B(u/\sqrt{2}, \cG_c, F)^{1/2} {\rm d} u, \sqrt{2} R\right\}.
\]
By our previous calculation,
\begin{align}\label{eq_pf_concen_entropybound}
 C_0\max \left\{\int_0^{\sqrt{2}R}  H_B(u/\sqrt{2},\cG_c, F)^{1/2}{\rm d}u,\sqrt{2}R\right\} & \leq C_0\max\left\{\int_0^{\sqrt{2}R} H_B(u/\sqrt{2},\cG, F)^{1/2}{\rm d}u,\sqrt{2}R\right\}    \nonumber \\
&= \sqrt{2} C_0\max\left\{\int_0^{R} H_B(\delta,\cG, F)^{1/2}{\rm d}\delta, R\right\}\nonumber \\
&\leq  \sqrt{2} C_0 \max\left\{C_3 R \sqrt{\log\left(1 + \frac{C_4}{R}\right)}, R\right\},
\end{align}
when $\delta_{T_t}\leq \frac{\bar{r} C_7}{\Al}$, which automatically holds since $\frac{\bar{r} C_7}{\Al} \geq 2\geq K \geq R$.
For the first term, we have 
\begin{align*}
    \sqrt{2}C_0  C_3 R \sqrt{\log\left(1 + \frac{C_4}{R}\right)} &=\sqrt{2}C_0  C_3 \sqrt{\frac{\log(T - t)}{T - t}} \sqrt{\log\left(1 + C_4\sqrt{\frac{T - t}{\log(T - t)}}\right)},\\
    & \leq \sqrt{2}C_0  C_3 \sqrt{\frac{\log(T - t)}{T - t}} \sqrt{\log\left((1 + C_4)(T - t)\right)},\\
    & \leq \sqrt{2}C_0  C_3(1 + C_4) \frac{\log(T - t)}{\sqrt{T - t}}\\
    & = \frac{\sqrt{2} v}{2(1 + C_3)(1 + C_4)}C_3(1 + C_4)\frac{\log(T - t)}{\sqrt{T - t}}\\
    & \leq  v \frac{\log(T - t)}{\sqrt{T - t}} = a.
\end{align*}
For the second term, we have
\begin{align*}
    \sqrt{2} C_0 R & = \sqrt{2} C_0 \sqrt{\frac{\log(T - t)}{T- t}}  \leq v \sqrt{\frac{\log(T - t)}{T- t}} \leq  v \frac{\log(T - t)}{\sqrt{T - t}} = a,
\end{align*}
for $ T - t > 3.$ Thus condition \eqref{eq: Thm511eq3} holds true. Finally, for condition \eqref{eq: Thm511eq4} we need to verify that $C_0^2 \geq C^2(C_1 + 1).$ We have 
\begin{align}\label{lem_con_Con511eq41X}
C^2(C_1 + 1) = C^2(2 v + 1) &= C^2(20 C^2 (1 + C_3)^2(1 + C_4)^2 + 5) \nonumber\\
&\leq\ 25 C^4 (1 + C_3)^2(1 + C_4)^2 \nonumber\\
&\leq\ \left(\frac{10C^2(1+C_3)^2(1 + C_4)^2 + 2}{2(1+C_3)(1 + C_4)}\right)^2= C_0^2.
\end{align}

With all four conditions checked, Lemma \ref{lem:helper-concentration-geer} implies
\begin{align*}
     &\P\left(\sup_{g\in \cG_c}\sqrt{T-t}\left|\frac{1}{T-t}\sum_{j=t+1}^{T} g(\ba_j,r_j) - \E_{(\ba, r)\sim F}g(\ba,r)\right|\geq a\right)\nonumber\\
    &\leq \quad\quad C\exp\left(-\frac{a^2}{2 C^2(C_1 + 1) R^2}\right),\\
    &= \quad\quad C\exp\left(- \frac{\left(10 C^2 (1 + C_3)^2(1 + C_4)^2 + 2\right)^2}{2 C^2 (20 C^2 (1 + C_3)^2(1 + C_4)^2 + 5)}\log(T - t)\right)\\
    &\leq \quad\quad C\exp\left(-2 \log(T - t)\right) = \frac{C}{(T - t)^2}.
\end{align*}
Combining the above, we may conclude that 
\begin{align}\label{eq: gradient-peeling-Gc}
\P\left(\sup_{g\in \cG_c}\left|\frac{1}{T-t}\sum_{j=t+1}^{T} g(\ba_j,r_j) - \E_{(\ba, r)\sim F}g(\ba,r)\right|\geq v \frac{\log(T - t)}{T - t}\right) \ \leq  \frac{C}{(T - t)^2},    
\end{align}
as long as $T - t \geq \max\left\{e^{\frac{v^2}{32} +1}, 3\right\} = e^{\frac{v^2}{32} +1},$ where $v$ is an absolute constant depending only on problem primitives and independent of $T - t$:
\begin{align}\label{eq_pf_concentra_vdef}    
v = 10 C^2 (1 + C_3)^2(1 + C_4)^2 + 2 = 10 C^2 \left(1 + 2\left(\frac{8 m}{\min\{\nu/2,1\}}\right)^{1/2}\right)^2\left( c_{\nu}^{\frac{1}{2}}\frac{\bar{r}}{\Al}(2 m \Au)^{\frac{\nu}{2}} + 3\right)^2 + 2,
\end{align}
and $e^{\frac{v^2}{32} +1} \geq e^{\frac{4}{32} +1} >3$.

Next, we consider $\cup_{s=1}^S \cG_s$. For any $c>0$, the union bound implies that
\begin{align}\label{eq: union-Gs1}
    & \P\left(\sup_{g\in \cup_{s=1}^S \cG_s}\frac{\left|\frac{1}{T-t}\sum_{j=t+1}^{T} g(\ba_j,r_j) - \E_{(\ba, r)\sim F}g(\ba,r)\right|}{\|g\|_F}\geq c\right)\nonumber\\
    \leq & \sum_{s=1}^S \P\left(\sup_{g\in \cG_s}\frac{\left|\frac{1}{T-t}\sum_{j=t+1}^{T} g(\ba_j,r_j) - \E_{(\ba, r)\sim F}g(\ba,r)\right|}{\|g\|_F}\geq c\right)\nonumber\\
    \leq & \sum_{s=1}^S \P\left(\sup_{g\in \cG_s}\left|\frac{1}{T-t}\sum_{j=t+1}^{T} g(\ba_j,r_j) - \E_{(\ba, r)\sim F}g(\ba,r)\right|\geq 2^s\delta_{T-t}c\right),
\end{align}
where the second inequality is because for any $g\in \cG_s$, $\|g\|_F\geq 2^s\delta_{T-t}$, $s=1,\ldots,S$. In the following, we bound the tail probability 
\begin{align*}
    \P\left(\sup_{g\in \cG_s}\left|\frac{1}{T-t}\sum_{j=t+1}^{T} g(\ba_j,r_j) - \E_{(\ba, r)\sim F}g(\ba,r)\right|\geq 2^s\delta_{T-t}c\right)
\end{align*}
for each $\cG_s$, $s=1,\ldots,S$.

Before applying Lemma \ref{lem:helper-concentration-geer}, we first check the conditions \eqref{eq: Thm511eq1} - \eqref{eq: Thm511eq4} hold. Take $n = T- t, K_s = 1, R_s = 2^{s + 1} \delta_{T- t} = 2^{s+ 1} \sqrt{\frac{\log(T - t)}{T - t}}.$ Recall we have set constant $v$ as in \eqref{eq_pf_concentra_vdef}, and we take $a = v 2^{s + 1} \frac{\log(T - t)}{\sqrt{T - t}}$, $C_0 = \frac{v}{2(1 + C_3)(1 + C_4)}$ and $C_1 = 2 v.$ Within the function class $\cG_s$, we have $\sup_{g \in \cG_s} \|g\|_{\infty}\leq \sup_{g \in \cG} \|g\|_{\infty} \leq K_s = 1$ and $\sup_{g \in \cG_s} \|g\|_{F} \leq R_s.$ Next we check conditions \eqref{eq: Thm511eq1}-\eqref{eq: Thm511eq4}. Condition \eqref{eq: Thm511eq1} is $a \leq C_1 \sqrt{n} R_s^2/(2 K_s).$ We have that 
\[
C_1 \frac{\sqrt{T - t} R_s^2 }{2 K_s} = v 2^{2(s + 1)} \frac{\log(T - t) }{ \sqrt{T -t}} \geq v 2^{s + 1} \frac{\log(T - t) }{ \sqrt{T -t}} = a,
\]
hence condition \eqref{eq: Thm511eq1} holds true.  For condition \eqref{eq: Thm511eq2}, we need $a \leq 8 \sqrt{2(T - t)} R_s$. We have
\begin{align*}
    8 \sqrt{2(T - t)} R_s =  2^{s + 4}\sqrt{2\log(T - t)} \geq v 2^{s + 1} \frac{\log(T - t)}{\sqrt{T - t}} = a,
\end{align*}
for all $T - t \geq e^{\frac{ v^2}{32} + 1}.$ Hence condition \eqref{eq: Thm511eq2} holds true. Condition \eqref{eq: Thm511eq3} is 
\[
a \geq C_0 \max\left\{\int_{0}^{\sqrt{2}R_s} H_B(u/\sqrt{2}, \cG_s, F)^{1/2} {\rm d} u, \sqrt{2} R_s\right\}.
\]
By our previous calculation as in \eqref{eq_pf_concen_entropybound},
\begin{align*}
 C_0\max& \left\{\int_0^{\sqrt{2}R_s}  H_B(u/\sqrt{2},\cG_s, F)^{1/2}{\rm d}u,\sqrt{2}R_s\right\} 
\leq  \sqrt{2} C_0 \max\left\{C_3 R_s \sqrt{\log\left(1 + \frac{C_4}{R_s}\right)}, R_s\right\},
\end{align*}
when $\delta_{T_t}\leq \frac{\bar{r} C_7}{\Al}$, which automatically holds since $\frac{\bar{r} C_7}{\Al} \geq 2\geq K_s \geq R_s$. For the first term, we have 
\begin{align*}
    \sqrt{2}C_0  C_3 R_s \sqrt{\log\left(1 + \frac{C_4}{R_s}\right)} &=\sqrt{2}C_0  C_3 2^{s + 1}\sqrt{\frac{\log(T - t)}{T - t}} \sqrt{\log\left(1 + C_42^{-s - 1}\sqrt{\frac{T - t}{\log(T - t)}}\right)},\\
    & \leq \sqrt{2}C_0  C_3 2^{s + 1}\sqrt{\frac{\log(T - t)}{T - t}} \sqrt{\log\left((1 + C_4)(T - t)\right)},\\
    & \leq \sqrt{2}C_0  C_3(1 + C_4)2^{s + 1} \frac{\log(T - t)}{\sqrt{T - t}}\\
    & = \frac{\sqrt{2} v}{2(1 + C_3)(1 + C_4)}C_3(1 + C_4)2^{s  +1}\frac{\log(T - t)}{\sqrt{T - t}}\\
    & \leq  v 2^{s  +1}\frac{\log(T - t)}{\sqrt{T - t}} = a.
\end{align*}
For the second term, we have
\begin{align*}
    \sqrt{2} C_0 R_s & = \sqrt{2} C_02^{s + 1} \sqrt{\frac{\log(T - t)}{T- t}}  \leq v 2^{s + 1}\sqrt{\frac{\log(T - t)}{T- t}} \leq  v 2^{s + 1} \frac{\log(T - t)}{\sqrt{T - t}} = a,
\end{align*}
for $ T - t > 3.$ Thus condition \eqref{eq: Thm511eq3} holds true. Finally, for condition \eqref{eq: Thm511eq4} we need to verify that $C_0^2 \geq C^2(C_1 + 1)$, which can be done identically as in \eqref{lem_con_Con511eq41X}. 


With all four conditions checked, Lemma \ref{lem:helper-concentration-geer} implies
\begin{align*}
     &\P\left(\sup_{g\in \cG_s}\sqrt{T-t}\left|\frac{1}{T-t}\sum_{j=t+1}^{T} g(\ba_j,r_j) - \E_{(\ba, r)\sim F}g(\ba,r)\right|\geq a\right)\nonumber\\
    &\leq \quad\quad C\exp\left(-\frac{a^2}{2 C^2(C_1 + 1) R_s^2}\right),\\
    &= \quad\quad C\exp\left(- \frac{\left(10 C^2 (1 + C_3)^2(1 + C_4)^2 + 2\right)^2}{2 C^2 (20 C^2 (1 + C_3)^2(1 + C_4)^2 + 5)}\log(T - t)\right)\\
    &\leq \quad\quad C\exp\left(-2 \log(T - t)\right) = \frac{C}{(T - t)^2}.
\end{align*}
Recall $a = v 2^{s+ 1} \frac{\log(T - t)}{\sqrt{T - t}} = v \sqrt{\log(T - t)} 2^{s + 1} \delta_{T - t}.$ The tail bound can be equivalently written as
\begin{align*}
\P\left(\sup_{g\in \cG_s}\left|\frac{1}{T-t}\sum_{j=t+1}^{T} g(\ba_j,r_j) - \E_{(\ba, r)\sim F}g(\ba,r)\right|\geq v\sqrt{\frac{\log(T  - t)}{T - t}} 2^{s + 1} \delta_{T - t}\right) \ \leq  \frac{C}{(T - t)^2},    
\end{align*}
as long as $T - t \geq \max\left\{e^{\frac{v^2}{32} +1}, 3\right\} = e^{\frac{v^2}{32} +1},$ where $v$ is as in \eqref{eq_pf_concentra_vdef}.
Now we plug the above back into \eqref{eq: union-Gs1} with $c = 2v\sqrt{\frac{\log(T - t)}{T- t}}$, and conclude that for $T - t \geq e^{\frac{v^2}{32} +1},$ it holds true that
\begin{align}\label{eq: gradient-peeling-Gs}
& \P\left(\sup_{g\in \cup_{s=1}^S \cG_s} \frac{\left|\frac{1}{T-t}\sum_{j=t+1}^{T} g(\ba_j,r_j) - \E_{(\ba, r)\sim F}g(\ba,r)\right|}{\|g\|_F}\geq 2v\sqrt{\frac{\log(T  - t)}{T - t}} \right) \NNN\\
&\quad\quad\quad\ \leq  \ \frac{CS}{(T - t)^2}\ \leq \ \frac{2C \log(T - t)}{(T  - t)^2},
\end{align}
where the second inequality is because $S = \lfloor\log(\delta_{T-t}^{-1})/\log 2\rfloor + 1 \leq 2 \log(T - t)$ for $T - t \geq e^{\frac{v^2}{32} +1}.$ 

Combining tail bounds \eqref{eq: gradient-peeling-Gc} and \eqref{eq: gradient-peeling-Gs} with another union bound, we conclude that
\begin{align*}
   & \P\left(\left|\frac{1}{T-t}\sum_{j=t+1}^{T} g(\ba_j,r_j) - \E_{(\ba, r)\sim F}g(\ba,r)\right|\geq 2v\sqrt{\frac{\log(T  - t)}{T - t}}{\|g\|_F}  +  v \frac{\log(T - t)}{T - t}, \forall g\in \cG \right) \NNN\\ 
   = & \ \P\left( {\left|\frac{1}{T-t}\sum_{j=t+1}^{T} g(\ba_j,r_j) - \E_{(\ba, r)\sim F}g(\ba,r)\right|}\geq 2v\sqrt{\frac{\log(T  - t)}{T - t}}{\|g\|_F}  +  v \frac{\log(T - t)}{T - t}, \forall g\in (\cup_{s=1}^S \cG_s)\cup \cG_c \right) \NNN\\ 
   \leq & \ \P\left({\left|\frac{1}{T-t}\sum_{j=t+1}^{T} g(\ba_j,r_j) - \E_{(\ba, r)\sim F}g(\ba,r)\right|}\geq 2v\sqrt{\frac{\log(T  - t)}{T - t}}{\|g\|_F}  +  v \frac{\log(T - t)}{T - t}, \forall g\in \cG_c\right) \NNN\\
   & \quad \quad \quad \quad + \ P\left({\left|\frac{1}{T-t}\sum_{j=t+1}^{T} g(\ba_j,r_j) - \E_{(\ba, r)\sim F}g(\ba,r)\right|}\geq 2v\sqrt{\frac{\log(T  - t)}{T - t}}{\|g\|_F}  +  v \frac{\log(T - t)}{T - t}, \forall g\in \cup_{s=1}^S \cG_s \right) \NNN\\
    & \ \leq \ \P\left(\sup_{g \in  \cG_c } \left|\frac{1}{T-t}\sum_{j=t+1}^{T} g(\ba_j,r_j) - \E_{(\ba, r)\sim F}g(\ba,r)\right| \geq v\sqrt{ \frac{\log(T - t)}{T - t}} \right)\\
    & \quad \quad \quad \quad \quad \quad \quad \quad + P\left(\sup_{g\in \cup_{s=1}^S \cG_s} \frac{\left|\frac{1}{T-t}\sum_{j=t+1}^{T} g(\ba_j,r_j) - \E_{(\ba, r)\sim F}g(\ba,r)\right|}{\|g\|_F}\geq 2v\sqrt{\frac{\log(T  - t)}{T - t}} \right) \\
    & \ \leq \   \frac{C}{(T - t)^2} +  \frac{2C\log(T - t)}{(T - t)^2} \ = \ \frac{3C\log(T - t)}{(T - t)^2},
\end{align*}
for $T - t \geq e^{\frac{v^2}{32} +1}$. Consequently, we conclude that
\begin{align}\label{eq: concetration-gradient-ultimate}
    {\left|\frac{1}{T-t}\sum_{j=t+1}^{T} g(\ba_j,r_j) - \E_{(\ba, r)\sim F}g(\ba,r)\right|} \leq \ 2v\sqrt{\frac{\log(T  - t)}{T - t}} \left(\|g\|_F  +  \sqrt{\frac{\log(T  - t)}{T - t}}\right), \forall g\in \cG
\end{align}
with probability at least $1 - \frac{3 C \log(T - t)}{(T -t)^2}$ for $T - t \geq e^{\frac{v^2}{32} +1}$. Note that 
\begin{align*}
g(\ba, r) = \frac{\Al}{4 m \bar r \Au}\bigg(\ba^\top(\lamb-\lama) \II_{\left\{r > \ba^\top \lama\right\}}\bigg) \in \cG
\end{align*}
with $\bp_4 = \lamb$ and $\bp_j = \lama$ for $j = 1,2,3.$ Also, 
\begin{align*}
&\|\ba^{\top} \left(\lamb - \lama\right)\II_{\{r > \ba^{\top} \lama\}}\|_F = \left(\E_{(\ba, r)\sim F}\left(\ba^{\top} \left(\lamb - \lama\right)\II_{\{r > \ba^{\top} \lama\}}\right)^2\right)^{\frac{1}{2}}\\
& \ \ \ \ \ \ \ = \ \sqrt{\E_{\ba \sim F^{\ba}}\left\{(\ba^{\top}(\lamb- \lama))^2(1 - F^r_{\ba}(\ba^{\top}\lama))\right\}}.   
\end{align*}
The desired result follows from \eqref{eq: concetration-gradient-ultimate} with
\[
\tilde C = 2\left(\frac{4 m \bar r \Au}{\Al}+ 1\right) \left(10 C^2 \left(1 + 2\left(\frac{8 m}{\min\{\nu/2,1\}}\right)^{1/2}\right)^2\left( c_{\nu}^{\frac{1}{2}}\frac{\bar{r}}{\Al}(2 m \Au)^{\frac{\nu}{2}} + 3\right)^2 + 2\right).
\]
\end{proof}

\begin{proof}[Proof of Lemma \ref{lemma:concentration_int}]
 The proof of Lemma \ref{lemma:concentration_int} is similar to the proof of Lemma \ref{lemma: concentration_gradient}. To start with, define the function class parameterized by $\bp = \bp_j \in \Omega,\  j = 1, \dots, 5$:
\begin{align*}
    \cG = \bigg\{g: g(\ba,r) = \frac{\Al}{4m\Au\bar r}\left( \int_{\ba^\top\bp_1}^{\ba^\top\bp_2}\II_{\{r > u\}} {\rm d}u - \left(\ba^\top\bp_3 - \ba^\top\bp_4\right)\II_{\left\{r > \ba^\top  \bp_5\right\}}\right),\ \bp_j \in \Omega, j = 1,\dots,5 \bigg\}.
\end{align*}
One can verify that $\|g\|_{\infty} = \|g\|_{L_{\infty}(\cS)} \leq 1$. Again we approximate $\cG$ by its $\delta$-net. More precisely, for any $\delta > 0$, we consider the grid points set $\Omega_0 \subset [0,\frac{\bar r}{\Al}]^{5 m}$, where each element of each point $\bq = (\bq^{(j)})_{j = 1,\dots,5} \in \Omega_0$ has the form $k\delta$, with $k = 0,1,...,\lfloor\frac{\Al}{\delta\bar r}\rfloor + 1$, and $\bq^{(1)}, \bq^{(2)},..., \bq^{(5)} \in \Omega$. Thus, we have $\mathbf{card}(\Omega_0) \leq (2 + \frac{\bar r}{\delta \Al })^{5 m}$. Let $N_0 = (2 + \frac{\bar r}{\delta\Al})^{5 m}$. Let 
\begin{align*}
     \cG = \bigg\{g: g(\ba,r) = \frac{\Al}{4m\Au\bar r}\left( \int_{\ba^\top\bq^{(1)}}^{\ba^\top\bq^{(2)}}\II_{\{r > u\}} {\rm d}u - \left(\ba^\top\bq^{(3)} - \ba^\top\bq^{(4)}\right)\II_{\left\{r > \ba^\top  \bq^{(5)}\right\}}\right),\ \bq^{(j)} \in \Omega_0, j = 1,\dots,5 \bigg\}.
\end{align*}
For a fixed $g \in \cG$ defined by $\bp$, let's denote by $\bq_U$ and $\bq_L$ its upper and lower $\delta$-approximation in $\cG_0$, namely, $\max\{\bq_U^{(j)} - \delta \mathbf{1},\mathbf{0}\} \leq \bp_j < \bq_U^{(j)}$ and $\bq_L^{(j)} < \bp_j\leq \bq_L^{(j)} +\delta \mathbf{1}$ for $j = 1, \dots, 5$ where the maximum is taken elementwisely. Then by definition $\|\bq_U - \bq_L\|_{\infty} \leq 2 \delta.$
Let 
\begin{align*}
    g_U(\ba,r) &= \frac{\Al}{4m\Au\bar r}\left( \int_{\ba^\top\bq_L^{(1)}}^{\ba^\top\bq_U^{(2)}}\II_{\{r > u\}} {\rm d}u - \left(\ba^\top\bq_L^{(3)} - \ba^\top\bq_U^{(4)}\right)\II_{\left\{r > \ba^\top  \bq_U^{(5)}\right\}}\right),\\
    g_L(\ba,r) &= \frac{\Al}{4m\Au\bar r}\left( \int_{\ba^\top\bq_U^{(1)}}^{\ba^\top\bq_L^{(2)}}\II_{\{r > u\}} {\rm d}u - \left(\ba^\top\bq_U^{(3)} - \ba^\top\bq_L^{(4)}\right)\II_{\left\{r > \ba^\top  \bq_L^{(5)}\right\}}\right).
\end{align*}
Then one can check that $g_L(\ba,r)\leq g(\ba,r)\leq g_U(\ba,r)$ for all $(\ba, r) \in \cS$, since $\ba$ is non-negative, and $\bq_L$ and $\bq_U$ satisfy the element-wise bound. We next bound $\|g_U - g_L\|_F$ by
\begin{align*} 
    \|g_U - g_L\|^2_F &= \ \E_{(\ba, r) \sim F} (g_U(\ba,r) - g_L(\ba,r))^2 \nonumber\\
    \leq & \left(\frac{\Al}{4m\Au\bar r}\right)^2 \E_{(\ba, r) \sim F} \bigg(\int_{\ba^\top\bq_L^{(1)}}^{\ba^\top\bq_U^{(2)}}\II_{\{r > u\}} {\rm d}u - \int_{\ba^\top\bq_U^{(1)}}^{\ba^\top\bq_L^{(2)}}\II_{\{r > u\}} {\rm d}u \\
    & \ \ \ \ \ \ \ \ \ \  \ \ \ \ \ \ \ \ \ \ \ \ \ \ \  +\ \left(\ba^\top\bq_U^{(3)} - \ba^\top\bq_L^{(4)}\right)\II_{\left\{r > \ba^\top  \bq_L^{(5)}\right\}} - \left(\ba^\top\bq_L^{(3)} - \ba^\top\bq_U^{(4)}\right)\II_{\left\{r > \ba^\top  \bq_U^{(5)}\right\}}\bigg)^2\\
    & \leq \ \left(\frac{\Al}{4m\Au\bar r}\right)^2 \E_{(\ba, r) \sim F} \bigg( \sum_{j = 1}^4\left(\ba^\top\bq_U^{(j)} - \ba^\top\bq_L^{(j)}\right) \\
    & \ \ \ \ \ \ \ \ \ \  \ \ \ \ \ \ \ \ \ \ \ \ \ \ \  +  \ \max\left(\ba^\top\bq_U^{(3)}, \ba^\top\bq_L^{(4)}\right)\left|\II_{\left\{r > \ba^\top  \bq_L^{(5)}\right\}} - \II_{\left\{r > \ba^\top  \bq_U^{(5)}\right\}}\right| \bigg)^2\\
     & \leq \ \left(\frac{\Al}{4m\Au\bar r}\right)^2 \E_{(\ba, r) \sim F} \bigg( 8 m \Au  \delta +  \frac{m \Au \bar r}{\Al}\left|\II_{\left\{r > \ba^\top  \bq_L^{(5)}\right\}} - \II_{\left\{r > \ba^\top  \bq_U^{(5)}\right\}}\right|\bigg)^2\\
     & \leq \ 2 \left(\frac{\Al}{4m\Au\bar r}\right)^2 \E_{(\ba, r) \sim F} \bigg( 64 m^2 {\Au}^2  {\delta}^2 +  \left(\frac{m \Au \bar r}{\Al}\right)^2\left|\II_{\left\{r > \ba^\top  \bq_L^{(5)}\right\}} - \II_{\left\{r > \ba^\top  \bq_U^{(5)}\right\}}\right|\bigg)\\
     & \leq \ 2\left(\frac{\Al}{4m\Au\bar r}\right)^2 \bigg( 64 m^2 {\Au}^2  {\delta}^2 +  \left(\frac{m \Au \bar r}{\Al}\right)^2\E_{\ba \sim F^{\ba}}\left|F^r_{\ba}(\ba^\top  \bq_L^{(5)}) - F^r_{\ba}(\ba^\top  \bq_U^{(5)}) \right|\bigg)\\
     & \leq \ 2\left(\frac{\Al}{4m\Au\bar r}\right)^2 \bigg( 64 m^2 {\Au}^2  {\delta}^2 +  \left(\frac{m \Au \bar r}{\Al}\right)^2c_{\nu}\E_{\ba \sim F^{\ba}}\left|\ba^\top  \bq_L^{(5)} - \ba^\top  \bq_U^{(5)} \right|^{\nu}\bigg)\\&\ \ \ \ \ \ \ \ \ \ \ \ \ \ \ \ \ \ \ \ \ \ \ \ \ \ \ \ \ \ \ \ \ \ \ \ \ \ \ \ \ \ \ \ \ \ \ \ \ \ \ \ \ \ \ \ \ \ \ \ \ \ \ \ \ \ \ \ \textrm{(under either assumptions)}\\
      & \leq \ 8  \bigg( \left(\frac{\Al \delta}{\bar r} \right)^2 +  c_{\nu}\left( 2 m \Au \delta \right)^{\nu }\bigg) \ \leq \ 8 \left(\frac{2\Al}{\bar r} +c_{\nu}^{\frac{1}{2}} \left( 2 m \Au\right)^{\frac{\nu}{2}}\right)^2\delta^{\min\{2, \nu \}}.
\end{align*}
As a result we have $\|g_U - g_L\|_F \leq 4 \left(\frac{2\Al}{\bar r} + c_{\nu}^{\frac{1}{2}}\left( 2 m \Au\right)^{\frac{\nu}{2}}\right)\delta^{\min\{1, \frac{\nu}{2} \}}$. Now similar to the proof of Lemma \ref{lemma: concentration_gradient}, we have 
\begin{align*}
    \int_0^{\delta} H_B(u, \cG, F)^{1/2} {\rm d} u \ \leq \ C'_3 \delta \sqrt{\log\left(1 + \frac{C'_4}{\delta}\right)}, 
\end{align*}
where $C'_3 = 2\left(\frac{10 m}{\min\{\nu/2, 1\}}\right)^{1/2}$ and $C'_4 = 4 \left(2 + c_{\nu}^{\frac{1}{2}}\frac{\bar r}{\Al}\left( 2 m \Au\right)^{\frac{\nu}{2}}\right)$. 

We notice that both functions
\begin{align*}
    &\int_{\ba^\top\lama}^{\ba^\top\lamb}(\II_{\left\{r > u\right\}} - \II_{\left\{r > \ba^\top  \lamb\right\}}){\rm d}u, \qquad \int_{\ba^\top\lamb}^{\ba^\top\lama}(\II_{\left\{r > u\right\}} - \II_{\left\{r > \ba^\top  \lamb\right\}}){\rm d}u
\end{align*}
are in $\cG$. The rest of the proof is identical to that of Lemma \ref{lemma: concentration_gradient}, and we omit the details.
\end{proof}

\section{Auxiliary Lemmas and Their Proof}\label{appendix:helper and auxiliary}
This section contains all auxiliary lemmas that will be used in constructing our ultimate proof of the {\sf CE} regret analysis. In \Cref{appendix:mklp-helper} we state and prove several results in LP theory. In \Cref{appendix:concentration-helper}, we state and prove classic concentration results from the empirical processes theory which is used to construct our uniform concentration bounds via peeling device. \Cref{appendix: other-helper} contains all other auxiliary lemmas that are used throughout the paper. 
\subsection{Properties of the Multi-Knapsack LP}\label{appendix:mklp-helper}
Recall that the multi-knapsack LP takes the following form
\begin{align}
    \max &\quad  \sum_{j = t + 1}^T r_j x_j \label{program: appendix primal LP} \tag{$\cP$}\\
    \mbox{s.t.} & \quad \sum_{j = t + 1}^T a_{i j} x_j \leq c_{i}, \quad i = 1, \dots, m,\NNN
    \\&\quad x_j \in [0, 1], \quad j = t + 1,\dots, T,\NNN
\end{align}
and
\begin{align}
    \min &\quad \sum_{i = 1}^m c_i \lambda_i + \sum_{j = t + 1}^T \xi_j \label{program: appendix dual LP original} \tag{$\cD$-1}\\
    \mbox{s.t.} & \quad \sum_{i = 1}^m a_{i j} \lambda_i + \xi_j \geq r_{j}, \quad j = t + 1, \dots, T,\NNN
    \\&\quad \lambda_i, \xi_j \geq 0, \quad j = t + 1,\dots, T;\quad i = 1, \dots, m, \NNN
\end{align}

We call Problem (\ref{program: appendix primal LP}) the primal multi-knapsack LP and Problem (\ref{program: appendix dual LP original}) the dual multi-knapsack LP that is specified by coefficients $t, \bc$ and $\cI_t = \{(\ba_j, r_j)\}_{j = t + 1}^T$. Denote by $V^{{\rm off}}_{t, \bc}(\cI_t)$ their optimal value, which by strong duality coincides. \cite{li2022online} observes that the dual problem has the following alternative, simple form.
\begin{lemma}\label{lem: helper-dual LP simplified}
    The dual problem (\ref{program: appendix dual LP original}) has the alternative, simple form
    \begin{align}\label{program: appendix dual LP simplified}
        \min_{\blambda \in \RR^m_{\geq 0}} \bc^{\top} \blambda + \sum_{j = t + 1}^T \left(r_j - \ba^{\top}_j \blambda \right)^+. \tag{$\cD$-2}
    \end{align}
\end{lemma}
\begin{proof}
    Observe that to reach optimality, we should set $\xi_j = 0$ if $\sum_{i = 1}^m a_{i j} \lambda_i \geq r_{j}$ and $\xi_j = r_j - \sum_{i = 1}^m a_{i j} \lambda_i$ otherwise, which is the same as $\xi_j = (r_j - \sum_{i = 1}^m a_{i j} \lambda_i)^+.$ Plugging into the objective function in Problem (\ref{program: appendix dual LP original}) and eliminating the variables $\xi_j$ completes the proof.
\end{proof}
\begin{lemma}[Complementary Slackness]\label{lem: helper-complementary slackness}
    Suppose $\bx^*$ and $\blambda^*$ are optimal solutions to Problem (\ref{program: appendix primal LP}) and Problem (\ref{program: appendix dual LP simplified}), respectively. Then they satisfy the complementary slackness property 
    \begin{align*}
        r_j >  \ba^\top_j \blambda\quad &\Rightarrow \quad x^*_j = 1, \\
        r_j <  \ba^\top_j \blambda \quad &\Rightarrow \quad x^*_j = 0, \mbox{ for all $j = t+1, \dots, T.$}
    \end{align*}
\end{lemma}
\begin{proof}
    Suppose $(\blambda^*, \bxi^*)$ is an optimal solution to Problem \eqref{program: appendix dual LP original}. By LP complementary slackness, we have for $i = 1, \dots, m$ and $j = t+1, \dots, T$,
    \begin{align*}
        & (1 - x^*_j)\xi^*_j = 0,\\
        &  \left(r_j - \sum_{i = 1}^m a_{i j} \lambda^*_i - \xi^*_j\right)x^*_j = 0,\\
        & \left(c_i - \sum_{j = t + 1}^T a_{i j} x_j \right)\lambda^*_i = 0.
    \end{align*}
If for some $j$ we have $r_j >  \ba^\top_j \blambda$, then by the dual constraint of Problem (\ref{program: appendix dual LP original}), we have $\xi^*_j \geq r_j -  \ba^\top_j \blambda > 0.$ Then by the first complementary slackness equality, we must have $1 - x^*_j = 0$, completing the proof of the first statement. 

Now if on the other hand, if we have $r_j <  \ba^\top_j \blambda$, then since $\xi^*_j \geq 0$, we have $r_j - \sum_{i = 1}^m a_{i j} \lambda^*_i - \xi^*_j \leq r_j - \sum_{i = 1}^m a_{i j} \lambda^*_i  < 0$. By the second complementary slackness equality, we must have $x^*_j = 0$, again completing the proof of the second statement. 
\end{proof}
\begin{lemma}\label{lem: helper-non-integer}
Suppose $\bx^*$ is an optimal basic feasible solution to Problem (\ref{program: appendix primal LP}), then $\bx^*$ has at most $m$ basic variables that are fractional.
\end{lemma}
\begin{proof}
    Refer to \cite{meanti1990probabilistic} Lemma 2.1.
\end{proof}
\begin{lemma}\label{lem: helper-offline-opt-induction}
Suppose $\{x^*_j\}_{j = t}^T$ is an optimal solution to the primal multi-knapsack LP specified by $t - 1, \bb$ and $\cI_{t-1}$, then
\[
V^{\rm off}_{t-1, \bb}(\cI_{t-1}) = r_t x^*_t +  V^{\rm off}_{t, \bb - \ba_t x^*_t}(\cI_{t}) = \max_{x \in [0, 1]} \left(r_t x +  V^{\rm off}_{t, \bb - \ba_t x}(\cI_{t})\right).
\]
\end{lemma}

\begin{proof}
 By definition,   $V^{\rm off}_{t-1, \bb}(\cI_{t-1}) = \sum_{j = t}^T r_j x^*_j = r_t x^*_t + \sum_{j = t+ 1}^T r_j x^*_j.$ On the one hand, $\{x^*_j\}_{j = t}^T$ is feasible, concretely, $\sum_{j = t}^T a_{i j} x^*_j \leq c_i, i = 1, \dots, m$, and $x^*_j \in [0, 1], j = t \dots, T.$ From the above we derive  $\sum_{j = t+1}^T a_{i j} x_j^* \leq c_i - a_{i t} x^*_t, i = 1, \dots, m$ and $x^*_j \in [0, 1], j = t + 1 \dots, T.$ In other words, $\{x^*_j\}_{j = t + 1}^T$ is feasible to the LP specified by $t - 1, \bb - \ba_t x^*_t$ and $\cI_{t}$. Thus $\sum_{j = t+ 1}^T r_j x^*_j \leq V^{\rm off}_{t, \bb -\ba_t x^*_t}$ and we conclude that 
 \[
 V^{\rm off}_{t-1, \bb}(\cI_{t-1})  \leq  r_t x^*_t +  V^{\rm off}_{t, \bb - \ba_t x^*_t}(\cI_{t}) \leq \max_{x \in [0, 1]} \left(r_t x +  V^{\rm off}_{t, \bb - \ba_t x}(\cI_{t})\right).
 \]
 On the other hand, consider an arbitrary optimal solution $\{\tilde x^*_j\}_{j = t +1}^T$ to the multi-knapsack LP specified by $t, \bb - \ba_t \tilde x^*_t$ and $\cI_{t}$, where $\tilde x^*_t = \argmax_{x \in [0, 1]} \left(r_t x +  V^{\rm off}_{t, \bb - \ba_t x}(\cI_{t})\right)$. We have $V^{\rm off}_{t, \bb - \ba_t \tilde x^*_t}(\cI_{t}) = \sum_{j = t + 1}^T r_j \tilde x^*_j$. Notice that $(\tilde x^*_t, \tilde x^*_{t + 1}, \dots, \tilde x^*_{T})$ is a feasible solution to the original LP specified by $t -1, \bb$ and $\cI_{t -1}$. We thus have
 \[
 V^{\rm off}_{t-1, \bb}(\cI_{t-1})  \geq  r_t \tilde x^*_t +  V^{\rm off}_{t, \bb - \ba_t \tilde x^*_t}(\cI_{t}) = \max_{x \in [0, 1]} \left(r_t x +  V^{\rm off}_{t, \bb - \ba_t x}(\cI_{t})\right).
 \]
 Combining the inequalities completes the proof.
\end{proof}

\subsection{Concentration and Empirical Processes}\label{appendix:concentration-helper}
We borrow techniques from the empirical processes theory to establish concentration results that will in turn be used in proving our main theorems. 

\begin{lemma}\label{lem:helper-concentration-geer}
    Suppose that $X$ is a random vector following distribution $P$ supported on $\cX$. $\cG$ is a function class such that $\sup_{g\in \cG}\|g\|_{L_\infty(\cX)}\leq K$, $\sup_{g\in \cG}\|g\|^2_{P}\leq R^2$. Take $a,C_0,C_1$ satisfying
    \begin{align}
        a\leq & C_1\sqrt{n}R^2/(2K),\label{eq: Thm511eq1}\\
        a\leq & 8\sqrt{2n}R,\label{eq: Thm511eq2}\\
        a\geq & C_0\max\left\{\int_0^{\sqrt{2}R} H_B(u/\sqrt{2},\cG,P)^{1/2}du,\sqrt{2}R\right\},\label{eq: Thm511eq3}\\
        C_0^2 \geq & C^2(C_1+1),\label{eq: Thm511eq4}
    \end{align}
    where $C$ is a universal constant. Then 
    \begin{align*}
        \P\left(\sup_{g\in \cG}\sqrt{n}\left|\frac{1}{n}\sum_{j=1}^n g(X_i) - \E_{X\sim P}g(X)\right|\geq a\right)\leq C\exp\left(-\frac{a^2}{2C^2(C_1+1)R^2}\right).
    \end{align*}
\end{lemma}
\begin{proof}[Proof of Lemma \ref{lem:helper-concentration-geer}] It suffices to show that the conditions of Theorem 5.11 of \cite{geer2000empirical} are satisfied. Following the notation in \cite{geer2000empirical}, we use 
\begin{align*}
    \rho_K(g)^2 = 2K^2\int \left(e^{|g|/K} - 1 - |g|/K\right){\rm d}P,
\end{align*}
and let $\mathscr{H}_{B,K}(\delta,\cG,P)$ be the generalized entropy with bracketing (see Definition 5.1 of \cite{geer2000empirical}). Lemmas 5.8 and 5.10 of \cite{geer2000empirical} imply that if $\sup_{g\in \cG}\|g\|_{L_\infty}\leq K_1/4<K_1/2$, then $\rho_{K_1}(g)^2\leq 2R^2 = R_1^2$ with $R_1 = \sqrt{2}R$, and $\mathscr{H}_{B,K_1}(\delta,\cG,P)\leq H_B(\delta/\sqrt{2},\cG,P)$ for all $\delta>0$. Let $K_1=4K$. Therefore, conditions of Theorem 5.11 of \cite{geer2000empirical} becomes
\begin{align*}
        a\leq & C_1\sqrt{n}R_1^2/K_1 = C_1\sqrt{n}R^2/(2K),\\
        a\leq & 8\sqrt{n}R_1 = 8\sqrt{2n}R,\\
        a\geq & C_0\max\left\{\int_{a/(2^6\sqrt{n})}^{R_1} \mathscr{H}_{B,K_1}(u,\cG,P)^{1/2}du,R_1\right\},\\
        C_0^2 \geq & C^2(C_1+1),
\end{align*}
where the third inequality is implied by
\begin{align*}
    a\geq &C_0\max\left\{\int_0^{\sqrt{2}R} H_B(u/\sqrt{2},\cG,P)^{1/2}du,\sqrt{2}R\right\}
    \geq C_0\max\left\{\int_{a/(2^6\sqrt{n})}^{\sqrt{2}R} H_B(u/\sqrt{2},\cG,P)^{1/2}du,\sqrt{2}R\right\}.
\end{align*}
The results of Theorem 5.11 of \cite{geer2000empirical} becomes
\begin{align*}
        \P\left(\sup_{g\in \cG}\sqrt{n}\left|\frac{1}{n}\sum_{j=1}^n g(X_i) - \E_{X\sim P}g(X)\right|\geq a\right)\leq & C\exp\left(-\frac{a^2}{C^2(C_1+1)R_1^2}\right),
    \end{align*}
which equals to $C\exp\left(-\frac{a^2}{2C^2(C_1+1)R^2}\right)$. This finishes the proof.
\end{proof}




\subsection{Other Auxiliary Lemmas}\label{appendix: other-helper}
This section contains the proof of all other results appeared earlier in the paper.
\begin{lemma}\label{lem: helper-line-integration}
For any $\blambda_1, \blambda_2 \in \RR^m$ and any $\ba \in \RR^m, r \in \RR, \bb \in \RR^m$, it holds true that
\begin{align*}
    h_{t,\bb}(\blambda_1, \ba, r) - h_{t,\bb}(\blambda_2, \ba, r) = & \phi_{t,\bb}(\blambda_2, \ba, r)^\top(\blambda_1-\blambda_2) + \int_{\ba^\top\blambda_1}^{\ba^\top\blambda_2}\left(\II_{\left\{r > v\right\}} - \II_{\left\{r > \ba^\top \blambda_2\right\}}\right){\rm d}v,
\end{align*}
\end{lemma}
\begin{proof}
Refer to the proof of Lemma 1 in \cite{li2022online}. 
\end{proof}

\begin{lemma}\label{lem: helper-cauchy-schwarz}
    For any $\blambda_1, \blambda_2 \in \RR^m_{\geq 0}$, 
    \[
    \E_{(\ba,r)\sim F}\left(\int_{\ba^\top\blambda_1}^{\ba^\top\blambda_2}\left(\II_{\left\{ r > v \right\}} - \II_{\left\{ r > \ba^\top \blambda_2 \right\}} \right){\rm d}v\right)^2  \leq \E_{\ba \sim F^{\ba}}\left(\left( \ba^{\top} (\blambda_1 - \blambda_2)\right)^2 \cdot \left|F^r_{\ba}(\ba^\top\blambda_1) - F^r_{\ba}(\ba^\top\blambda_2)  \right|\right).
    \]
\end{lemma}
\begin{proof}
    The Cauchy-Schwarz inequality implies that
    \begin{align*}
        &\E_{(\ba,r)\sim F}\left(\int_{\ba^\top \blambda_1}^{\ba^\top\blambda_2}\left(\II_{\left\{ r > v \right\}} - \II_{\left\{ r > \ba^\top \blambda_2 \right\}} \right){\rm d}v\right)^2\\  
        &\quad\quad\quad\quad\quad\quad \leq \ \ \  \E_{(\ba,r)\sim F}\left(\left| \int_{\ba^{\top} \blambda_1}^{\ba^{\top} \blambda_2} 1 {\rm d} v\right| \cdot \left| \int_{\ba^\top \blambda_1}^{\ba^\top\blambda_2}\left(\II_{\left\{ r > v \right\}} - \II_{\left\{ r > \ba^\top \blambda_2 \right\}} \right)^2{\rm d}v \right|\right)\\
        &\quad\quad\quad\quad\quad\quad = \ \ \  \E_{(\ba,r)\sim F}\left(\left| \ba^{\top} (\blambda_1 - \blambda_2)\right| \cdot \left| \int_{\ba^\top \blambda_1}^{\ba^\top\blambda_2}\left(\II_{\left\{ r > v \right\}} - \II_{\left\{ r > \ba^\top \blambda_2 \right\}} \right)^2{\rm d}v \right|\right)\\
        &\quad\quad\quad\quad\quad\quad \leq \ \ \ \E_{(\ba,r)\sim F}\left(\left| \ba^{\top} (\blambda_1 - \blambda_2)\right| \cdot \left| \int_{\ba^\top \blambda_1}^{\ba^\top\blambda_2}\left(\II_{\left\{ r > \ba^\top\blambda_1 \right\}} - \II_{\left\{ r > \ba^\top \blambda_2 \right\}} \right){\rm d}v \right|\right)\\
        &\quad\quad\quad\quad\quad\quad = \ \ \ \E_{(\ba,r)\sim F}\left(\left( \ba^{\top} (\blambda_1 - \blambda_2)\right)^2 \cdot \left| \II_{\left\{ r > \ba^\top\blambda_1 \right\}} - \II_{\left\{ r > \ba^\top \blambda_2 \right\}} \right|\right) \\
        &\quad\quad\quad\quad\quad\quad = \ \ \ \E_{\ba\sim F^{\ba}}\left(\left( \ba^{\top} (\blambda_1 - \blambda_2)\right)^2 \cdot \left|F^r_{\ba}(\ba^\top\blambda_1) - F^r_{\ba}(\ba^\top\blambda_2)  \right|\right). 
    \end{align*}
\end{proof}

\begin{lemma}\label{lem: helper-boundedness-lambda}
    It holds true that $\ft, \dt, \dtt \in \Omega$ for any $\bb_{t - 1}$, any $\cI_{t - 1}$ and $t = 1, \dots, T.$
\end{lemma}
\begin{proof}
    We only prove the Lemma for $\ft$. The argument for $\dt$ and $\dtt$ is similar. For any $\blambda = (\lambda_1, \dots, \lambda_m)^{\top}\notin \Omega$ but is feasible ($\blambda \in \RR^m_{\geq 0}$), there exists $i$ such that $\lambda_i > \frac{\bar r}{\underline A}$. Let $\blambda' \triangleq (\lambda_1, \dots, \lambda_{i - 1}, \frac{\bar r}{\underline A}, \lambda_{i + 1}, \dots, \lambda_m)^{\top}.$ Since $\ba^{\top} \blambda > \ba^{\top} \blambda' \geq a_i \lambda_i > {\bar r}$, it holds path-wisely that $(r - \ba^{\top} \blambda)^+ = (r - \ba^{\top} \blambda')^+ = 0.$ However, $\bb_{t - 1}^{\top} \blambda > \bb_{t - 1}^{\top} \blambda'$ as all elements are non-negative, which leads to $f_{t, \bb_{t - 1}}(\blambda) > f_{t, \bb_{t - 1}}(\blambda').$ In other words, $\blambda$ cannot be the optimal solution. Thus $\ft$ must be in $\Omega$ and we have proved the Lemma.
\end{proof}
\begin{coro}\label{coro: helper-bounded-norm-lambda}
The optimal solutions $\ft, \dt, \dtt$ have uniformly bounded norm for any $\bb_{t - 1}$, any $\cI_{t - 1}$ and $t = 1, \dots, T$,
\begin{align*}
    0 \leq &\|\ft\|_{\infty}, \|\dt\|_{\infty}, \|\dtt\|_{\infty} \leq \frac{\bar r}{\underline A},\\
    0 \leq  &\|\ft\|, \|\dt\|, \|\dtt\| \leq \sqrt{m}\frac{\bar r}{\underline A}.
\end{align*}
\end{coro}
\begin{lemma}\label{lem: helper-gradient-non-negativity}
 For any $\blambda \in \Omega$, any $\bb_{t- 1} \in \RR^m_{\geq 0}$ and any $t = 1, \dots, T$, 
 \begin{align*}
     &\ \E_{(\ba, r) \sim F}\phi_{t,\bb_{t - 1}}(\ft, \ba, r)^\top(\blambda -\ft)\geq 0,\\
     & \sum_{j=t+1}^T \phi_{t,\bb_{t - 1}}(\dt, \ba_j, r_j)^\top(\blambda -\dt) \geq - \frac{2 m^2 {\bar A} {\bar r}}{\underline A}, \quad a.s.
     \\
     & \sum_{j=t+1}^T \phi_{t,\bb_{t - 1} - \ba_t}(\dtt, \ba_j, r_j)^\top(\blambda -\dtt) \geq - \frac{2 m^2 {\bar A} {\bar r}}{\underline A}, \quad a.s.  \quad\quad\quad\quad \textrm{(s.t. feasibility constraint)}
 \end{align*}
\end{lemma}
\begin{proof}
    By Lemma \ref{lem: helper-line-integration}, we have 
    \[
    f_{t, \bb}(\blambda) - f_{t, \bb}(\ft) = \E_{(\ba, r)\sim F}\phi_{t,\bb}(\ft, \ba, r)^\top(\blambda -\ft) + \E_{(\ba, r)\sim F}\int_{\ba^\top\blambda}^{\ba^\top\ft}\left(\II_{\left\{r > v\right\}} - \II_{\left\{r > \ba^\top \ft \right\}}\right){\rm d}v.
    \]
    We argue by contradiction. Suppose there exists $\blambda \in \RR^m_{\geq 0}$ such that $\E_{(\ba, r)\sim F}\phi_{t,\bb_{t-1}}(\ft, \ba, r)^\top(\blambda -\ft) = - \eta < 0.$ Consider the following linear parameterization $\blambda(s) = \ft + s(\blambda - \ft)$ for $s \in [0, 1]$ and correspondingly
    \begin{align*}
        &f_{t, \bb_{t-1}}(\blambda(s)) - f_{t, \bb_{t-1}}(\ft)\\
        &\quad\quad\quad\quad= \E_{(\ba, r)\sim F}\phi_{t,\bb_{t-1}}(\ft, \ba, r)^\top(\blambda(s) -\ft) + \E_{(\ba, r)\sim F}\int_{\ba^\top\blambda(s)}^{\ba^\top\ft}\left(\II_{\left\{r > v\right\}} - \II_{\left\{r > \ba^\top \ft \right\}}\right){\rm d}v\\
        &\quad\quad\quad\quad =  \E_{(\ba, r)\sim F}\phi_{t,\bb_{t-1}}(\ft, \ba, r)^\top(\blambda(s) -\ft)  + \E_{\ba \sim F^{\ba}} \int_{\ba^\top\blambda(s)}^{\ba^\top\ft}\left(F^r_{\ba}(\ba^\top \ft) - F^r_{\ba}(v)\right){\rm d}v\\
        &\quad\quad\quad\quad \leq  \E_{(\ba, r)\sim F}\phi_{t,\bb_{t-1}}(\ft, \ba, r)^\top(\blambda(s) -\ft)  + \E_{\ba \sim F^{\ba}} \int_{\ba^\top\blambda(s)}^{\ba^\top\ft}\left(F^r_{\ba}(\ba^\top \ft) - F^r_{\ba}(\ba^\top\blambda(s))\right){\rm d}v\\
        &\quad\quad\quad\quad =  \E_{(\ba, r)\sim F}\phi_{t,\bb_{t-1}}(\ft, \ba, r)^\top(\blambda(s) -\ft)  + \E_{\ba \sim F^{\ba}}\left\{ \left(\ba^\top\ft - \ba^\top\blambda(s)\right)\left(F^r_{\ba}(\ba^\top \ft) - F^r_{\ba}(\ba^\top\blambda(s))\right)\right\}\\
        &\quad\quad\quad\quad =  s\E_{(\ba, r)\sim F}\phi_{t,\bb_{t-1}}(\ft, \ba, r)^\top(\blambda -\ft)  + s \E_{\ba \sim F^{\ba}} \left\{\left(\ba^\top\ft - \ba^\top\blambda\right)\left(F^r_{\ba}(\ba^\top \ft) - F^r_{\ba}(\ba^\top\blambda(s))\right)\right\}\\
        &\quad\quad\quad\quad =  - s \eta  + s \E_{\ba \sim F^{\ba}}\left\{\left(\ba^\top\ft - \ba^\top\blambda\right)\left(F^r_{\ba}(\ba^\top \ft) - F^r_{\ba}(\ba^\top\blambda(s))\right)\right\}\\
        &\quad\quad\quad\quad \leq  - s \eta  + s \frac{2 m {\bar A}{\bar r} }{\underline{A}} \E_{\ba \sim F^{\ba}}\left[\left|F^r_{\ba}(\ba^\top \ft) - F^r_{\ba}(\ba^\top\blambda(s))\right|\right],
    \end{align*}
    where 
    in the last inequality we use the fact that both $\ft$ and $\blambda$ belong to $\Omega$ and that $|\ba^\top\ft - \ba^\top\blambda|$ is upper bounded by absolute constants (cf. Lemma \ref{lem: helper-boundedness-lambda} and Corollary \ref{coro: helper-bounded-norm-lambda}). By the continuity of $F^r_{\ba}$, we have 
    \[
    \lim_{s \to 0} \E_{\ba \sim F^{\ba}}\left[\left|F^r_{\ba}(\ba^\top \ft) - F^r_{\ba}(\ba^\top\blambda(s))\right|\right] = 0.
    \]
    Consequently, there must exist $s > 0$, such that $\E_{\ba \sim F^{\ba}}\left[\left|F^r_{\ba}(\ba^\top \ft) - F^r_{\ba}(\ba^\top\blambda(s))\right|\right] < \frac{\eta {\underline A}}{4 m {\bar A} {\bar r}}.$ Plugging back and  we conclude that $f_{t, \bb_{t - 1}}(\blambda(s)) - f_{t, \bb_{t - 1}}(\ft) \leq -\frac{s \eta}{2} < 0,$ contracting the fact that $\ft$ is the minimizer of $f_{t, \bb_{t - 1}}(\cdot)$ among all $\blambda \in \RR^m_{\geq 0}$. Therefore it must be that $\E_{(\ba, r)\sim F}\phi_{t,\bb_{t- 1}}(\ft, \ba, r)^\top(\blambda -\ft)\geq 0$ and we conclude the proof of the first case. The proof of the second and third case of the Lemma is in spirit very similar. We again argue by contradiction. Suppose there exists $\blambda \in \RR^m_{\geq 0}$ such that $\sum_{j = t + 1}^T \phi_{t,\bb_{t-1}}(\dt, \ba_j, r_j)^\top(\blambda -\dt) = - \frac{2 m^2 {\bar A} {\bar r}}{\underline A} - \eta < - \frac{2 m^2 {\bar A} {\bar r}}{\underline A}.$ Consider the linear parameterization $\blambda(s) = \dt + s(\blambda - \dt)$ for $s \in [0, 1]$ and correspondingly 
    \begin{align*}
        &g_{t, \bb_{t-1}}(\blambda(s)) - g_{t, \bb_{t-1}}(\dt)\\
        &\quad\quad\quad\quad= \frac{1}{T - t}\sum_{j = t+ 1}^T\phi_{t,\bb_{t-1}}(\dt, \ba_j, r_j)^\top(\blambda(s) -\dt) + \frac{1}{T - t}\sum_{j = t+ 1}^T\int_{\ba^\top_j \blambda(s)}^{\ba^\top_j \dt}\left(\II_{\left\{r_j > v\right\}} - \II_{\left\{r_j > \ba^\top_j \dt \right\}}{\rm d}v\right),\\
        &\quad\quad\quad\quad = \frac{1}{T - t}\left(- s \eta - s \frac{2 m^2 {\bar A} {\bar r}}{\underline A} + \sum_{j = t+ 1}^T\int_{\ba^\top_j \blambda(s)}^{\ba^\top_j \dt}\left(\II_{\left\{r_j > v\right\}} - \II_{\left\{r_j > \ba^\top_j \dt \right\}}\right){\rm d}v\right),\\
        &\quad\quad\quad\quad \leq \frac{1}{T - t}\left(- s \eta - s \frac{2 m^2 {\bar A} {\bar r}}{\underline A} + \sum_{j = t+ 1}^T\int_{\ba^\top_j \blambda(s)}^{\ba^\top_j \dt}\left(\II_{\left\{r_j > \ba^\top_j \blambda(s) \right\}} - \II_{\left\{r_j > \ba^\top_j \dt \right\}}\right){\rm d}v\right),\\
        &\quad\quad\quad\quad = \frac{1}{T - t}\left(- s \eta - s \frac{2 m^2 {\bar A} {\bar r}}{\underline A} +  \sum_{j = t+ 1}^T(\ba^\top_j \dt - \ba^\top_j \blambda(s))\left(\II_{\left\{r_j > \ba^\top_j \blambda(s) \right\}} - \II_{\left\{r_j > \ba^\top_j \dt \right\}}\right)\right),\\
        &\quad\quad\quad\quad = \frac{1}{T - t}\left(- s \eta - s \frac{2 m^2 {\bar A} {\bar r}}{\underline A} +  s \sum_{j = t+ 1}^T(\ba^\top_j \dt - \ba^\top_j \blambda)\left(\II_{\left\{r_j > \ba^\top_j \blambda(s) \right\}} - \II_{\left\{r_j > \ba^\top_j \dt \right\}}\right)\right),\\
        &\quad\quad\quad\quad \leq \frac{1}{T - t}\left(- s \eta - s \frac{2 m^2 {\bar A} {\bar r}}{\underline A} +  s \frac{2 m {\bar A}{\bar r} }{\underline{A}} \sum_{j = t+ 1}^T\left|\II_{\left\{r_j > \ba^\top_j \blambda(s) \right\}} - \II_{\left\{r_j > \ba^\top_j \dt \right\}}\right|\right).
    \end{align*}
    Consider the index set $W \triangleq \{j \in [t + 1, T]: r_j = \ba^{\top}_j \dt \}.$ Since $F^r_{\ba}$ is a continuous distribution for any $\ba$, where $\ba \in \RR^m_{\geq 0}$ is a $d$-dimensional vector, furthermore $\{(\ba_j, r_j)\}_{j = t+ 1}^T$ are independent of each other, it follows that $\P(|W| > m) = 0$ (Otherwise we have $m + 1$ \emph{i.i.d.} drawn vectors $(\ba_j, r_j)$ belonging to the same $m$-dimensional hyperplane $y = \ba^{\top} \dt$). For fixed set of $\{(\ba_j, r_j)\}_{j = t+ 1}^T$ and $\dt$, there exists a small enough $s>0$, such that 
    \[
    \II_{\left\{r_j > \ba^\top_j \blambda(s) \right\}} - \II_{\left\{r_j > \ba^\top_j \dt \right\}} \neq 0 \quad \Leftrightarrow \quad i \in W.
    \]
    Hence almost surely we have 
    \[
        \sum_{j = t+ 1}^T\left|\II_{\left\{r_j > \ba^\top_j \blambda(s) \right\}} - \II_{\left\{r_j > \ba^\top_j \dt \right\}}\right| \leq |W| \leq m.
    \]
    Plugging back and we conclude that $g_{t, \bb_{t-1}}(\blambda(s)) - g_{t, \bb_{t-1}}(\dt) \leq - \frac{s \eta}{T-t} < 0$, contradicting the fact that $\dt$ is the minimizer of $g_{t, \bb_{t - 1}}.$ Therefore, it must be that $\sum_{j=t+1}^T \phi_{t,\bb_{t - 1}}(\dt, \ba_j, r_j)^\top(\ft -\dt) \geq - \frac{2 m^2 {\bar A} {\bar r}}{\underline A}, \quad a.s.$ and we conclude the second case of the proof. The third case is nearly identical to the second half, and we omit its proof.
\end{proof}

\begin{lem}
    \label{lem: ball-beats-orthogonality}
    Suppose Assumption \ref{assum: small-probability-starting-from-zero} holds. Then there exists $w > 0$ and $\theta \in (0, 2\pi)$, such that for any $\ba \in {\rm supp}(F^{\ba})$ and any 
    $\blambda_1, \blambda_2$,
    \begin{align*}
         \E_{\ba' \sim F^{\ba}, \ba' \in \cB(\ba, r)} (\ba'^\top \blambda_1 - \ba'^{\top} \blambda_2)^2 \geq l_f\frac{\pi^{m/2}}{\Gamma\left(\frac{m}{2}+1\right)}\left(\frac{r\wedge w}{6} \sin{(\theta)}\right)^{m + 2}\|\blambda_1 - \blambda_2\|^2,
    \end{align*}
    where $\cB(\ba, r)$ is the ball centered at $\ba$ with radius $r > 0$. 
    
\end{lem}

\begin{proof}
By Assumption \ref{assum: small-probability-starting-from-zero}, since $\textrm{supp}(F^{\ba})$ is convex, $\textrm{supp}(F^{\ba})$ is a Lipschitz domain, thus satisfies the \textit{uniform cone condition}.
That is, there exist an opening angle $\theta \in (0, \pi)$, and a radius $w > 0$, such that for every point $\ba \in \textrm{supp}(F^{\ba})$, there exists a cone $C_{\ba}$ with direction represented by a unit vector $\bv_{\ba} \in \mathbb{R}^m$ and vertex at $\ba$, defined by
\[
C_{\ba} = \left\{ \ba + s \by \,:\, s \in [0, w], \, \by \in \mathbb{S}^{m-1}, \, \langle \by, \bv_{\ba} \rangle \geq \cos(\theta) \right\},
\]
such that $C_{\ba} \subset \textrm{supp}(F^{\ba})$, where $\mathbb{S}^{m-1}$ is the unit sphere in $\mathbb{R}^m.$ 

Furthermore, the probability density associated to $F^{\ba}$ is bounded from below by a constant $l_f > 0.$ 

Observe that ball $\mathcal{B}\left(\ba + \frac{r \wedge w}{2}\bv_{\ba}, \frac{r \wedge w}{2} \sin{(\theta)}\right)$ is contained in $\cB(\ba, r) \cap C_\ba.$ In what follows, we shall refer to this ball by $\cB$ for notational simplicity, and denote by $\bq^* \triangleq \ba + \frac{r \wedge w}{2}\bv_{\ba}$ and $r^* \triangleq \frac{r \wedge w}{2} \sin{(\theta)}.$ Without loss of generality, suppose ${\bq^*}^\top\left(\blambda_1 - \blambda_2\right) \geq 0$, otherwise we can consider ${\bq^*}^\top\left(\blambda_2 - \blambda_1\right)$ instead. Let  $\bq = \bq^* + \frac{2}{3} r^* \frac{(\blambda_1 - \blambda_2)}{\|\blambda_1 - \blambda_2\|}$, and thus
\begin{align*}
   \bq^\top\left(\blambda_1 - \blambda_2\right) = {\bq^*}^\top\left(\blambda_1 - \blambda_2\right) +  \frac{2}{3}r^* \frac{(\blambda_1 - \blambda_2)^\top}{\|\blambda_1 - \blambda_2\|}(\blambda_1 - \blambda_2) \geq \frac{2}{3}r^* \|\blambda_1 - \blambda_2\|.
\end{align*}
Clearly, $\cB(\bq, \frac{r^*}{3}) \subset \cB$. For any $\bq' \in \cB(\bq, \frac{r^*}{3})$,
\begin{align*}
    \bq'^{\top}(\blambda_1 - \blambda_2) &= \bq^{\top}(\blambda_1 - \blambda_2) - (\bq - \bq')^{\top}(\blambda_1 - \blambda_2)\\
    &\geq \frac{2}{3}r^* \|\blambda_1 - \blambda_2\| - \|\bq - \bq'\|\cdot \|\blambda_1 - \blambda_2\|\\
    &\geq \frac{2}{3}r^* \|\blambda_1 - \blambda_2\| - \frac{r^*}{3}\|\blambda_1 - \blambda_2\| = \frac{r^*}{3}\|\blambda_1 - \blambda_2\|.
\end{align*}
Therefore, we have 
\begin{align*}
    \E_{\ba' \sim F^{\ba}, \ba' \in \cB} (\ba'^\top \blambda_1 - \ba'^{\top} \blambda_2)^2 &\geq \E_{\ba' \sim F^{\ba},  \ba' \in \cB(\bq, \frac{r^*}{3})} (\ba'^\top \blambda_1 - \ba'^{\top} \blambda_2)^2, \\
    & \geq \mathbb{P}\left(\ba' \in \cB(\bq, \frac{r^*}{3})\right) \frac{{r^*}^2}{9}\|\blambda_1 - \blambda_2\|^2,\\
    & \geq l_f\frac{\pi^{m/2}}{\Gamma\left(\frac{m}{2}+1\right)}\frac{{r^*}^m}{3^m}\frac{{r^*}^2}{9}\|\blambda_1 - \blambda_2\|^2,
\end{align*}
where the last inequality follows from the fact that $F^\ba$ has lower bounded density by Assumption \ref{assum: small-probability-starting-from-zero}, the volume of the ball $\cB(\bq, \frac{r^*}{3})$ is $\frac{\pi^{m/2}}{\Gamma\left(\frac{m}{2}+1\right)}\frac{{r^*}^m}{3^m}$, and that $\cB \subseteq \cB(\ba, r) \cap C_\ba \subseteq \textrm{supp}(F^\ba).$ The lemma thus follows from the fact that $\cB \subseteq \cB(\ba, r) \cap C_\ba \subseteq  \cB(\ba, r).$   
\end{proof}


\begin{lemma}\label{lem: helper-continuity-under-assumption-2-3}
    Suppose $\ba \in \textrm{supp}(F^{\ba})$ and $\dt$ satisfy 
    \begin{align*}
        \frac{l_{\ba} + r_{\ba}}{2} - \frac{1}{8}c_\nu^{\frac{1}{\nu}} \leq \ba^{\top} \dt \leq \frac{l_{\ba} + r_{\ba}}{2} + \frac{1}{8}c_\nu^{\frac{1}{\nu}}.
    \end{align*}
    Then for all $\ba' \in \cB(\ba, r_1)$, where $\cB(\ba, r_1)$ is a ball centered at $\ba$ with radius $r_1 = \frac{1}{8}\left(\left(c_L + \sqrt{m}\frac{\bar r}{\Al}\right)^{-1}\wedge 1\right) c_\nu^{\frac{1}{\nu}}$, it holds that
 \begin{align*}
     \frac{l_{\ba'} + r_{\ba'}}{2} - \frac{1}{4}c_\nu^{\frac{1}{\nu}}\leq \ba'^{\top} \dt \leq \frac{l_{\ba'} + r_{\ba'}}{2} + \frac{1}{4}c_\nu^{\frac{1}{\nu}}.
 \end{align*}
\end{lemma}

\begin{proof}
    We only prove one direction of this inequality as the other direction is nearly identical. For any $\ba' \in \cB(\ba, r_1)$, we have that 
    \begin{align*}
        \ba'^{\top} \dt & = \ba^{\top} \dt + \left(\ba - \ba'\right)^{\top} \dt\\
        &\leq  \ba^{\top} \dt + \|\ba - \ba'\|\|\dt\| \qquad (\textrm{by Cauchy-Schwarz}),\\
        & \leq \frac{l_{\ba} + r_{\ba}}{2} + \frac{1}{8}c_\nu^{\frac{1}{\nu}} + \ba^{\top} \dt + \|\ba - \ba'\|\|\dt\|\\
        &\leq \frac{l_{\ba'} + r_{\ba'}}{2} + c_L\|\ba' - \ba\| + \frac{1}{8}c_\nu^{\frac{1}{\nu}} +   \|\ba - \ba'\|\|\dt\| \qquad\qquad\qquad (\textrm{by Lipschitzness of $l_\ba$ and $r_\ba$})\\
        & \leq \frac{l_{\ba'} + r_{\ba'}}{2} + \frac{1}{8}c_\nu^{\frac{1}{\nu}} +  \frac{1}{8}\left(c_L + \sqrt{m}\frac{\bar r}{\Al}\right)^{-1} c_\nu^{\frac{1}{\nu}} \times \left(c_L + \sqrt{m}\frac{\bar r}{\Al}\right) \\
        &\qquad\qquad\qquad\qquad\qquad\qquad (\textrm{by definition of $r_1$ and Lemma \ref{lem: helper-boundedness-lambda}})\\
        & = \frac{l_{\ba'} + r_{\ba'}}{2} + \frac{1}{4}c_\nu^{\frac{1}{\nu}},
    \end{align*}
    completing the proof.
\end{proof}

\end{document}